\documentclass[a4paper,10pt]{amsart}

\pdfoutput=1
\usepackage{amsmath}
\usepackage{amsthm}
\usepackage{amsfonts}
\usepackage{amssymb}
\usepackage[all]{xy}
\xyoption{line}
\usepackage{hyperref}
\usepackage{enumitem}

\usepackage{graphicx}
\newcommand{\ZZ}{\mathbb{Z}}
\def\CC{\mathbb{C}}

\newcommand{\QQ}{\mathcal{Q}}

\newcommand{\A}{\mathcal{A}}

\newcommand{\ee}{\varepsilon}
\newcommand{\im}{\textrm{Im}}
\newcommand{\Ker}{\textrm{Ker}}
\DeclareMathOperator{\Rep}{Rep}
\newcommand{\degg}{\leq_{\deg}}
\DeclareMathOperator{\GL}{GL}
\DeclareMathOperator{\G}{G}
\DeclareMathOperator{\OO}{O}
\DeclareMathOperator{\SP}{SP}
\DeclareMathOperator{\ind}{ind}
\DeclareMathOperator{\Aut}{Aut}
\DeclareMathOperator{\Quot}{Quot}
\DeclareMathOperator{\Hom}{Hom}
\DeclareMathOperator{\Lie}{Lie}
\DeclareMathOperator{\End}{End}
\DeclareMathOperator{\Ext}{Ext}
\DeclareMathOperator{\Stab}{Stab}
\DeclareMathOperator{\coker}{coker}
\DeclareMathOperator{\rad}{rad}

\newcommand{\dimv}{{\bf dim}\,}

\newcommand{\form}{\langle-,-\rangle}

\usepackage{amsthm}
\theoremstyle{plain}
\usepackage{amssymb,amsmath,amsfonts,amscd,euscript}

\numberwithin{equation}{section}
\newtheorem{thm}{Theorem}[section]
\newtheorem*{mainthm}{Theorem 7.1}
\newtheorem{prop}[thm]{Proposition}
\newtheorem{lem}[thm]{Lemma}
\newtheorem{cor}[thm]{Corollary}
\theoremstyle{remark}
\newtheorem{rem}[thm]{Remark}
\newtheorem{definition}[thm]{Definition}
\newtheorem{example}[thm]{Example}

\newtheorem{question}{Question}[section]

\begin{document}

\title[Degenerations for symmetric quivers]{On degenerations and extensions\\ of symplectic and orthogonal\\ quiver representations}
\author{Magdalena Boos, Giovanni Cerulli Irelli}
\address{Magdalena Boos:\newline Ruhr University Bochum, Faculty of Mathematics,  44780 Bochum, (Germany).} \email{magdalena.boos-math@rub.de}
\address{Giovanni Cerulli Irelli:\newline
Department S.B.A.I.. Sapienza-Universit\`a di Roma. Via Scarpa 10, 00164, Roma (Italy)}
\email{giovanni.cerulliirelli@uniroma1.it}
\begin{abstract}
We discuss degenerations of symplectic and orthogonal representations of symmetric quivers and algebras with self-dualities. As in the non-symmetric case, we define a partial ordering, that we call symmetric Ext-order which gives a sufficient criterion for a symmetric degeneration. Then a detailed discussion of type A  quivers and their (symmetric) representation theory via Auslander-Reiten theory leads to our main theorem which 
states that the symmetric degeneration order of a symmetric quiver of finite type
is induced by the "usual" degeneration order between representations of the underlying
quiver.
\end{abstract}
\maketitle
\tableofcontents

\newpage
\section{Introduction}\label{Sec:Intro}

\noindent Let $X$ be a complex algebraic variety acted upon by a group $\G$. Suppose that the pair $(\G,X)$ satisfies the following: 
\begin{enumerate}[label=(\arabic*)]
\item both  $\G$ and $X$ are equipped with involutions $\rho:\G\rightarrow \G$, $g\mapsto g^\rho$ and $\Delta: X\rightarrow X$, $x\mapsto \Delta x$ such that $\Delta(g\cdot \Delta x)=g^\rho\cdot x$. Denote by $\G^\rho\subset \G$ and $X^\Delta\subset X$ the sets of fixed points;
\item the group $\G$ is a subgroup of the group of invertible elements $E^\times$ of a finite-dimensional associative algebra $E$ over $\CC$;
\item the anti-involution of $\G$ given by $g\mapsto g^\ast:=(g^\rho)^{-1}$ extends to a $\CC$--linear anti-involution $f\mapsto f^\ast$ on the algebra $E$;
\item for every fixed point $x\in X^\Delta$, its stabilizer $H=\textrm{Stab}_{\G}(x)$ is the group of invertible elements of its linear span $\textrm{Span}_\CC(H)\subset E$.
\end{enumerate}
Then Magyar, Weyman and Zelevinsky \cite[Section~2.1]{MWZ} (generalizing a result of Richardson \cite{R} and one of Derksen and Weyman \cite[Theorem~2.6]{DW}) prove that for every fixed point $x\in X^\Delta$
\begin{equation}\label{Eq:MWZThm}
\G x\cap X^\Delta=\G^\rho x.
\end{equation}
Equation~\eqref{Eq:MWZThm} means that the orbit of a fixed point by the ``small'' group $\G^\rho$ is  the intersection of its orbit by the ``big'' group $G$ with the set of fixed points. It is then natural to ask if the same happens for the Zariski orbit closures:
\begin{question}\label{General Question}
Is it true that $\overline{\G x}\cap X^\Delta=\overline{\G^\rho x}$, for every $x\in X^\Delta$ ?
\end{question}
\begin{example}\label{Ex:ClassicalIntro}
Let $X\subset \textrm{Mat}_n(\CC)$ be the variety of $n\times n$ nilpotent complex matrices and $\G=\GL_n$ acting on $X$ by conjugation. Let $\Delta$ and $\rho$ be the involutions on $X$ and $\G$ such that $\G^\rho=SP_n$ is the symplectic group (or $\G^\rho=\OO_n$ is the orthogonal group) and $X^\Delta\subset\mathfrak{sp}_n=\Lie(\SP_n)$ (or $X^\Delta\subset\mathfrak{o}_n=\Lie(\OO_n)$) is the nilpotent cone. It is straightforward to check that the hypotheses (1)--(4) are satisfied for the pair $(X,\G)$.
By classical work of Freudenthal, Gerstenhaber and Hesselink (compare with \cite[Proposition~2.1]{KP}) it is known that the answer to Question~\ref{General Question} is positive in this case. 
\end{example}
\noindent
A pair $(X,\G)$ as above can be found in the context of representation theory of  finite dimensional algebras endowed with a self-duality (see Section~\ref{Sec:RelationWithMWZ}) and we give another proof of Magyar, Weyman and Zelevinsky's result  in this specific context in Section~\ref{Sec:IsoClasses}. We are aware of  an example arising in this context where the answer to Question~\ref{General Question} is negative \cite{BC}. In this paper we provide a positive answer to Question~\ref{General Question} in case of \emph{symmetric quiver algebras of finite representation type}.
\\[1ex]\noindent
Let us explain the context, the results and our motivations to seek Question~\ref{General Question}.

\subsubsection*{Quiver algebras}
Let $\mathrm{k}=\CC$ be the field of complex numbers. A \emph{quiver} $\QQ$ is a finite oriented graph given by a quadruple $\QQ=(\QQ_0,\QQ_1,s,t)$ where $\QQ_0$ denotes the finite set of  \emph{vertices} of $\QQ$, $\QQ_1$ denotes the finite set of  \emph{edges} and $s,t:\QQ_1\rightarrow \QQ_0$ are two functions that provide the orientation $\alpha: s(\alpha) \rightarrow t(\alpha)$ of the edges.  For simplicity of notation we consider the elements of $\QQ_1$ as oriented edges or arrows.  A \emph{path} in $\QQ$ is defined to be a sequence of arrows $\omega=\alpha_s \cdots\alpha_1$, such that $t(\alpha_{k})=s(\alpha_{k+1})$ for all $k$; formally we include a path $\ee_i$ of length zero for each $i\in \QQ_0$ starting and ending in $i$. The \emph{path algebra} $\mathrm{k}\QQ$ of $\QQ$ is defined to be the $\mathrm{k}$-vector space with a basis given by the set of all paths in $\QQ$. The multiplication of two paths is defined by concatenation of paths. Let $R$ be the \emph{arrow ideal} of $\mathrm{k}\QQ$ which is the (two-sided) ideal generated by the arrows of $\QQ$. Let $I\subseteq \mathrm{k}\QQ$  be an ideal such that there is an integer $s$ with $R^s\subseteq I\subseteq R^2$ (i.e. it is \emph{admissible}). The quotient algebra $\A:=\mathrm{k}\QQ/I$ is finite-dimensional and associative, and we refer to it as a \emph{quiver algebra} (standard references are  \cite{ARS, ASS, CB, Ringel}). 
\subsubsection*{Symmetric quiver algebras}
A  \emph{symmetric quiver} as defined in \cite{DW}  is a pair $(\QQ,\sigma)$ where $\QQ$ is a finite quiver and $\sigma:Q\rightarrow Q^{op}$ is an involutive isomorphism of $\QQ$ with its opposite $\QQ^{op}$, i.e. it is an involution of $\QQ_0$ and of $\QQ_1$, reversing the orientation of arrows. Notice that $\sigma$ is defined on the underlying graph of $\QQ$ and once this is fixed we say that $\QQ$ and $\sigma$ are \emph{compatible}.  For example, let us consider the Dynkin diagram of type $A_3$, $\xymatrix@1@C=15pt{1\ar@{-}|\alpha[r]& 2\ar@{-}|\beta[r]& 3}$, and let $\sigma$ be its non-trivial involution given by $\sigma(1)=3$, $\sigma(2)=2$, and $\sigma(\alpha)=\beta$. Let $\xymatrix@1@C=15pt{\QQ(1)=1\ar|(.7)\alpha[r]& 2\ar|\beta[r]& 3}$ and $\xymatrix@1@C=15pt{\QQ(2)=1\ar|(.7)\alpha[r]& 2&\ar|\beta[l] 3}$ be two quivers of type $A_3$. Then $\QQ(1)$ is compatible with $\sigma$, i.e. $(\QQ(1),\sigma)$ is a symmetric quiver, but $\QQ(2)$ is not compatible with $\sigma$. Actually, $\QQ(2)$ cannot be endowed with the structure of a symmetric quiver, since $\sigma$ is the only non-trivial involution of its underlying graph.  

Let $(\QQ,\sigma)$ be a symmetric quiver. For every admissible ideal $I\subset \mathrm{k}\QQ$ such that $\sigma(I)= I$, the quiver algebra $\A=\mathrm{k}\QQ/I$ is isomorphic (via $\sigma$) to its opposite  $\A^{op}=\mathrm{k}\QQ^{op}/\sigma(I)$. The pair $(\A,\sigma)$ is called a \emph{symmetric quiver algebra}.

\subsubsection*{(Symmetric) representation theory}
The representation theory of symmetric quivers or \emph{generalized quivers with dimension vectors} was developed by Derksen and Weyman in \cite{DW} and it has strong connections with the theory of symmetric spaces and $\theta$-groups of Vinberg \cite{V} and the results of Kac \cite{K1,K2}, as explained in \cite{DW} and \cite{Shm}. 
We recall here a few facts about it. Let $(\mathcal{A},\sigma)$ be a symmetric quiver algebra. An $\A$-module or an $\A$-representation is a representation of the quiver $\QQ$ satisfying the relations from $I$. Thus, a representation is a pair  $M=(V,f)$ where $V=\oplus_{i\in \QQ_0}V_i$ is a finite dimensional $\QQ_0$-graded vector space and $f=(f_\alpha:V_{s(\alpha)}\rightarrow V_{t(\alpha)})_{\alpha\in \QQ_1}$ is a collection of linear maps such that $f_{\pi}=0$ for every $\pi\in I$. The vector space $V$ is called the \emph{underlying vector space} of the representation $M$ and its graded dimension $\mathbf{d}=\dimv V=(\dim V_i)_{i\in \QQ_0}$ is called the \emph{dimension vector} of $M$ or $V$. Notice that $V$ is itself an $\A$-module, by choosing all the linear maps to be zero; it is called the semisimple module of dimension vector $\mathbf{d}$. A morphism of $\A$-modules $h:(V,f)\rightarrow (W,g)$ is a collection of linear maps $(h_i:V_i\rightarrow W_i)_{i\in \QQ_0}$ such that $h_{t(\alpha)}\circ f_{\alpha}=g_\alpha\circ h_{s(\alpha)}$. We denote by $\textrm{Rep}(\A)$ the category of finite dimensional representations of $\A$. Let $V=\oplus_{i\in \QQ_0}V_i$ be a $\QQ_0$-graded complex vector space 
of dimension vector $\mathbf{d}$. We denote by $R(\A,V)$ the variety of $\A$-modules having $V$ as underlying vector space.  Thus
$$
R(\A,V)\subseteq R(\mathrm{k}\QQ,V) := \!\!\!\!\bigoplus_{\alpha:i\rightarrow j\in \QQ_1}\Hom_\mathrm{k}(V_i,V_j).
$$
We denote by $\GL^\bullet(V):=\prod_{i\in \QQ_0} \textrm{GL}(V_i)$ the group of graded automorphisms of $V$. We sometimes think of $\GL^\bullet(V)$ as embedded into $\GL(V)$ as diagonal block matrices (see \cite[Section~1]{DW}). The group $\GL^\bullet(V)$ acts on $R(\A,V)$ by change of basis, i.e. given $g=(g_i)_{i\in\QQ_0}\in \GL^\bullet(V)$ and $M=(M_\alpha)_{\alpha\in \QQ_1}\in R(\A,V)$ the representation $g\cdot M$ is defined by $(g\cdot M)_\alpha=g_{t(\alpha)}\circ M_\alpha\circ g_{s(\alpha)}^{-1}$. The $\GL^\bullet(V)$-orbits are the  isomorphism classes of $\A$-representations with underlying vector space $V$. 

Let $\ee$ be $+1$ or $-1$. We say that
a bilinear form $\langle-,-\rangle:V\times V\rightarrow \CC$ on $V$ is a  \emph{$\sigma$-compatible $\ee$-form} on $V$ if: 
\begin{enumerate}
\item the form $\form$ is non-degenerate;
\item the form $\form$ is \emph{compatible with $\sigma$}, i.e. $\langle-,-\rangle|_{V_i\times V_j}=0$ if $j\neq \sigma(i)$; 
\item the form $\form$ is an \emph{$\ee$-form}: i.e. $\langle v,w\rangle=\ee\langle w,v\rangle$ for every $v,w\in V$.
\end{enumerate}
The pair $(V,\form)$ is called an \emph{$\ee$-quadratic space for $(\QQ,\sigma)$ or for $(\A,\sigma)$}. Notice that the dimension vector of $V$ is \emph{$\sigma$-symmetric} in this case, i.e. $\mathbf{d}_{\sigma(i)}=\mathbf{d}_i$ for every $i\in\QQ_0$. Indeed, by (ii), the  non-degenerate bilinear form $\form$ induces an isomorphism between $V_i$ and the linear dual of $V_{\sigma(i)}$.  

Every endomorphism $f$ of $V$ has a unique \emph{adjoint} $f^\star$ with respect to the non-degenerate bilinear form $\form$ defined by the condition $\langle f(v),w\rangle=\langle v,f^\star(w)\rangle$, for all $v,w,\in V$. 
We denote by $\G(V,\form)=\{g\in \GL(V)| g=(g^{\star})^{-1}\}$ the group of isometries of $(V,\langle-,-\rangle)$. 

\noindent
A point $M\in R(\A,V)$ can be considered as  an endomorphism of $V$. Following \cite{DW}, we say that  $M\in R(\A,V)$ is an \emph{$\ee$-representation} of $(\A,\sigma)$ with respect to $(V,\form)$ if 
\begin{itemize}
\item[(iv)] $M^\star+M=0$.
\end{itemize}
Condition (iv) means that $M$ lies in the Lie algebra of the group  $\G(V,\form)$ and condition (ii) implies that $M_\alpha^\star=-M_{\sigma(\alpha)}$, i.e. $\langle M_\alpha(v),w\rangle+\langle v, M_{\sigma({\alpha})}(w)\rangle=0$ for every arrow $\alpha:i\rightarrow j$, $v\in V_i$ and $w\in V_{\sigma(j)}$.   For $\ee=+1$, $M$ is hence called an \emph{orthogonal} representation of $(\A,\sigma)$ and for $\ee=-1$, it is called a \emph{symplectic} representation of $(\A,\sigma)$. 
\\[1ex]
We denote by 
$
R(\A,V)^{\form, \ee}=\{M\in R(\A,V)|\, M^\star+M=0\}
$
the variety of $\ee$-representations of $\A$ with respect to $(V,\form)$. We denote by   
$\G^\bullet(V,\form):=\G(V,\form)\cap \GL^\bullet(V)$ the group of graded isometries of $(V,\form)$. Thus $g=(g_i)\in \GL^\bullet(V)$ belongs to $\G^\bullet(V,\form)$ if and only if $g_{\sigma(i)}=(g_i^{\star})^{-1}$. The action of $\GL^\bullet(V)$ on $R(\A,V)$  induces an action of $\G^\bullet(V,\form)$ on $R(\A,V)^{\form,\ee}$ by change of basis (see Section~\ref{Sec:RelationWithMWZ}). 

\begin{example}\label{Ex:A2Quiver} [Type $(A_2,\ee)$]. Let $R=\mathrm{Mat}_{n\times n}(\CC)$ be the vector space of complex $n\times n$ matrices. 
For $\ee=-1$ let  $R^{\ee}=Sym_n$ denote the subspace of symmetric matrices and for $\ee=1$ let $R^{\ee}=ASym_n$ denote the space of anti-symmetric matrices. Let $V=\CC^n\oplus \CC^n$. We view $R\subset \End(V)$  as $\left(\begin{array}{cc}{\scriptstyle 0}&{\scriptstyle 0}\\ {\scriptstyle R}&{\scriptstyle 0}\end{array}\right)$ so that the elements of $R$ are considered as linear maps from the first copy of $\CC^n$ to the second one. 
Let $\xymatrix@1@C=15pt{\QQ=1\ar[r]& 2}$ be a quiver of type $A_2$ and let $\sigma$ be the anti-involution on $\QQ$.  Then $R$ can be seen as the representation variety $R_\mathbf{(n,n)}(\CC\QQ)$. On $V$ we define the $\sigma$-compatible $\ee$-form $\langle v,w\rangle:=v^tJw$ where $J=\left(\begin{array}{cc}{\scriptstyle 0}&{\scriptstyle \ee\mathbf{1}_n}\\ {\scriptstyle\mathbf{1}_n}&{\scriptstyle 0}\end{array}\right)$. Then, a short calculation shows that the variety of $\ee$-representations of $\QQ$ with respect to $(V,\form)$ is $R^\ee=\{A\in R| A^t+\ee A=0\}$ and thus we recover the two subspaces defined above. The group $\GL^\bullet(V)=\GL_n\times \GL_n$ acts on $R$ by base change $((g_1,g_2),A)\mapsto g_2Ag_1^{-1}$, and the group $\G^\bullet(V,\form)=\{(g_1,g_2)|\, g_1=(g_2^t)^{-1}\}\simeq \GL_n$ acts on $R^\ee$ by congruence $(g_2,A)\mapsto g_2Ag_2^t$. 
Two $n\times n$ matrices are in the same $\GL^\bullet(V)$-orbit if and only if they have the same rank. The theorem of Magyar, Weyman and Zelevinky applies here (see Section~\ref{Sec:RelationWithMWZ}) and states the well-known fact that two symmetric (or anti-symmetric) matrices are congruent if and only if they have the same rank. 
\end{example}

\begin{example}\label{Ex:A3Quiver}[Type $(A_3,\ee)$] Let $U$ be a finite dimensional complex vector space and let $U^\ast$ denote its linear dual. Let $\ee$ be $+1$ or $-1$ and let $(W,\form)$ be an $\ee$-quadratic vector space. The isometry group $\G(W,\form)$ equals $\SP(W)$  if $\ee=-1$ and $\OO(W)$ if $\ee=1$.  Let $R=\Hom(U,W)\oplus \Hom(W,U^\ast)$. Then $R$ is a representation variety for the quiver $\xymatrix@1@C=15pt{\QQ=1\ar|(.6){\alpha}[r]& 2\ar|{\beta}[r]&3}$ of type $A_3$ and we denote its elements as $(f_\alpha,f_\beta)$. Given $f_\alpha\in \Hom(U,W)$ we denote by $f_\alpha^\star\in \Hom(W,U^\ast)$ the linear map defined by $f_\alpha^\star(w)(u)=\langle f_\alpha(u),w\rangle$ for all $w\in W$ and $u\in U$. Let $R^\ee=\{(f_\alpha,f_\alpha^\star)\}\subset R$ be the subspace of fixed points for the involution $(f_\alpha, f_\beta)\mapsto (f_\beta^\star,f_\alpha^\star)$. It is easy to see that the fixed point set $R^\ee$ can be realized as a variety of $\ee$-representations for the symmetric quiver $(\QQ,\sigma)$.

Let us consider the group $\G=\GL(U)\times \GL(W)\times \GL(U^\ast)$. As usual, for an endomorphism $h$ of $U$ we denote by $h^\ast$ the endomorphism of $U^\ast$ given by $h^\ast(g)(u)=g(h(u))$ for all $g\in U^\ast$ and $u\in U$.  Similarly, for an endomorphism $h$ of $W$ we denote by $h^\star$ its adjoint with respect to $\form$ which is the unique endomorphism of $W$ such that $\langle h(w_1),w_2 \rangle=\langle w_1, h^\star(w_2)\rangle$ for every $w_1,w_2\in W$. Let us consider the involution on $\G$ given by $(g_1,g_2,g_3)\mapsto ((g_3^{\ast})^{-1}, (g_2^\star)^{-1}, (g_1^{\ast})^{-1})$ and let $\G^\ee$ be the subgroup of fixed points. We notice that $\G^\ee\simeq \GL(U)\times \G(W,\form)$. 

The group $\G$ acts on $R$ by change of basis, i.e. given $g=(g_1,g_2,g_3)$ and $f=(f_\alpha, f_\beta)\in R$, $g\cdot f=(g_2 f_\alpha g_1^{-1}, g_3 f_\beta g_2^{-1})$. It is well-known that the only invariants for this action are $\underline{rk}(f)=(rk(f_\alpha), rk(f_\beta), rk(f_\beta\circ f_\alpha))$ and thus $f'=(f'_\alpha,f'_\beta)$ lies in the $\G$-orbit of $f=(f_\alpha,f_\beta)$ if and only if $\underline{rk}(f)=\underline{rk}(f')$.

The action of $\G$ on $R$ restricts to an action of $\G^\ee$ on $R^\ee$. 
We notice that for a point $f=(f_\alpha,f_\alpha^\star)$ in $R^\ee$, the rank of $f_\alpha^\star$ is equal to the rank of $f_\alpha$ and the rank of $f_\alpha^\star\circ f_\alpha$ is the rank of the form on $W$ restricted to the image of $f_\alpha$, and thus these two ranks are preserved by the action of $\G^\ee$ on $R^\ee$.

The theorem of Magyar, Weyman and Zelevinsky applies to this example (see Section~\ref{Sec:RelationWithMWZ}) and states that two points $f=(f_\alpha, f_\alpha^\star)$ and $f'=(f'_\alpha,(f'_\alpha)^\star)$ in $R^\ee$ are in the same $\G^\ee$-orbit if and only if $rk(f_\alpha)=rk (f'_\alpha)$ and the rank of the form restricted to the image of $f_\alpha$ is equal to the rank of the form restricted to the image of $f'_\alpha$. 
\end{example}

\subsubsection*{(Symmetric) degenerations and orderings}
For  $M,N\in R(\A,V)$ we denote $M\degg N$ if and only if $N\in\overline{\GL^\bullet(V) M}$ and we say that \emph{$M$ degenerates to $N$} or that \emph{$N$ is a degeneration of $M$}. For $\ee=+1$ or $-1$ and two $\ee$-representations $M,N\in R(\A,V)^{\form,\ee}$, we  denote $M\degg^\ee N$ if and only if $N\in \overline{\G^\bullet(V,\form) M}$ and say that \emph{$M$ degenerates symmetrically to $N$} or that \emph{$N$ is a symmetric degeneration of $M$}. In Section~\ref{Sec:RelationWithMWZ} we show that the theorem of Magyar, Weyman and Zelevinsky applies to this situation and a positive answer to Question~\ref{General Question} means that $\degg$  and $\degg^ \ee$ coincide on $R(\A,V)^{\form,\ee}$.  In Section~\ref{Sec:IsoDegs} we define a third order that we call \emph{symmetric Ext-order}, denoted by $\leq_{\textrm{Ext}}^\ee$ (see Definition~\ref{Def:ExtSymmOrder}).  We obtain Corollary~\ref{cor:ExtDeg}, 
which states that 
$\xymatrix{\leq_{\Ext}^\ee\ar@{=>}[r]&\leq_{\deg}^\ee}.$

\subsubsection*{Algebras of finite representation type} Question~\ref{General Question} is particularly interesting when the orbit closures can be described by a finite number of invariants. In the context of symmetric representation varieties, this requirement is fullfilled by an algebra $\A$ which is of finite representation type. Indeed, in this case, Zwara \cite{Zwara}  proved that the degeneration order $\degg$  is equivalent to the so-called \emph{Hom-order} (see Section~\ref{Subsec:DegHomExt}). This implies that the orbit closures in $X$ are described by checking a finite number of inequalities  which is independent on $X$ and $G$.  Even in this special case, the answer to Question~\ref{General Question} is negative  \cite{BC}.   It is then an open problem to find conditions on $(X,G)$ or on the algebra $\A$ such that the answer to Question~\ref{General Question} is positive. In this paper we investigate this problem in the case of a symmetric hereditary quiver algebras of finite representation type. These algebras are the path algebras of symmetric quivers of finite type. A symmetric quiver is called of \emph{finite type} if it admits only a finite number of \emph{indecomposable $\ee$--representations} (see Section~\ref{Sec:IndecEpsilon} for a description of indecomposable $\ee$-representations). In \cite[Proposition~3.3]{DW} it is shown that a (connected) symmetric quiver is of finite type if and only if  it is a symmetric orientation of a Dynkin diagram of type $A$. We hence restrict ourselves to this situation. 

For $n\geq 2$, we denote by $A_n$ the Dynkin diagram of type $A_n$:
with the following labelling of vertices and edges:
$$
\xymatrix{
A_n:&1\ar@{-}^{e_1}[r]&2\ar@{-}^{e_2}[r]&\cdots\ar@{-}^{e_{n-2}}[r]&n-1\ar@{-}^(.6){e_{n-1}}[r]&n
.}
$$

Let $\sigma$ be the non-trivial automorphism of $A_n$ (i.e. $\sigma(i)=n+1-i$ and $\sigma(\alpha_{i})=\alpha_{n-i}$). In case $n$ is odd there is a unique $\sigma$-fixed vertex and no $\sigma$-fixed arrow; if $n$ is even there is a unique $\sigma$-fixed edge and no $\sigma$-fixed vertex. We distinguish the two cases by writing $A_{odd}$ and $A_{even}$. Let $(\QQ,\sigma)$ be  a symmetric quiver of type $A_n$. A typical example is the equioriented quiver  that we denote by

$$
\xymatrix{
\stackrel{\rightarrow}{A_n}:&1\ar[r]&2\ar
[r]&\cdots\ar[r]&n-1\ar[r]&n.
}
$$
Other examples are the alternating quivers with an even number of vertices (having $\alpha$ as $\sigma$-fixed arrow)
$$
\xymatrix@!C=6pt@!R=6pt{
&2\ar[dl]\ar[dr]&&4\ar[dl]\ar[dr]&&\ar[dl]\cdots\ar[dr]&&n+1\ar_\alpha[dl]\ar[dr]&&\ar[dl]\cdots\ar[dr]&&2n\ar[dl]\\
1&&3&&5&\cdots&n&&n+2&\cdots&2n-1&
}
$$
and those with an odd number of vertices (having $n+1$ as $\sigma$-fixed vertex):
$$
\xymatrix@C=3pt@!R=5pt{
&2\ar[dl]\ar[dr]&&\cdots\ar[dr]\ar[dl]&&n\ar[dr]\ar[dl]&& && && &\\
1&&3&&n-1&&n+1\ar[dr]& &n+3\ar[dr]\ar[dl]& &2n-1\ar[dr]\ar[dl]& &2n+1\ar[dl]\\
&&&&&&&n+2&&\cdots &&2n &
}
$$

\subsubsection{Main result}  The following is the main result of the paper: 
\begin{mainthm}
Let $(Q,\sigma)$ be a symmetric quiver of  type $A$. Let $\ee$ be $+1$ or $-1$.
Then on the $\ee$-representations we have
\begin{equation}\label{Eq:MainIntro}
\xymatrix{\leq_{\Ext}^\ee\ar@{<=>}[r]&\leq_{\deg}^\ee\ar@{<=>}[r]&\degg}
\end{equation}
\end{mainthm} 
\noindent 
The implication $\xymatrix@1{\leq_{\Ext}^\ee\ar@{=>}[r]&\leq_{\deg}^\ee}$ is proved in Corollary~\ref{Sec:RepTheoryTypeA} for any symmetric quiver algebra. The implication $\xymatrix@1{\degg\ar@{=>}[r]&\leq_{\Ext}^\ee}$ is the most surprising and it is proved in Section~\ref{Sec:MainResult} (see Theorem~\ref{Thm:MainThm}). The proof is constructive in the sense that given two $\ee$-representations $M,N\in R(\QQ,V)^{\form,\ee}$ such that $M\degg N$ it shows how to  inductively find a sequence of $\ee$-representations $M=M(0), M(1),\cdots, M(k)=N$  such that there is a one-parameter subgroup $\lambda_i(t)\in \G^\bullet(V,\form)$ which fullfills $\lim_{t\rightarrow 0}\lambda_i(t)\cdot M(i)=M(i+1)$ for every $i$ (see Section~\ref{Sec:Example} for examples). This strategy is inspired by a variation of Bongartz's proof of the classical implication $\xymatrix@1{\leq_{\deg}\ar@{=>}[r]&\leq_{\Ext}}$ for Dynkin quivers that we recall in Corollary~\ref{Cor:EquivalenceOrdersDynkin}). Thus, it is heavily based on the celebrated cancellation theorem of Bongartz \cite[Theorem~2.4]{Bongartz} that we recall in Subsection~\ref{Sec:TheoremBongartz}. The main new ingredient is what we call \emph{generic $\ee$-subquotients} of an $\ee$-representation by an indecomposable \emph{isotropic} subrepresentation. Proposition~\ref{Prop:KeyLemma}  states that (under some mild hypotheses) it is possible to embed \emph{isotropically} an indecomposable representation into an $\ee$-representation and this embedding is generic  (see Subsection~\ref{Sec:GenericQuotKer} for the definition of generic quotients, generic embeddings and their dual version). 

Our goal is to make the paper self-contained and spare the reader a long search through the literature. In Section~\ref{Sec:RepTheoryTypeA} we recall with proofs many special features of the representation theory and Auslander-Reiten theory of quivers of type $A$, used in the proof of Theorem~\ref{Thm:MainThm}. We address the reader to the references only for those aspects which hold in greater generality (e.g. for Dynkin quivers, or for gentle algebras) and whose proofs are not simpler when restricted to our context. Subsection~\ref{Sec:GenericQuot} contains an explicit description of generic quotients by  indecomposables. Section~\ref{Sec:SymmTypeA} is dedicated to highlight essential facts about $\ee$-representations of quivers of type $A$, including the key concepts of generic isotropic subrepresentations and generic $\ee$-subquotients.

We conjecture that  Theorem~\ref{Thm:MainThm} holds for every symmetric algebra which is representation-directed. 
It is an open problem to investigate if Theorem~\ref{Thm:MainThm} holds true for every symmetric quiver algebra of finite representation type whose  indecomposable modules are rigid (compare with \cite[Theorem~2]{Zwara}). 
\begin{example}
In the context of Example~\ref{Ex:A2Quiver}, Theorem~\ref{Thm:MainThm} states that given two symmetric (or anti-symmetric) matrices $A$ and $B$, then $B\in \overline{\GL_n A}$ if and only if $rk(B)\leq rk(A)$.
\end{example}
\begin{example}
In the context of Example~\ref{Ex:A3Quiver}, Theorem~\ref{Thm:MainThm} states that given two adjoint pairs $f=(f_\alpha,f_\alpha^\star)$ and $f'=((f')_\alpha,(f')_\alpha^\star)$ in $R^\ee$, then $f'\in \overline{\G^\ee f}$ if and only if $\underline{rk}(f')\leq \underline{rk}(f)$.
\end{example}
\subsubsection*{Motivations}
One of our motivations to seek Theorem~\ref{Thm:MainThm} comes from the study of linear degenerations of the symplectic (complete) flag variety, in analogy with \cite{CFFFR, CFFFR2}. Indeed, let $\A=\mathrm{k}\QQ$ be the path algebra of  the equioriented quiver of type $A_{2n-1}$ and let $(V,\form)$ be a $(-1)$-quadratic space for $(\QQ,\sigma)$ of dimension vector $\mathbf{d}=(2n,2n,\cdots, 2n)$. Then one can construct a $\GL^\bullet(V)^{\form, (-1)}$-equivariant family $\mathcal{X}\rightarrow R(\A,V)^{\form,(-1)}$ whose fiber over a symplectic representation consists of its \emph{isotropic} (actually Lagrangian) subrepresentations of dimension vector $\mathbf{e}=(1,2,\cdots, 2n-1)$. One sees immediately that the fiber over  the generic point is the complete symplectic flag variety. It is then an interesting problem to understand the geometric properties of the special fibers, which can be considered as ``linear'' degenerations of the symplectic flag variety. The first step to study this family, is to understand the orbit closures in the base. This is now clear thanks to Theorem~\ref{Thm:MainThm}. 

The general answer of Question \ref{General Question} is furthermore interesting from an algebraic Lie-theoretic point of view: It helps to answer questions about orbit closures in the nilpotent cone. In more detail, the above described classical example of reductive actions on the nilpotent cone can be generalized in many ways (e.g. restrict to Borel-actions on certain subvarieties of nilpotent matrices as in \cite{BCIE}). This setup can be translated via an associated fibre bundle to the (symmetric) representation theory of a particular quiver with relations.

\textbf{Acknowledgements.} We thank Andrea Maffei, Corrado De Concini, Hans Franzen, Francesco Esposito and Alessandro D'Andrea for helpful conversations. We thank Raquel Coelho, Pierre-Guy Plamondon, Daniel Labardini-Fragoso, Jan Schr\"oer,  and Gregorz Zwara for discussions about a possible generalization of Theorem~\ref{Thm:MainThm} to algebras associated with partial triangulations of polygons. 

This work was sponsored by DFG Forschungsstipendium BO 5359/1-1, DFG Rückkehrstipendium BO 5359/3-1  and DFG Sachbeihilfe BO 5359/2-1.

\section{Algebras with self-dualities}\label{Sec:SymQuiver}
\noindent
In this section we give the proof of the result of Magyar, Weyman and Zelevinsky mentioned in the introduction, applied to the symmetric representation varieties. Along the way we introduce notation that will be used throughout the paper.
\subsection{The self-duality}
In this section, we introduce the self-duality which comes along with a symmetric quiver algebra. Let $\A:=\mathrm{k}\QQ/I$ be a symmetric quiver algebra associated with a symmetric quiver $(\QQ,\sigma)$. The anti-involution $\sigma$ induces an isomorphism $\sigma:\A\rightarrow \A^{op}$ which induces an equivalence $\sigma:\textrm{Rep}(\A)\rightarrow \textrm{Rep}(\A^{op})$ of the representation categories. By composing with the standard $\mathrm{k}$-duality $D=\Hom(-,\mathrm{k})$ we get a self-duality on $\textrm{Rep}(\A)$ that we denote by ${}^\ast:\textrm{Rep}(\A)\rightarrow \Rep(\A)$. 

\begin{example}\label{Ex:TypeA3SelfDuality} 
Let $\xymatrix@1@C=20pt{\QQ=1\ar|(.6)\alpha[r]& 2\ar|\beta[r]& 3}$. The functor ${}^\ast$ is defined on objects by

$$
\def\g#1{\save [].[rr]!C="g#1"*+<10pt>[F-:<2pt>]\frm{}\restore}%
\xymatrix@1@C=30pt{\g1 M=(V_1\ar|(.6)f[r]& V_2\ar|g[r]& V_3)&\g2  M^\ast=(V_3^\ast\ar|(.6){(g^\ast)}[r]&V_2^\ast\ar|{(f^\ast)}[r]& V_1^\ast).
\ar @{|->} "g1";"g2"
}
$$
and on morphisms 
$$
\def\g#1{\save [].[drrr]!C="g#1"*+<10pt>[F-:<2pt>]\frm{}\restore}%
\xymatrix@1{
\g1 M_1:\ar^h[d]&V_1\ar^{f_1}[r]\ar^{h_1}[d]& V_2\ar^{g_1}[r]\ar^{h_2}[d]& V_3\ar_{h_3}[d]&\g2 M_2^\ast:\ar^{h^\ast}[d]&W_3^\ast\ar^{g_2^\ast}[r]\ar^{h_3^\ast}[d]&W_2^\ast\ar^{f_2^\ast}[r]\ar^{h_2^\ast}[d]& W_1^\ast\ar_{h_1^\ast}[d]\\
*{M_2}:&W_1\ar^{f_2}[r]& W_2\ar^{g_2}[r]& W_3&  M_1^\ast:&V_3^\ast\ar^{g_1^\ast}[r]&V_2^\ast\ar^{f_1^\ast}[r]& V_1^\ast
\ar @{|->} "g1";"g2"
}
$$
\end{example}
\noindent
Here and always throughout the paper, given a vector space $V$, we denote by $V^\ast=\Hom(V,\mathrm{k})$ its linear dual  and given a linear map $f:U\rightarrow V$ between two vector spaces, its dual is the linear map $f^\ast:V^\ast \rightarrow U^\ast$ defined by $f^\ast(h)(u)=h(f(u))$ for every $h\in V^\ast$ and $u\in U$.

\subsection{The functor $\nabla$}\label{Sec:Nabla}
It is convenient to introduce the following twist of the self-duality ${}^\ast$. For a $\QQ_0$-graded vector space $V=\oplus_{i\in \QQ_0} V_i$, we define its ``twisted dual'' $\nabla V$ as the $\QQ_0$-graded vector space whose $i$-th component is $(\nabla V)_i=V^\ast_{\sigma(i)}$. We define a functor $\nabla:\Rep(\A)\rightarrow \Rep(\A)$ as follows: given a representation $M\in R(\A,V)$ its twisted dual is a representation $\nabla M\in R(\A,\nabla V)$ defined by $\nabla(M)_\alpha=-M_{\sigma(\alpha)}^\ast$ for every arrow $\alpha$ (notice the minus sign); given a morphism $h:M\rightarrow N$ its twisted dual is defined by $(\nabla h)_i=h_{\sigma(i)}^\ast$, for every vertex $i\in\QQ_0$. 
\begin{example}\label{Ex:TypeA3Nabla}
Let $\xymatrix@1@C=20pt{\QQ=1\ar|(.6)\alpha[r]& 2\ar|\beta[r]& 3}$. The functor $\nabla$ is defined  on objects by
$$
\def\g#1{\save [].[rr]!C="g#1"*+<10pt>[F-:<2pt>]\frm{}\restore}%
\xymatrix@1@C=35pt{\g1 M=(V_1\ar|(.6)f[r]& V_2\ar|g[r]& V_3)&\g2\nabla M=(V_3^\ast\ar|(.6){(-g^\ast)}[r]&V_2^\ast\ar|{(-f^\ast)}[r]& V_1^\ast).
\ar @{|->} "g1";"g2"
}
$$ and on morphisms as the self-duality ${}^\ast$ (see Example~\ref{Ex:TypeA3SelfDuality}).
\end{example}
\subsection{The isomorphism $\Psi$}\label{Subsec:Psi}
Let $(V,\form)$ be an $\ee$-quadratic space for the symmetric quiver $(\QQ,\sigma)$.   The form $\form$ induces the isomorphism $\Psi: V\rightarrow \nabla V$ given by $\Psi(v)=\langle v,-\rangle$ for every $v\in V$. Moreover, this isomorphim satisfies the relation: 
\begin{equation}\label{Eq:Psi}
\nabla \Psi=\ee\Psi.
\end{equation}
Indeed, by canonically identifying $\nabla\nabla V\simeq V$ via the evaluation map, we get $\nabla \Psi (v_1)(v_2)=\Psi(v_2)(v_1)=\langle v_2, v_1\rangle=\ee\langle v_1,v_2\rangle=\ee\Psi(v_1)(v_2)$.
Viceversa, if $\Psi:V\rightarrow \nabla V$ is an isomorphism of $\QQ_0$-graded vector spaces satisfying \eqref{Eq:Psi}, then the bilinear form on $V$ given by $\langle v_1,v_2\rangle:=\Psi(v_1)(v_2)$ is a $\sigma$-compatible $\ee$-form on $V$. Thus the datum of a $\sigma$-compatible $\ee$-form on $V$ is equivalent to the datum of an isomorphism $\Psi:V\rightarrow \nabla V$ satisfying \eqref{Eq:Psi}. With abuse of terminology we sometimes say that $\Psi$ is an $\ee$-form on $V$.

With this ingredient, the definition of an $\ee$-representation  is reformulated as follows: a representation $M$ with underlying vector space $V$ is an $\ee$-representation with respect to $\form$ if and only if  $\Psi$ is an $\A$-isomorphism from $M$ to $\nabla M$. Thus we sometimes use the following notation for the variety $R(\A,V)^{\form,\ee}$: 
\begin{equation}\label{Eq:RepVarPsi}
R(\A,V)^{\Psi,\ee}=\{M\in R(\A,V)|\, \Psi\in \Hom_\A(M,\nabla M)\}.
\end{equation}
In terms of $\Psi$, an element $g\in\GL^\bullet(V)$ belongs to the group of graded isometries $\G^\bullet(V,\form)$ if and only if $\nabla g\circ\Psi\circ g=\Psi$ and we use the following notation:
$$
\G^\bullet(V,\Psi)=\{g\in\GL^\bullet(V)|\, \nabla g\circ\Psi\circ g=\Psi\}.
$$
Two $\ee$-representations $M,N\in R(\A,V)^{\Psi,\ee}$ are isomorphic as $\ee$-representations if there exists an isomorphism $\theta:M\rightarrow N$ of $\A$-representations such that $\nabla\theta\circ\Psi\circ\theta=\Psi$. 
\noindent
We sometimes write a point of $R(\A,V)^{\Psi,\ee}$ as $(M,\Psi)$ to highlight the dependency on $\Psi$ and we say that $M$ is an $\ee$-representation with respect to the $\ee$-form $\Psi$.
\begin{rem}\label{Rem:Psi}
An $\ee$-representation $M$ is isomorphic to its `` twisted dual'' $\nabla M$ via the isomorphism $\Psi$. We sometimes omit to mention this isomorphism and freely identify $M$ and $\nabla M$. Thus, for example, if $f:L\rightarrow M$ is an homomorphism then $\nabla f: \nabla M\rightarrow \nabla L$ and we sometimes write $\nabla f\circ f:L\rightarrow \nabla L$ instead of $\nabla f\circ\Psi\circ f$.
\end{rem}

\subsection{Embedding into the general context}\label{Sec:RelationWithMWZ}
In this section we show that the theorem of Magyar, Weyman and Zelevinsky mentioned in the introduction applies to the symmetric representation varieties. 
Let $(\A,\sigma)$ be a symmetric quiver algebra, let $(V,\form)$ be an $\ee$-quadratic space for $(\A,\sigma)$, let $X=R(\A,V)$ and $\G=\GL^\bullet(V)$. Let $E=\End^\bullet(V)=\prod_{i\in\QQ_0}\End(V_i)$ be the finite dimensional algebra of graded endomorphisms of $V$. Thus $\G\subset E$ is the group of invertible elements of $E$. For $M\in R(\A,V)$ we denote by $\Delta M\in R(\A, V)$ the representation $(\Delta M)_\alpha:=-M_{\sigma(\alpha)}^\star: V_i\rightarrow V_j$ for every arrow $\alpha:i\rightarrow j$. Similarly, for $f=(f_i)_{i\in\QQ_0}\in E$ we denote by $\Delta f\in E$ the graded endomorphism of $V$ given by $(\Delta f)_i=f_{\sigma(i)}^\star:V_i\rightarrow V_i$. Notice that ${}^\star$ denotes the adjoint with respect to the bilinear form $\form$ and thus the functor $\Delta$ is slightly different from the functor $\nabla$ defined in Section~\ref{Sec:Nabla}. On the pair $(X,G)$ we consider the two involutions $\Delta:X\rightarrow X: M\mapsto \Delta M$ and $\rho:G\rightarrow G: g\mapsto g^\rho=\Delta(g)^{-1}$.
The variety $R(\A,V)^{\form,\ee}$ of $\ee$-representations of $(\A,\sigma)$ with respect to $(V,\form)$ is hence the variety $X^\Delta$ of $\Delta$-fixed points and the group $\G^\bullet(V,\form)$ of graded isometries of $(V,\form)$ is the group $G^\rho$ of $\rho$-fixed points.  We now have to check that the four hypotheses mentioned in the introduction are satisfied by the pair $(X,G)$ endowed with the two involutions $\Delta$ and $\rho$. 
\begin{enumerate}[label=(\arabic*)]
\item Let $\alpha:i\rightarrow j$ be an arrow of $\QQ$ and thus $\sigma(\alpha):\sigma(j)\rightarrow \sigma(i)$ is the symmetric arrow. Let $g\in G$ and $M\in X$. Then $(g\cdot M)_\alpha=g_j\circ M_\alpha\circ g_i^{-1}$ and $(g\cdot \Delta M)_{\sigma(\alpha)}=-g_{\sigma(i)}\circ M_{\alpha}^\star\circ g_{\sigma(j)}^{-1}$. Thus, 
\begin{eqnarray*}
\Delta(g\cdot \Delta M)_\alpha&=&-(g\cdot \Delta M)_{\sigma(\alpha)}^\star =(g_{\sigma(i)}\circ M_{\alpha}^\star \circ g_{\sigma(j)}^{-1})^\star\\
&=&(g_{\sigma(j)}^{-1})^\star\circ M_{\alpha}\circ g_{\sigma(i)}^\star=(g^\rho\cdot  M)_\alpha.
\end{eqnarray*}
\item The group $G$ is the group of invertible elements of the algebra $E$.
\item The anti-involution of $G$ given by $g\mapsto (g^\rho)^{-1}=\Delta g$ is the restriction to $G$ of the linear anti-involution of $E$ given by $f\mapsto \Delta f$.
\item For every $\ee$-representation $M$, let $H=\Aut(M)$ denote its $G$-stabilizer. Then an element $h\in H$ is a collection $h=(h_i: V_i\rightarrow V_i)_{i\in \QQ_0}$ of linear isomorphisms such that $h_jM_\alpha h_i^{-1}=M_\alpha$ for every arrow $\alpha:i\rightarrow j$ of $\QQ$. This means that $h_jM_\alpha=M_\alpha h_i$. The linear span of $H$ hence consists of those $f=(f_i)\in E$ such that $f_jM_\alpha=M_\alpha f_i$ and $H$ is the group of invertible elements of its linear span.
\end{enumerate}
Thus  \cite[Prop.~2.1]{MWZ} applies to this situation and we have $G\cdot M\cap X^\Delta=G^\rho M$ for every $M\in X^\Delta$. We give another prove of this result in the subsequent section~\ref{Sec:IsoClasses}.

\subsection{Automorphism groups of $\ee$-representations} 
Let $M\in R(\A,V)^{\Psi,\ee}$ be an $\ee$-representation with respect to the $\ee$-form $\Psi:V\rightarrow\nabla V$. The $\A$-homomorphism space $\Hom(M,\nabla M)$ admits a decomposition $\Hom(M,\nabla M)=\Hom(M,\nabla M)^{\nabla}\oplus\Hom(M,\nabla M)^{-\nabla}$ into eigenspaces for the linear involution $f\mapsto\nabla f$, where 
\begin{equation}\label{Eq:HomPMNabla}
\Hom(M,\nabla M)^{\pm\nabla}:=\{\theta\in\Hom(M,\nabla M)|\,\nabla\theta=\pm\theta\}.
\end{equation} 
For instance, $\Psi\in\Hom(M,\nabla M)^{\ee\nabla}$.
An \emph{automorphism of the $\ee$-representation} $(M,\Psi)$ is an automorphism $\theta:M\rightarrow M$ of the $\A$-representation $M$ such that $\nabla\theta\circ\Psi\circ\theta=\Psi$; in terms of the bilinear form this means that $\langle\theta(v_1),\theta(v_2)\rangle=\langle v_1,v_2\rangle$ for every $v_1,v_2\in V$. We denote by $\Aut(M,\Psi)$ the group of automorphisms of $(M,\Psi)$. In particular, $\Aut(M,\Psi)$ is a subgroup of $\GL^\bullet(V)=\prod_{i\in Q_0}\GL(V_i)$. The Lie algebra $\Lie\Aut(M,\Psi)$  consists of those $\theta\in\Hom(M,M)$ such that
$$
\Psi\circ\theta+\nabla\theta\circ\Psi=0
$$
i.e. $\langle\theta(v_1),v_2\rangle+\langle v_1,\theta(v_2)\rangle=0$ for every $v_1,v_2\in V$. 
The dimension of $\Aut(M,\Psi)$ can be computed as follows.

\begin{lem}\label{Lem:SymplDimAut}Let $(M,\Psi)$ be an $\ee$-representation. Then
\begin{equation}
\Lie \Aut(M,\Psi)\simeq\Hom(M,\nabla M)^{-\ee\nabla}
\end{equation}
In particular, $\dim\Aut(M,\Psi)=\dim\Hom(M,\nabla M)^{-\ee\nabla}$.
\end{lem}
\begin{proof}
Consider the map $\zeta:\Lie \Aut(M,\Psi)\rightarrow \Hom(M,\nabla M)^{-\ee\nabla}:\theta\mapsto\Psi\circ\theta$.
This map is well-defined: $\Psi\circ\theta$ is indeed a homomorphism $M\rightarrow\nabla M$ and moreover
$$
\nabla (\Psi\circ\theta)=\nabla\theta\circ\nabla\Psi=\ee\nabla\theta\circ\Psi=-\ee \Psi\circ\theta.
$$
The map $\zeta$ is an isomorphism of vector spaces: indeed it is clearly linear and the inverse is given by 
$\zeta':\Hom(M,\nabla M)^{-\ee\nabla}\rightarrow \Lie \Aut(M,\Psi):f\mapsto\Psi^{-1}\circ f$. \qedhere
\end{proof}

\subsection{Isomorphism classes of symmetric representations}\label{Sec:IsoClasses}
The following result was proved by Derksen and Weyman in \cite[Theorem~2.6]{DW} for symmetric quivers without relations and it is a special case of the result proved by Magyar, Weyman and Zelevinsky mentioned in the introduction (see Section~\ref{Sec:RelationWithMWZ}). We provide a different proof for symmetric quiver algebras.

\begin{thm}\label{thm:isoclasses}
Let $M, N\in R(\A,V)^{\Psi,\ee}$  be two $\ee$-representations of $\A$. Then they are isomorphic as $\ee$-representations if and only if they are isomorphic as $\A$-representations.
\end{thm}
\begin{proof}
If $M$ and $N$ are isomorphic as $\ee$-representations, then they are isomorphic as $\A$-representations since the group $\G^\bullet(V,\Psi)$ is a subgroup of $\GL^\bullet(V)$. To prove the converse, we choose an isomorphism $\theta:M\rightarrow N$ and we prove that there exists $\rho\in\Aut(M)$, such that $\theta\circ\rho$  fulfills $\nabla (\theta\circ\rho)\circ\Psi\circ(\theta\circ\rho)=\Psi$.
\\[1ex]
The group $\Aut(M)$ acts from the right on $\Hom(M,\nabla M)^{\ee\nabla}$ as follows
$$
\xymatrix@R=5pt{
\Hom(M,\nabla M)^{\ee\nabla}\times \Aut(M)\ar[r]&\Hom(M,\nabla M)^{\ee\nabla},\\
(\xi,\rho)\ar@{|->}[r]&\xi\cdot\rho:=\nabla\rho\circ\xi\circ\rho.
}
$$
\\[1ex]
Let $\Hom^0(M,\nabla M)^{\ee\nabla}\subseteq\Hom(M,\nabla M)^{\ee\nabla}$ denote the open subset consisting of invertible $(\ee\nabla)$-invariant homomorphisms. Since $\Psi\in\Hom(M,\nabla M)^{\ee\nabla}$ is an isomorphism, we see that $\Hom^0(M,\nabla M)^{\ee\nabla}$ is non-empty and hence dense in the vector space $\Hom(M,\nabla M)^{\ee\nabla}$. 
\\[1ex]
The $\Aut(M)$-action above descends to an action on $\Hom^0(M,\nabla M)^{\ee\nabla}$. We denote by $\Stab_{\Aut(M)}(\pi)$ the stabilizer of a point $\pi\in\Hom^0(M,\nabla M)^{\ee\nabla}$ for this action. Then  $\Stab_{\Aut(M)}(\pi)=\Aut(M,\pi)$. By Lemma \ref{Lem:SymplDimAut} we get $\dim\Stab_{\Aut(M)}(\pi)=\dim \Hom(M,\nabla M)^{-\ee\nabla}$. 
It follows that
$$
\dim \Aut(M)=\dim\Hom(M,\nabla M)=\dim \Stab_{\Aut(M)}(\pi)+\dim\Hom(M,\nabla M)^{\ee\nabla}.
$$
We hence see that every point in $\textrm{Hom}^0(M,\nabla M)^{\ee\nabla}$ has a dense $\Aut(M)$-orbit. Since $\Hom(M,\nabla M)^{\ee\nabla}$ is irreducible, being a vector space, two such orbits meet. Both $\Psi$ and $\nabla\theta\circ\Psi\circ\theta$ lie in $\textrm{Hom}^0(M,\nabla M)^{\ee\nabla}$. It follows that there exists $\rho\in\textrm{Aut}(M)$ such that $(\nabla\theta\circ\Psi\circ\theta)\cdot\rho=\Psi$, which is what we wanted to prove.
\end{proof}
\noindent
The following fact is very well-known. It basically says that two complex symmetric or anti-symmetric non-degenerate forms are conjugate under $\GL$. In terms of symmetric representation theory, it says that if $M$ is an $\ee$-representation with respect to some $\ee$-quadratic space $(V,\form)$, then for any other choice $\form'$ of a $\sigma$-compatible $\ee$-form on $V$ there is $M'\simeq M$ such that $M'$ is an $\ee$-representation with respect to  $(V,\form')$. This fact is very useful when one  makes computations with $\ee$-representations, since it allows to change the form at the best convenience.
\begin{cor}\label{cor:NoMatterForm}
Let $\Psi, \Psi'\in \Hom^0(V,\nabla V)^{\ee\nabla}$ be two $\sigma$-compatible $\ee$-forms on a $\QQ_0$-graded vector space $V$. Then there exists $g\in \GL^\bullet(V)$ such that $\nabla g\circ \Psi\circ g=\Psi'$. In particular, if $M$ is an $\ee$-representation with respect to $\Psi$, then $M'=g^{-1}\cdot M$ is an $\ee$-representation with respect to $\Psi'$.
\end{cor}
\begin{proof}
In the proof of Theorem~\ref{thm:isoclasses} replace $M$ with the semisimple representation $V$. Since $\Aut(V)=\GL^\bullet(V)$ we get the first part. To get the second part, we notice that if $\Psi\in \Hom(M,\nabla M)^{\ee\nabla}$ then $\Psi'=\nabla g\circ\Psi\circ g\in \Hom(M',\nabla M')^{\ee\nabla}$ where $M'=g^{-1}\cdot M$. 
\end{proof}
\begin{rem}\label{Rem:NoForm}
In view of Corollary~\ref{cor:NoMatterForm} we can omit to mention the form with respect to which a representation is an $\ee$-representation and we just say that a representation is an $\ee$-representation, without specifying the form. We hence sometimes use the notation $R_\mathbf{d}^\ee$ instead of $R(\A,V)^{\form, \ee}$ or $R(\A,V)^{\Psi, \ee}$ to denote the variety of $\ee$-representations of a fixed dimension vector $\mathbf{d}$. This is sloppy, since its points are defined only up to isomorphisms, but we will be very careful and use it only when there is no possibility of confusion. Similarly, we use the notation $\G_\mathbf{d}^\ee$ instead of $\G^\bullet(V,\form)$ or $\G^\bullet(V,\Psi)$ when the dependency on the form is not so relevant.
\end{rem}
\subsection{Isotropic subrepresentations}
Let $M\in R(\A,V)^{\Psi,\ee}$ be an $\ee$--representation with respect to the isomorphism $\Psi$ induced by the $\ee$-form $\langle-,-\rangle$ on the underlying vector space $V$. Let $\iota: N \hookrightarrow M$ be a subrepresentation. The \emph{orthogonal space} of $\iota(N)$ in $V$ is denoted with $\iota(N)^\perp$ and it is defined by 
\[\iota(N)^\perp=\{v\in V|\, \langle v,\iota(n)\rangle=0,\,\forall n\in N\}.\] 
In terms of the isomorphism $\Psi$, $\iota(N)^\perp$ is described in the following lemma. 
\begin{lem}\label{Lem:NperpNabla}
$\iota(N)^\perp=\Ker(\nabla\iota\circ\Psi)\simeq\nabla(M/\iota(N))$. In particular, it is a subrepresentation of $M$.
\end{lem} 
\begin{proof}
By definition $\langle v,\iota(n)\rangle=\Psi(v)(\iota(n))=\nabla\iota\Psi(v)(n)$.
\\[1ex]
The isomorphism $\Ker(\nabla\iota\circ\Psi)\simeq\nabla(M/\iota(N))$ is induced by the following commutative diagram with exact rows,  where the lower short exact sequence is obtained from $\xymatrix@1@C=10pt{0\ar[r]&N\ar^\iota[r]&M\ar[r]&M/\iota(N)\ar[r]&0}$ by applying $\nabla$:
$$
\xymatrix{
0\ar[r]&\ar[r]\iota(N)^\perp\ar@{..>}[d]&M\ar^{\nabla\iota\Psi}[r]\ar^\Psi[d]&\nabla N\ar[r]\ar@{=}[d]&0\\
0\ar[r]&\ar[r]\nabla(M/\iota(N))&\nabla M\ar^{\nabla\iota}[r]&\nabla N\ar[r]&0.
}
$$
More concretely, the $\A$-representation $\nabla(M/\iota(N))$ is the annihilator of $\iota(N)$ in the dual space $\nabla(M)$ and thus the isomorphism $\iota(N)^\perp\rightarrow\nabla(M/\iota(N))$ is simply given by $m\mapsto \langle m,-\rangle=\Psi(m)$.
\end{proof}
\noindent By definition, $\iota(N)\cap\iota(N)^\perp$ is the kernel of the form restricted to $\iota(N)^\perp$. Thus, $\iota(N)^\perp/\iota(N)\cap\iota(N)^\perp$ is an $\ee$-representation. 
\\[1ex]
A subrepresentation $\iota: N \hookrightarrow M$ is called \emph{isotropic} if $\iota(N)\subseteq \iota(N)^\perp$. 
\begin{cor}\label{Cor:Isotropic}
Let $N$  be an $\A$--representation, and $M$  be an $\ee$--representation. Let $f: N \rightarrow M$ be a non-zero homomorphisms. Then the image of $f$ is isotropic if and only if $\nabla f\circ\Psi\circ f=0$. 
\end{cor}

\subsection{Indecomposable $\ee$-representations}\label{Sec:IndecEpsilon}
Given two $\ee$-representations $(M_1,\Psi_1)$ and $(M_2,\Psi_2)$ their \emph{direct sum} is the $\ee$-representation $(M_1\oplus M_2,\Psi_1\oplus\Psi_2)$. An $\ee$-representation is called \emph{indecomposable} or \emph{$\ee$-indecomposable} if it is not the direct sum of two non-trivial $\ee$-representations. 
By the Theorem of Krull, Remak and Schmidt and by Theorem~\ref{thm:isoclasses}, every $\ee$--representation  can be written in an essentially unique way as a direct sum of indecomposable $\ee$--representations.  
The following description of the indecomposable $\ee$-representations  is proved in [Proposition~2.7]\cite{DW} for symmetric quivers without relations (see also \cite[Lemma~4.5]{Shm}); 
\begin{lem} Let $M$ be an indecomposable $\ee$-representation. Then one and only one of the following three cases can occur:
\begin{enumerate}
\item[(I)] $M$ is indecomposable as an $\A$-representation. In this case $M$ is called of type $(I)$, for ``indecomposable''.
\item[(S)] There exists an indecomposable $\A$-representation $T$ such that $M=T\oplus\nabla T$  and $T\not\simeq \nabla T$. In this case $M$ is called of type $(S)$, for ``split''.
\item[(R)] There exists an indecomposable $\A$-representation $T$ such that $M=T\oplus\nabla T$  $T\simeq \nabla T$. In this case $M$ is called of type $(R)$ for ``ramified''.
\end{enumerate}
\end{lem}
\begin{proof}
The proof of Derksen and Weyman [Proposition~2.7]\cite{DW} uses only the fact that the endomorphism ring of an indecomposable representation is local; since this fact is true for every algebra by Fitting's lemma \cite[Corollary~4.8]{ASS}, the same proof applies to the case of symmetric quiver algebras.
\end{proof}
\section{Isotropic degenerations}\label{Sec:IsoDegs}
\noindent
In order to approach the degeneration order for (non-symmetric) $\A$-representations there is a famous second ordering, namely the Ext-order. It is defined on the isoclasses of $\A$--representations of dimension vector $\mathbf{d}$ as follows (\cite{AbeasisDelFraEquioriented,Bongartz,Riedtmann,R}): 
Given $M,N\in R(\A,V)$,
\begin{equation}\label{Eq:ExtOrder}
\xymatrix@R=0pt{
M\leq_{\Ext} N \ar@{<=>}^(.22){\textrm{def}}[r]& \textrm{There exist representations }M_1,\cdots, M_k \textrm{ such that for }\\& \textrm{ every }i\textrm{ there exists a short exact sequences }\\&0\rightarrow U_i\rightarrow M_{i-1}\rightarrow V_i\rightarrow 0\\& \textrm{ such that }M_1=M, M_k=N, \, M_i\simeq U_i\oplus V_i.}
\end{equation}
Indeed, by \cite[Lemma~2.1]{Bongartz}: $\xymatrix{\leq_{\Ext}\ar@{=>}[r]&\degg}$.
In this section we provide an analogue of this result for $\ee$-representations. 

\begin{thm}\label{Thm:IsoDeg}
Let $M\in R(\A,V)^{\Psi,\ee}$ be an $\ee$--representation and let $\iota:L\hookrightarrow M$ be an isotropic subrepresentation. Then 
\begin{equation}\label{Eq:IsoDeg}
M\degg^ \ee L\oplus\nabla L\oplus (\iota(L)^\perp/\iota(L)).
\end{equation}
More precisely, there is a one-parameter subgroup $\lambda(t)\in\G^\bullet(V,\Psi)$ and a point $m\in \G^\bullet(V,\Psi) M$ such that $\lim_{t\rightarrow0}\lambda(t).m\in \G^\bullet(V,\Psi) (L\oplus\nabla L\oplus (\iota(L)^\perp/\iota(L)))$. 
\end{thm}
\begin{proof}
Since $\iota(L)\subset V$ is isotropic, for every $i\in\QQ_0$ we can choose a basis of $\iota(L)_i\subset V_i$, extend it to a basis of $\iota(L)^\perp_i\subset V_i$ and then complete it to a basis $\mathcal{B}_i$ of $V_i$ so that the $\ee$-form $\Psi=\sum_{i\in\QQ_0}\Psi_i:V\rightarrow \nabla V$ is represented in the basis $\mathcal{B}=\cup_{i\in\QQ_0}\mathcal{B}_i$ of $V$ and the dual basis $\mathcal{B}^\ast$ of $\nabla V$ by the matrices
\[
\Psi_i=\left(\begin{array}{ccc}0&0&\ee\mathbf{1}\\0&\varphi_i&0\\\mathbf{1}&0&\end{array}\right): L_i\oplus Y_i\oplus \nabla L_i\rightarrow  \nabla  L_i\oplus \nabla Y_i\oplus  L_i,
\]
where $Y_i$ denotes the span of the basis elements of $\iota(L)^\perp_i$ which do not belong to $\iota(L)_i$ and $\varphi=\oplus_i\varphi_i$ is the induced $\ee$-form on $Y=\oplus Y_i$. 
\\[1ex]
Let $\alpha:i \rightarrow j$ be an arrow of $\mathcal{Q}$. The linear map
$M_{\alpha}:L_i\oplus Y_i\oplus \nabla L_i\rightarrow L_j\oplus Y_j\oplus \nabla L_j$
is represented in the basis $\mathcal{B}$ by a block matrix of the form
$$
M_{\alpha}=\left(\begin{array}{ccc}L_{\alpha}&\nu_{\alpha}&\xi_{\alpha}\\
0&Y_{\alpha}&\mu_{\alpha}\\
0&0&\nabla L_{\alpha}\end{array}\right)
$$
where $Y=\iota(L)^\perp/\iota(L)$ (there are zeros below the diagonal because both $\iota(L)$ and $\iota(L)^\perp$ are subrepresentations of $M$).  For $t\in\CC^\ast$, let $\lambda(t)=(\lambda(t)_i)_{i\in\QQ_0}\in \GL^\bullet(V)$ be the one-parameter subgroup whose $i$-th component $\lambda(t)_i: L_i\oplus Y_i\oplus \nabla L_i\rightarrow L_i\oplus Y_i\oplus \nabla L_i$ is represented in the basis $\mathcal{B}_i$ by the block matrix
$$
\lambda(t)_i=\left(\begin{array}{ccc}t\mathbf{1}&0&0\\0&\mathbf{1}&0\\0&0&t^{-1}\mathbf{1}\end{array}\right).
$$
An immediate calculation shows that  
$\nabla \lambda(t) \Psi \lambda(t) = \Psi$
and hence $\lambda(t)$ is a subgroup of $\G^\bullet(V,\Psi)$. Let  $M(t):= \lambda(t)\cdot M$, i.e. 
\[M(t)_{\alpha}=\lambda(t)_j M_\alpha\lambda(t)_i^{-1}=\left(\begin{array}{ccc}L_{\alpha}&t\nu_{\alpha}&t^2\xi_{\alpha}\\
0&Y_{\alpha}&t\mu_{\alpha}\\
0&0&\nabla L_{\alpha}\end{array}\right) \]
Thus,  $M(t)\in\G^\bullet(V,\Psi) M$ for $t\neq 0$ and $\lim_{t\rightarrow 0}M(t)\cong L \oplus Y\oplus \nabla L$. \qedhere
\end{proof}
\noindent
\begin{definition}\label{Def:ExtSymmOrder}
Let $M,N\in R(\A,V)^\ee$ 
\begin{equation}\label{Eq:DefDegExtSym}
\xymatrix@R=0pt{
M\leq_{\Ext}^\ee N \ar@{<=>}^(.25){def}[r]& \textrm{There exist }\ee-\textrm{representations }M_1,\cdots, M_k\in R(\A,V)^\ee\\ &\textrm{ such that for every }i=2,\cdots, k\textrm{ there exists a short}\\&\textrm{exact sequence }0\rightarrow L_i\rightarrow M_{i-1}\rightarrow V_i\rightarrow 0\\& \textrm{ such that }M_1=M, M_k=N,
L_i\textrm{ is isotropic in }M_{i-1}\\&\textrm{ and }M_i\simeq L_i\oplus \nabla L_i\oplus L_i^{\perp}/L_i.}
\end{equation}
\end{definition}
\noindent
Theorem~\ref{Thm:IsoDeg} implies at once:
\begin{cor}\label{cor:ExtDeg}
$\xymatrix{\leq_{\Ext}^\ee\ar@{=>}[r]&\leq_{\deg}^\ee}$.
\end{cor}
\noindent
The opposite implication is not true in general, since an $\ee$-representation could degenerate to an indecomposable $\ee$-representation of type $(I)$ (see \cite{BC}). This is not the case for symmetric quivers of type $A$, though, as we show in Section \ref{Sec:MainResult}.

\section{Degenerations and Bongartz's cancellation theorem}\label{Sec:Orders}
\noindent In this section we recall the definition of the degeneration, the Hom, and the Ext order on the variety of $\A$-representations of a fixed dimension vector. Then we recall Bongartz's cancellation theorem and how to use it to show that these three orders coincide for Dynkin quivers. This argument is analogous to the argument that we will use to prove the main theorem  in Section~\ref{Sec:MainResult}. 
\subsection{Degeneration, Hom and Ext order}\label{Subsec:DegHomExt}
Let $\A$ be a quiver algebra. Given two $\A$-representations $X$ and $Y$ we use the notation $[X,Y]:=\textrm{dim Hom}_{\A}(X,Y)$ and $[X,Y]^1:=\textrm{dim Ext}^1_{\A}(X,Y)$. We denote by $\ind(\A)$ the set of isoclasses of indecomposable $\A$-modules.
Let $\mathbf{d}\in \ZZ_{\geq0}^{Q_0}$ be a dimension vector and let $V$ be a $Q_0$-graded complex vector space of dimension vector $\mathbf{d}$. 
\begin{itemize}
\item
The \emph{degeneration order} $\degg$ on $R(\A,V)$ is defined by 
$$
\xymatrix{
M\degg N\ar@{<=>}^(.45){def}[r]& N\in \overline{\GL^\bullet(V) M}.
}
$$
\item The Hom-order $\leq_{\Hom}$ on $R(\A,V)$ is defined by 
$$
\xymatrix{
M\leq_{\Hom} N\ar@{<=>}^(.3){def}[r]& [M,E]\leq [N,E] \textrm{ for every }E\in\ind(\A).
}
$$
\item 
The Ext-order $\leq_{\Ext}$ on $R(\A,V)$ is defined in \eqref{Eq:ExtOrder}.
\end{itemize} 
These orders were first considered by Abeasis-Del Fra for quivers of type $A$ \cite{AbeasisDelFraEquioriented, AbeasisDelFra}, and then extended to general quiver algebras by Riedtmann \cite{Riedtmann}, Bongartz \cite{Bongartz} and Zwara \cite{Zwara}.
It is well-known that 
$$
\xymatrix{
\leq_{\Ext}\ar@{=>}[r]&\degg\ar@{=>}[r]&\leq_{\Hom}
}
$$
(see \cite[Proposition~2.1]{Riedtmann}, \cite[Lemma~1.1]{Bongartz}).
It is a remarkable fact that these three orders are equivalent for every algebra of finite representation type whose indecomposables are all rigid \cite[Theorem~2]{Zwara}. In particular, they are all equivalent for Dynkin quivers. In the subsequent Section~\ref{Sec:TheoremBongartz} we recall the proof of this result.

\subsection{Generic quotients and generic kernels}\label{Sec:GenericQuotKer}
We recall from \cite[Section~2.3]{Bongartz} the crucial definition of generic quotients and of generic kernels.

Let $M\in R(\A, V)$. Let $\iota:L\hookrightarrow M$ be a subrepresentation of dimension vector $\mathbf{e}$. For every $i\in\QQ_0$, we fix a basis $\{v_1^{(i)},\cdots, v_{e_i}^{(i)}\}$ of the vector subspace $\iota(L)_i\subseteq V_i$ and extend it to a basis $\mathcal{B}_i$ of $V_i$. With this choice we identify $E_i=\textrm{End}(V_i)$ with the vector space of $d_i\times d_i$-matrices and $\GL^\bullet(V)\subset E=\oplus_{i\in \QQ_0} E_i$ as the open subset  of tuples of invertible matrices.  Moreover, $M=(M_\alpha)_\alpha$ is represented in these bases as a collection of $2\times 2$ upper triangular block matrices of the form $\left(\begin{array}{cc}{\scriptstyle L_{\alpha}}&{\scriptstyle \star}\\{\scriptstyle 0}&{\scriptstyle \star} \end{array}\right)$.  Let $G(L,M)\subset \GL^\bullet(V)$ be the set of  all collections $g=(g_i)_{i\in\QQ_0}$ of invertible matrices such that for every arrow $\alpha:i\rightarrow j$ of $\QQ$ the matrix $g_j^{-1} M_\alpha g_i$ is a $2\times 2$ upper triangular block matrix having $L_\alpha$ as upper left block. Let $U=\iota(L)\subset V$ and let us consider the representation variety $R(\QQ,V/U)$. Bongartz considers the map $\varphi: G(L,M)\rightarrow R(\QQ,V/U)$ which associates to $g$ the collection $Q=(Q_\alpha)_\alpha$ of matrices such that $Q_\alpha$ is the lower right block of $g_j^{-1} M_\alpha g_i$. The image of $\varphi$ is the set of all possible quotients of $M$ by $L$ and it is denoted by $\textrm{Quot}(L,M)$. 
\begin{lem}
$\Quot(L,M)$ is an irreducible, constructible subset of $R(\QQ, V/U)$ which is stable by the action of $\GL^\bullet(V/U)$. 
\end{lem}
\begin{proof}
Given $f_i\in E_i$ we denote by $\overline{f}_i$ the matrix consisting of the first $e_i$ columns of $f_i$. Then  Bongartz points out that $g\in G(L,M)$ if and only if $M_\alpha\overline{g_{s(\alpha)}}=\overline{g_{t(\alpha)}} L_\alpha$. These equations are linear in the entries of the matrices $\overline{g}$ and hence of $g$. Thus $G(L,M)$ is an open and dense subset of the vector subspace $E(L,M)\subset E$ consisting of all $f=(f_i)$ such that $M_\alpha\overline{f_{s(\alpha)}}=\overline{f_{t(\alpha)}} L_\alpha$. In particular, $G(L,M)$ is an irreducible locally closed set. Thus, $\textrm{Quot}(L,M)$ is a constructible irreducible set, by Chevalley's theorem. The fact that it is stable by base change is obvious, since if $Q$ is a quotient of $M$ by $L$ and $Q'\simeq Q$ then $Q'$ is a quotient of $M$ by $L$. 
\end{proof}
\noindent
A quotient $Q$ of $M$ by $L$ is called a \emph{generic quotient} of $M$ by $L$ if its orbit is dense in $\textrm{Quot}(L,M)$. In this case, a monomorphism $\iota:L\hookrightarrow M$ such that $Q\simeq M/\iota(L)$ is called a \emph{generic embedding of $L$ into $M$}.  Of course, generic quotients and generic embeddings might not exist. They exist in case $\A$ has finite representation type. In particular, they exists in case $\A$ is the path algebra of a Dynkin quiver.

Dually, let $\xymatrix{\pi:M\ar@{->>}[r]&Q}$  be a quotient of dimension vector $\mathbf{d-e}$. Then as above one constructs the set of all possible kernels of an epimorphism from $M$ to $Q$ and shows that this set is a construcible irreducible subset of $R(\QQ,W)$, stable by base change, for an appropriate vector subspace $W$ of $V$ of dimension vector $\mathbf{e}$. The generic orbit in this set is called the \emph{generic kernel} of $M$ onto $Q$. 

\subsection{The cancellation theorem of Bongartz}\label{Sec:TheoremBongartz}
We recall  the famous ``cancellation theorem'' of Bongartz \cite[Theorem~2.4]{Bongartz} and how to use it to prove the equivalence of the three orders of Subsection~\ref{Subsec:DegHomExt} for Dynkin quivers. 
\begin{thm}[Cancellation theorem]\cite[Theorem~2.4]{Bongartz}\label{Thm:Bongartz}
Let $M, N\in R(\A,V)$ such that $M\degg N$. 
\begin{enumerate}
\item Let $L\in \Rep(\QQ)$ such that $[L,M]=[L,N]$. Then if $L$ embeds into $N$, $L$ embeds into $M$, too. In this case, if the generic quotient of $M$ by $L$ exists, it degenerates to every quotient of $N$ by $L$.
\item Let $Q\in \Rep(\QQ)$ such that $[M,Q]=[N,Q]$. Then if $N$ surjects onto $Q$, $M$ surjects onto  $Q$, too. In this case, if the generic kernel of $M$ onto  $Q$ exists, it degenerates to the kernel of every surjective map from $N$ to $Q$.
\end{enumerate}
\end{thm}
\noindent
A connected quiver $\QQ$ is called a Dynkin quiver if it is an orientation of a simply-laced Dynkin diagram of type $A$, $D$ or $E$. 
\begin{cor}\label{Cor:EquivalenceOrdersDynkin}
Let $\QQ$ be a Dynkin quiver and let $\A=\mathrm{k}\QQ$ be its path algebra. Then the three partial orders $\leq_{\Ext}$, $\degg$ and $\leq_{\Hom}$ are equivalent on $R(\A,V)$, for all $V$.
\end{cor}
\begin{proof}
Let $M, N\in R(\A,V)$ such that $M\degg N$. Let us prove that $M\leq_{\Ext} N$ by induction on their total dimension. If the dimension is zero, there is nothing to prove. Suppose that the dimension is greater than zero. Since $\QQ$ is Dynkin, its path algebra $\A$ is representation-directed (\cite[Lemma~IX.1.1 and Definition~IX.3.2]{ASS}). In particular, there exists a partial order $\preceq$ on $\ind(\A)$ defined by $L\preceq E$ if there is a sequence of non-zero non-isomorphisms $L=L(1)\rightarrow L(2)\rightarrow\cdots \rightarrow L(h)=E$ for some $L(i)\in\ind(\A)$. Let $L$ be a $\preceq$-minimal indecomposable direct summand of $N$. Then $[L,N]^1=0$. By upper semicontinuity, we have $[L,M]^1\leq [L,N]^1=0$ which implies, using the homological interpretation \eqref{Eq:HomologicalInterpretationEulerForm} of the Euler-Ringel-form for hereditary algebras, that $[L,N]=[L,M]$. By Theorem~\ref{Thm:Bongartz}, $L$ embeds into $M$. Let $Q$ be the generic quotient of $M$ by $L$. Then $Q\degg \overline{N}$ where $\overline{N}$ is the direct complement of $L$ in $N$. By induction we get $Q\leq_{\Ext} \overline{N}$. Since there is a short exact sequence $0\rightarrow L\rightarrow M\rightarrow Q\rightarrow 0$ we get $M\leq_{\Ext} L\oplus Q\leq_{\Ext} L\oplus \overline{N}=N$, as claimed. 
The proof of the implication $\xymatrix@1{\leq_{\Hom}\ar@{=>}[r]&\degg}$ is given in \cite[Proposition~3.2]{Bongartz} for representation-directed algebras and it does not seem to simplify when restricted to Dynkin quivers. 
\end{proof}

\section{Quivers of type A}\label{Sec:RepTheoryTypeA}
\noindent
In this section we recall with proofs some useful features of the representation theory and the Auslander-Reiten theory of quivers of type $A$.  In the last subsection, Proposition~\ref{Prop:GenQuot} contains a description of the generic quotients by indecomposables which turns out to be very useful in the proof of Theorem~\ref{Thm:MainThm}.  Throughout the section, $\QQ$ denotes a quiver of type $A_n$ and $\A$ denotes its complex path algebra. 
\subsection{Representation Theory}
We denote the category of finite-dimensional complex $\QQ$-representations by $\Rep(\QQ)$ . 
The category $\Rep(\QQ)$ is abelian and Krull-Schmidt.
The set $\ind(\QQ$) of isoclasses of indecomposable $\QQ$--representations is finite by Gabriel's Theorem and the dimension vector induces a bijection between $\ind(\QQ)$ and the positive roots of the root system of type $A_n$. We denote by $\alpha_i$, $i=1,\cdots, n$ the simple roots. A positive root $\alpha_{i,j}=\sum_{k=i}^j\alpha_k$,  $1\leq i\leq j\leq n$, is the dimension vector of the indecomposable $\QQ$-representation that we denote by $U_{i,j}$, here $(U_{i,j})_k=\CC$ for $i\leq k\leq j$ and every linear map between two non-zero vector spaces is an isomorphism.  Given a vertex $j\in\QQ_0$, we denote by $S_j=U_{j,j}$ the corresponding simple representation, by $P_j$ its projective cover, and by $I_j$ its injective envelope. More concretely, $P_j=U_{t_1,t_2}$ and $I_j=U_{s_1,s_2}$ where $s_1,s_2,t_1,t_2$ are the vertices (possibly equal to $j$) such that there are maximal directed paths $\xymatrix@1{j\ar@{..>}[r]&t_k}$ and $\xymatrix@1{s_k\ar@{..>}[r]&j}$ in $\QQ$ (i.e. $t_k$ is either equal to $j$ or it is a sink of $\QQ$ and $s_k$ is either equal to $j$ or it is a source of $\QQ$), for $k=1,2$. The \emph{radical} of $P_i$, i.e. the kernel of the map $\xymatrix@1{P_i\ar@{->>}[r]&S_i}$, is the direct sum of at most two indecomposable projectives $\rad(P_i)=P_{j_1}\oplus P_{j_2}$ where $j_1$ and $j_2$ are the terminal vertices of the arrows of $\QQ$ starting from $i$ (there can be at most two).

An indecomposable $\QQ$--representation $L$ is \emph{thin}, meaning that $\dim L_j=[P_j,L]=[L,I_j]\leq 1$ for every vertex $j\in \QQ_0$. 

Every $M\in\Rep(\QQ)$ admits a canonical projective resolution, or standard resolution (see e.g. \cite[page~7]{CB})
$$
\xymatrix{
0\ar[r]&\bigoplus_{\alpha:i\rightarrow j\in \QQ_1}P_j^{d_i}\ar[r]&\bigoplus_{i\in\QQ_0}P_i^{d_i}\ar[r]&M\ar[r]&0
}
$$
where $\mathbf{d}=\mathbf{dim}\,M$ is the dimension vector of $M$.  As a consequence, Rep($\QQ$) is hereditary, i.e. $\Ext^i(-,-)=0$ for $i\geq2$ or equivalently the subrepresentations of a projective are projective. Moreover, for every $M,N\in\Rep(\QQ)$ the following formula holds (apply $\Hom(-,N)$ to the standard resolution of $M$):
\begin{equation}\label{Eq:HomologicalInterpretationEulerForm}
[M,N]-[M,N]^1=\langle\mathbf{dim}\,M,\mathbf{dim}\,N\rangle_\QQ.
\end{equation}
Here,  for $\mathbf{e},\mathbf{d}\in\ZZ^{\QQ_0}$ the \emph{Ringel-Euler form} $\langle-,-\rangle_\QQ:\ZZ^{\QQ_0}\times\ZZ^{\QQ_0}\rightarrow\ZZ$ is defined as
$
\langle\mathbf{e},\mathbf{d}\rangle_\QQ=\sum_{i\in\QQ_0}e_id_i-\sum_{\alpha:i\rightarrow j\in\QQ_1}e_id_j.
$
We notice that
\begin{equation}\label{Eq:KthEntryDimVectEulerForm}
d_k=\langle\mathbf{d},\dimv I_k\rangle_\QQ=\langle\dimv P_k,\mathbf{d}\rangle_\QQ.
\end{equation}

Since $\ind(\QQ)$ is finite, for every dimension vector $\mathbf{d}$, the representation variety  $R_\mathbf{d}$ contains a dense $\G^\bullet_\mathbf{d}$-orbit. A point of this orbit is called a \emph{generic} $\QQ$-representation of dimension vector $\mathbf{d}$. Let $M$ be such a point. Then the dimension of its orbit equals the dimension of $R_\mathbf{d}$ which is $\sum_{i\rightarrow j\in\QQ_1}d_id_j$; the dimension of this orbit is also equal to $\dim \G^\bullet_\mathbf{d}-\dim \Aut(M)=\sum_{i\in\QQ_0}d_i^2-[M,M]$. Thus, by \eqref{Eq:HomologicalInterpretationEulerForm} we see that $M$ is a generic representation if and only if $[M,M]^1=0$. A $\QQ$-representation  is called \emph{rigid} if $[M,M]^1=0$. Thus $M$ is generic if and only if it is rigid. 

The \emph{Auslander-Reiten translate} $\tau\simeq D\Ext^1(-, A)$ is a  functor in $\Rep(\QQ)$, where $D=\Hom(-,\mathrm{k})$ denotes the standard $\mathrm{k}$-duality. It has a quasi-inverse  $\tau^-\simeq \Ext^1(D(-), \A)$ (see \cite[Section~6]{CB}). In particular, $\tau P$ and $\tau^-I$ are zero for every projective $P$ and every injective $I$. For $L\in\ind(\QQ)$ not projective, $\tau L$ is indecomposable and  $\tau^-\tau L\simeq L$. For $E\in\ind(\QQ)$ not injective,  $\tau^-E$ is  indecomposable and $\tau\tau^- E\simeq E$. In \cite[Section~3.1.2]{Schiffler}  several methods to compute $\tau L$ and $\tau^-L$ are collected. The following Auslander-Reiten formulas hold (see e.g \cite[Lemma~6.1]{CB}): 
\begin{equation}\label{Eq:ARFormula}
[L,E]^1=[E,\tau L]=[\tau^- E,L].
\end{equation}
For every $L\in\ind(\QQ)$ there  exist two vertices $i,j\in \QQ_0$ and two non-negative integers $s,t\geq0$ such that $L=\tau^{-s}P_i=\tau^t I_j$. 

A \emph{path} in Rep($\QQ$) is a sequence 
\begin{equation}\label{Eq:PathInAmod}
\xymatrix{
L_1\ar^{f_1}[r]&L_2\ar^{f_2}[r]&\cdots\ar^{f_k}[r]&L_{k+1}
}
\end{equation}
of non-zero non-isomorphisms between the indecomposable $\QQ$--representations $L_1$, $L_2$, $\cdots$, $L_{k+1}$ (notice that the composition of these maps can be zero). A path \eqref{Eq:PathInAmod} is cyclic if $L_1=L_k$. It is well-known that $\Rep(\QQ)$  is \emph{representation-directed} (see (\cite[Lemma~IX.1.1 and Definition~IX.3.2]{ASS}), i.e. it does not contain cyclic paths. 
\begin{lem}\label{Lemma:OneDimensional}
Let $L,E\in\ind(\QQ)$. Then 
\begin{enumerate}
\item $[L,E]\leq 1$ and $[L,E]^1\leq1$
\item  $[L,E][L,E]^1=0$.
\end{enumerate}
\end{lem}
\begin{proof}
Let $s\geq0$ and $i\in\QQ_0$ such that $L=\tau^{-s}P_i$.
For (i), we notice that $[L,E]=[P_i,\tau^{s}E]\leq 1$, since $\tau^s E$ is indecomposable and hence thin. The second inequality follows from the  Auslander formulas \eqref{Eq:ARFormula}.

To prove (ii) suppose that $[L,E][L,E]^1\neq 0$ and let $0\rightarrow E\rightarrow F\rightarrow L\rightarrow 0$ be a non-split short exact sequence. Since $[L,E]\neq0$,  we get a cyclic path $L\rightarrow E\rightarrow F_i\rightarrow L$ (where $F_i$ is an indecomposable direct summand of $F$) contradicting the fact that $\Rep(\QQ)$ is representation-directed.
\end{proof}
\begin{lem}\label{Lemma:ImageIndecomposable}
Let $L,E\in\ind(\QQ)$ and let $f:L\rightarrow E$ be a non-zero homomorphism. Then the image of $f$ is indecomposable.
\end{lem}
\begin{proof}
Apply $\Hom(L,-)$ to the short exact sequence 
$\xymatrix@1@C=10pt{0\ar[r]&\textrm{Im}(f)\ar[r]&E\ar[r]&\textrm{Coker}(f)\ar[r]&0}$ and get $[L,\textrm{Im}(f)]=[L,E]=1$. Then apply $\Hom(-,\im(f))$ to the short exact sequence $\xymatrix@1@C=10pt{0\ar[r]&\textrm{Ker}(f)\ar[r]&L\ar[r]&\im(f)\ar[r]&0}$ and get $[\im(f),\im(f)]=1$, proving that $\im(f)$ is indecomposable. 
\end{proof}
\begin{rem}
Lemma~\ref{Lemma:ImageIndecomposable} is not true for other Dynkin quivers.
\end{rem}

\begin{lem}\label{Lem:MiddleTerm}
Let $L,E\in\ind(\QQ)$ such that $[E,L]^1\neq 0$. Let $$\xymatrix{\xi:0\ar[r]& L\ar^\iota[r]& F\ar^p[r] &E\ar[r]& 0}$$ be a non-split short exact sequence. Then the following properties hold for $F$:
\begin{enumerate}
\item $F$ is rigid and it has at most two indecomposable direct summands.
\item $F$ is indecomposable if and only if $[L,E]=0$.
\item $F$ is decomposable if and only if $[L,E]=1$. In this case $[F,F]=2$.
\item $[F,L]^1=[E,F]^1=0$
\end{enumerate} 

\end{lem}
\begin{proof}
Since $[L,L]^1=0$, by applying the covariant functor $\Hom(L,-)$ to $\xi$ we  get the short exact sequence
\begin{equation}\label{Eq:LemmaSequence}
0\rightarrow \Hom(L,L)\rightarrow \Hom(L,F)\rightarrow \Hom(L,E)\rightarrow 0
\end{equation}
and hence $[L,F]=[L,L]+[L,E]=1+[L,E]\leq 2$ (by Lemma~\ref{Lemma:OneDimensional}). Since $\xi$ is not split, and $L$ and $E$ are indecomposables, we see that necessarily $[L,F_i]\neq0$ and $[F_i,E]\neq0$ for every indecomposable direct summand $F_i$ of $F$. Thus $F$ has at most two indecomposable direct summands. Moreover, $F$ is indecomposable if and only if $[L,F]=1$ if and only if $[L,E]=0$. If $F$ is indecomposable, it is rigid.

If $F$ is not indecomposable, then $[L,F]=2$, $[L,E]=1$ and $F$ is the direct sum of two indecomposables, say $F_1$ and $F_2$.  
Let us show that $F_1\not\simeq F_2$, $[F_1,F_2]=[F_2,F_1]=0$ and $[F,F]^1=0$. Since $[L,F_1]=1$, then $[F_1,L]=0$ because $F_1\not\simeq L$ and $\Rep(\QQ)$ is representation-directed. Thus, when we apply $\Hom(F_1,-)$ to $\xi$ we get $[F_1,F]\leq [F_1,E]=1$ and hence $[F_1,F]=[F_1,F_1]=1$ and $[F_1,F_2]=0$. Moreover,  $[F_1,F]^1=[F_1,L]^1+[F_1,E]^1$. Apply $\Hom(-,L)$ to $\xi$ and get the short exact sequence 
$$
0\rightarrow \Hom(L,L)\rightarrow \Ext^1(E,L)\rightarrow \Ext^1(F,L)\rightarrow 0.
$$
Since $[L,L]=[E,L]^1=1$, we get $[F,L]^1=0$. Since $[F_1,E]=[F_2,E]=1$ then $[F,E]^1=0$ by Lemma~\ref{Lemma:OneDimensional}. We conclude that $F$ is rigid. To get $[E,F]^1=0$ apply $\Hom(E,-)$ to $\xi$ and conclude as above. 
\end{proof}
\noindent
\subsection{Auslander-Reiten Theory} In this subsection we recall the definition of the Auslander-Reiten quiver of $\QQ$ and how to use it to get a basis of the Hom and Ext-space between two indecomposable $\QQ$-representations.

Let $L$ be an indecomposable non-projective $\QQ$--representation. Since $[L,\tau L]^1=[\tau L,\tau L]=[L,L]=1$, there exists a non-split short exact sequence 
\begin{equation}\label{Eq:ARSequenceEndingL}
\xymatrix{
\xi:0\ar[r]&\tau L \ar^f[r]& X\ar^g[r]&L\ar[r]&0.}
\end{equation}
This sequence has the remarkable property that it is \emph{almost split}, or an \emph{AR-sequence}; this means that 1) it is not split, 2) for every $E\in\ind(\QQ)$, $E\not\simeq L$,  the induced map $g_\ast:\Hom(E,X)\rightarrow \Hom(E,L)$ is surjective and 3) for every $E'\in\ind(\QQ)$, $E'\not\simeq \tau L$, the induced map $f^\ast:\Hom(X,E')\rightarrow \Hom(\tau L, E')$ is surjective. These \emph{almost split properties} follow at once from Lemma~\ref{Lemma:OneDimensional}; indeed,  by the AR-formulas \eqref{Eq:ARFormula}, $[E,\tau L]^1=[\tau L,\tau E]$ and $[L,E']^1=[E',\tau L]$.
Dually, for every non-injective $L\in\ind(\QQ)$ there exists an almost split sequence starting from $L$: 
\begin{equation}\label{Eq:ARSequenceStartingL}
\xymatrix{
\xi:0\ar[r]&L \ar^f[r]& X\ar^g[r]&\tau^-L\ar[r]&0.}
\end{equation}
The \emph{Auslander-Reiten quiver} or \emph{AR-quiver} of Rep($\QQ$)  is a translation quiver \cite{Ringel} denoted by $\Gamma_\QQ$ and defined as follows: It has $\ind(\QQ)$ as set of vertices; thus a vertex of $\Gamma_\QQ$ denotes an isoclass of an indecomposable $\QQ$-representation; with abuse of notation, we denote  a vertex of $\Gamma_\QQ$ and an element of the corresponding isoclass with the same symbol. The arrows of $\Gamma_\QQ$ are defined as follows: given $X_i, L\in\ind(\QQ)$ there is an arrow $X_i\rightarrow L$ in $\Gamma_\QQ$ if $X_i$ is a direct summand of the middle term of the almost sequence ending in $L$ in case $L$ is non-projective, or, if $L$ is projective, $X_i$ is a direct summand of the radical of $L$. Thus, an arrow $X_i\rightarrow L$ corresponds to a homomorphism from $X_i$ to $L$ which by Lemma~\ref{Lemma:OneDimensional} is unique, up to non-zero-scalar multiples. Finally, the translation is given by the functor $\tau$; it sends a non-projective vertex $L$ to the non-injective vertex $\tau L$ and an arrow $f:E\rightarrow L$ between non-projectives  to the arrow $\tau f:\tau E\rightarrow \tau L$ between non-injectives. 

We always use the standard convention to depict $\Gamma_\QQ$ so that the translation is represented by a horizontal dashed arrow which points from right to left, and the arrows of $\Gamma_\QQ$ either point towards the north-east or the south-east.

There are several methods to compute $\Gamma_\QQ$ (see e.g. \cite[Section~3.1]{Schiffler}). One is the so-called \emph{knitting algorithm} which works as follows: the subquiver of $\Gamma_\QQ$ supported on the indecomposable projectives is a copy of $\QQ^{op}$, with $P_i$ put at vertex $i$; starting from this one inductively \emph{knits} by using the formula $\mathbf{dim}\, \tau^-L=\mathbf{dim}\,X-\mathbf{dim}\,L$ to get the vertex $\tau^- L$ knowing $X$ and $L$; the algorithm stops when $\mathbf{dim}\,X-\mathbf{dim}\,L$ is not a dimension vector. We choose to draw $P_1$ at the top in order to fix a unique way to depict $\Gamma_\QQ$. Let us sketch three examples.

\begin{example}\label{Ex:ARQuiverA3} Let $\xymatrix@1{\QQ=\stackrel{\rightarrow}{A_3}=1\ar[r]&2\ar[r]&3}$. Then $\Gamma_\QQ$ is
$$
\
\xymatrix@!R=8pt@!C=8pt{
 &                         &P_1=I_3\ar[dr]& \\
                     &P_2\ar[ur]\ar[dr]&                       &I_2\ar[dr]\ar@{-->}_\tau[ll]&\\
P_3=S_3\ar[ur]&                         &S_2\ar[ur]\ar@{-->}_\tau[ll]&     &I_1=S_1\ar@{-->}_\tau[ll]
}
$$
\end{example}
\begin{example}\label{Ex:ARQuiverAEvenAltern}
Let $\QQ$ be $\vcenter{\xymatrix@R=5pt@C=5pt{1\ar[dr]&&3\ar[dr]\ar[dl]&\\&2&&4\\}}$. Its AR-quiver $\Gamma_\QQ$ is given by
$$
\xymatrix@!R=10pt@!C=10pt{
&P_1=U_{1,2}\ar[dr]&&U_{3,4}=I_4\ar@{-->}_(.45)\tau[ll]\ar[dr]&\\
P_2=U_{2,2}\ar[dr]\ar[ur]&&U_{1,4}\ar@{-->}_(.45)\tau[ll]\ar[dr]\ar[ur]&&U_{3,3}=I_3\ar@{-->}_\tau[ll]\\
&P_3=U_{2,4}\ar[dr]\ar[ur]&&U_{1,3}=I_2\ar@{-->}_(.45)\tau[ll]\ar[dr]\ar[ur]&\\
P_4=U_{4,4}\ar[ur]&&U_{2,3}\ar@{-->}_(.45)\tau[ll]\ar[ur]&&U_{1,1}=I_1\ar@{-->}_\tau[ll]
}
$$
\end{example}

\begin{example}\label{Ex:ARQuiverAOddAltern}
Let $\QQ$ be $\vcenter{\xymatrix@R=5pt@C=5pt{&2\ar[dl]\ar[dr]&&&\\1&&3\ar[dr]&&5\ar[dl]\\&&&4&}}$. Its AR-quiver $\Gamma_\QQ$ is given by
$$\xymatrix@!R=6pt@!C=6pt{
&P_1=U_{1,1}\ar[dr]&&U_{2,4}\ar[dr]\ar@{-->}_(.4)\tau[ll]&&U_{5,5}=I_5\ar@{-->}_\tau[ll]&\\
&&P_2=U_{1,4}\ar[dr]\ar[ur]&&U_{2,5}=I_4\ar[dr]\ar[ur]\ar@{-->}_(.45)\tau[ll]&&\\
&P_3=U_{3,4}\ar[dr]\ar[ur]&&U_{1,5}\ar[dr]\ar[ur]\ar@{-->}_(.4)\tau[ll]&&U_{2,3}=I_3\ar[dr]\ar@{-->}_\tau[ll]&\\
P_4=U_{4,4}\ar[dr]\ar[ur]&&U_{3,5}\ar[dr]\ar[ur]\ar@{-->}_(.4)\tau[ll]&&U_{1,3}\ar[dr]\ar[ur]\ar@{-->}_\tau[ll]&&U_{2,2}=I_2\ar@{-->}_(.55)\tau[ll]\\
&P_5=U_{4,5}\ar[ur]&&U_{3,3}\ar[ur]\ar@{-->}_(.4)\tau[ll]&&U_{1,2}=I_1\ar[ur]\ar@{-->}_(.55)\tau[ll]&\\
}
$$
\end{example}
\subsubsection*{Paths and homomorphisms} 
A \emph{path} of \emph{length} $k$ in $\Gamma_\QQ$ is a directed path 
\begin{equation}\label{Eq:PathInGamma}
\xymatrix{
p:L_1\ar^{f_1}[r]&L_2\ar^{f_2}[r]&\cdots\ar^{f_k}[r]&L_{k+1}
}
\end{equation}
of the quiver $\Gamma_\QQ$. For $k=0$ this is the trivial path from $L_1$ to itself. We identify this path with the path of $\Rep(\QQ)$ obtained by choosing a representative of each arrow $f_i$. If $L_1$ is non-projective, then the translation $\tau$ induces  the \emph{translated path} $\tau p:\xymatrix@C=40pt{
\tau L_1\ar|{(\tau f_1)}[r]&\tau L_2\ar|{(\tau f_2)}[r]&\cdots\ar|(.4){(\tau f_k)}[r]&\tau L_{k+1}.
}
$

We say that a path $p$ in $\Gamma_\QQ$ as \eqref{Eq:PathInGamma} \emph{factors through an almost split sequence} if $k\geq 2$ and there exists an index $1\leq i\leq k-1$ such that $L_i=\tau L_{i+2}$.

A \emph{zero-relation}  is a path in $\Gamma_\QQ$ that factors through an almost split sequence whose middle term is indecomposable. By Lemma~\ref{Lem:MiddleTerm}, a zero-relation is zero.

\begin{lem}\label{Lem:ZeroRelations}
Let $E,L\in\ind(\QQ)$ and let $p$ be a path from $E$ to $L$  in $\Gamma_\QQ$ which factors through an almost split sequence. Then $L$ is not projective and there exists a path from $E$ to  $\tau L$. In particular,  there exists a path from $E$ to every direct  summand of the middle term of the almost split sequence ending in $L$.
\end{lem}
\begin{proof}
To fix notation, we assume that $p$ is given by  \eqref{Eq:PathInGamma} with $E=L_1$ and $L=L_{k+1}$ and let $i$ be an index such that $L_i=\tau L_{i+2}$. Then $L_{i+2}$ is not projective and there is an induced path from $L_{i+2}$ to $L$. Thus there is the translated path from $L_i=\tau L_{i+2}$ to $\tau L$ and hence from $E$ to $\tau L$. The statement is now obvious.
\end{proof}
\noindent
We say that two vertices $E$ and $L$ of $\Gamma_\QQ$ have \emph{distance} $\textrm{dist}(E,L):=d$ if there exists a path in $\Gamma_\QQ$ from $E$ to $L$ of length $d$. Since every path from $\tau L$ to $L$ has length two, we see that this notion is well-defined.

\begin{prop}\label{Prop:PathBasisHom}
Let $E, L\in\ind(\QQ)$. Then $[E,L]=0$ if and only if either there are no paths in $\Gamma_\QQ$ from $E$ to $L$ or there exists a zero-relation from $E$ to $L$. If $[E,L]\neq0$ then every path from $E$ to $L$ is a basis of $\Hom(E,L)$. 
\end{prop}
\begin{proof}
We fix $E\in\ind(\QQ)$ and let $L$ vary in $\ind(\QQ)$. If $L\simeq E$ then $[E,L]=1$ and there is a unique (trivial) path from $E$ to $L$. Let us assume that $L\not\simeq E$. For $L\in\ind(\QQ)$ let $d_L=\textrm{max}\{\textrm{dist}(P_i,L)|\, i\in\QQ_0\}$ be the maximal distance from an indecomposable projective to $L$. We proceed by induction on $d_L\geq0$. For simplicity, in this proof, by a path we mean a path in $\Gamma_\QQ$.

Suppose, first, that $L=P_i$ is projective (this case includes the base of  induction).  Then the only paths ending in $L$ start from a projective. Thus if $E$ is not projective, $[E,L]=0$ and there are no paths from $E$ to $L$. If $E=P_j$ is projective, then $[E,L]=1$ if and only if the vertices $i$ and $j$ of $\QQ$ are joined by a directed path $\xymatrix@1{i\ar@{..>}[r]&j}$ of $\QQ$. In this case, there is a unique path from $P_j$ to $P_i$ and by Lemma~\ref{Lem:ZeroRelations} it is not a zero-relation. This proves the lemma in the case when $L$ is projective. 

Let us assume that $L\in\ind(\QQ)$ is not projective, in particular $d_L\geq 1$.  Let us consider the almost split sequence $\xi$ \eqref{Eq:ARSequenceEndingL} ending in $L$. In view of Lemma~\ref{Lem:MiddleTerm} either its middle term $X$ is indecomposable or it has two non-isomorphic direct summands $X=X_1\oplus X_2$. For simplicity, by a path from $E$ to $X$ we mean a path from $E$ to an indecomposable direct summand of $X$. We notice that $d_{\tau L}<d_L$ and $d_{X_{i}}<d_L$. Since $E\not\simeq L$, the almost split properties give
\begin{equation}\label{Eq:AlmostSplitProofLem}
[E,L]=[E,X]-[E,\tau L].
\end{equation}  
Moreover, again because $E\not\simeq L$, every path from $E$ to $L$ factors through $X$. Clearly,  every path from $E$ to $X$ gives rise to a path from $E$ to $L$. Thus, there are no paths from $E$ to $L$ if and only if there are no paths from $E$ to $X$. In this case, by induction,  $[E,X]=0$ which by \eqref{Eq:AlmostSplitProofLem} implies $[E,L]=0$. 

We assume from now on that there exists a path from $E$ to $L$. 

Let  $p$ be a zero-relation from $E$ to $L$ and let us show that $[E,L]=0$. Let $X_1$ be the indecomposable direct summand of $X$ such that $p$ factors through $X_1$. If $X=X_1$ then $[E,X]\leq 1$, if not then $[E,X_1]=0$ by induction, since $\xi$ is not a zero-relation. In all cases we have  that $[E,\tau L]\leq [E,X]\leq 1$. By Lemma~\ref{Lem:ZeroRelations} there exists a path from $E$ to $\tau L$. If $[E,\tau L]=0$ then by induction there exists a zero-relation from $E$ to $\tau L$; thus there exists a zero-relation from $E$ to every direct summand of $X$, and by induction $[E,X]=0$. By \eqref{Eq:AlmostSplitProofLem}, $[E,L]=0$. If $[E,\tau L]=1$, then $[E,X]=[E,\tau L]=1$ and by \eqref{Eq:AlmostSplitProofLem}, $[E,L]=0$. 

If $[E,L]=0$ then by \eqref{Eq:AlmostSplitProofLem}, $[E,X]=[E,\tau L]\in\{0,1\}$. If $[E,X]=[E,\tau L]=0$ then by induction there is a zero-relation from $E$ to $X$, and hence there is a zero-relation from $E$ to $L$, too. If $[E,X]=[E,\tau L]=1$ then, by induction, there exists a path $p$ from $E$ to $\tau L$. Moreover, either $X$ is indecomposable or $[E,X_1]=1$ and $[E,X_2]=0$. In the first case, $\xi$ is a zero-relation and by composing with $p$,  there exists a zero-relation from $E$ to $L$. In the second case, by composing with $p$, we see that there exists a path from $E$ to $X_2$, and since $[E,X_2]=0$ there exists a zero-relation from $E$ to $X_2$ and hence there exists a zero-relation from $E$ to $L$, too.

If $[E,L]=1$ then by \eqref{Eq:AlmostSplitProofLem} either $[E,X]=2$ and $[E,\tau L]=1$ or $[E,X]=1$ and $[E,\tau L]=0$. In the first case every path from $E$ to $X$ is not a zero-relation and hence the same holds for every path from $E$ to $L$, since $\xi$ is not a zero-relation. In the second case,  either $X=X_1$ is indecomposable or $[E,X_1]=1$ and $[E,X_2]=0$. Then there are no paths from $E$ to $\tau L$;  indeed, a zero-relation from $E$ to $\tau L$ induces  a zero-relation from $E$ to $X_1$, contradicting $[E,X_1]=1$; in particular, by Lemma~\ref{Lem:ZeroRelations}, there are no paths from $E$ to $X_2$. Thus, every path from $E$ to $L$ factors through $X_1$, does not factor through $\tau L$ and therefore it is not a zero-relation. 
\end{proof}

\begin{lem}\label{Lem:Convexity}
Let $L_1, L_2, L_3\in\ind(\QQ)$ such that $[L_1,L_2]=[L_2,L_3]=[L_1,L_3]=1$. Then the composition map 
$$
\Hom(L_1,L_2)\times\Hom(L_2,L_3)\rightarrow \Hom(L_1,L_3):\;(f,g)\mapsto g\circ f
$$
is surjective. 
\end{lem}
\begin{proof}
Let $p_{21}$ be a path from $L_1$ to $L_2$ and let $p_{32}$ be a path from $L_2$ to $L_3$ in $\Gamma_\QQ$. By Proposition~\ref{Prop:PathBasisHom}, $p_{21}$ is a basis of  $\Hom(L_1,L_2)$ and $p_{32}$ is a basis of  $\Hom(L_2,L_3)$. The composite map $p=p_{32}\circ p_{21}$ is a path from $L_1$ to $L_3$ and again by Proposition~\ref{Prop:PathBasisHom} is a basis of $\Hom(L_1,L_3)$.
\end{proof}
\begin{cor}\label{Cor:Convexity}
Let $L, E\in\ind(\QQ)$, such that $[L,E]=1$. Let $s\geq 1$ be an integer such that $[L,\tau^{-s}E]=1$. Then $[E,\tau^{-s}E]=1$ and the composition map 
$$
\Hom(L,E)\times\Hom(E,\tau^{-s}E)\rightarrow \Hom(L,\tau^{-s}E):\;(f,g)\mapsto g\circ f
$$
is surjective. 
\end{cor}
\begin{proof}
Let $q$ be a path from $L$ to $E$ in $\Gamma_\QQ$. Since $\tau^{-s}E$ is on the right of $E$ in $\Gamma_\QQ$ there exists a path $p$ from $E$ to $\tau^{-s}E$ in $\Gamma_\QQ$. Then  $p\circ q$ is a path from $L$ to $\tau^{-s}E$. If $[E,\tau^{-s}E]=0$, then by Proposition~\ref{Prop:PathBasisHom}, there exists a zero-relation $p'$ from $E$ to $\tau^{-s} E$. Then, the path  $p'\circ q$ would be a a zero-relation, too, contradicting the hypothesis $[L,\tau^{-s}E]=1$. The statement is then a consequence of Lemma~\ref{Lem:Convexity}.
\end{proof}
\begin{rem}
Corollary~\ref{Cor:Convexity} is not true for other Dynkin quivers (e.g. \cite[Page~94]{Schiffler}).
\end{rem}

For $L\in\ind(\QQ)$ the function $\dim\Hom(L,-):\ind(\QQ)\rightarrow\ZZ_{\geq0}$  is called a \emph{hammock function} in the terminology of \cite{Brenner}. The \emph{support} of the hammock function $\dim\Hom(L,-)$ (or for brevity of $\Hom(L,-)$) is the set of all $E\in\ind(\QQ)$ such that $[L,E]\neq0$; it is called a \emph{hammock}. By using Proposition~\ref{Prop:PathBasisHom}, Lemma~\ref{Lem:Convexity} and Corollary~\ref{Cor:Convexity} one can show that these hammocks are ``convex''. They have a typical rectangular shape, which explains the terminology.   Dually, the same holds for the function $\textrm{dim}\Hom(-,L)$ and for its support. In \cite[Section~3.1.4]{Schiffler} there are many examples of these supports. We notice that this convexity property of the hammocks is not true for other Dynkin quivers (see e.g. \cite[Section~3.3.4.1]{Schiffler}).

\subsubsection*{The relation $\preceq$} Given $E,L\in\textrm{ind}(\QQ)$, we say that $E\preceq L$ if there exists a path in $\Gamma_\QQ$ which starts in $E$ and ends in $L$. Since $\Rep(\QQ)$ is representation-directed  the relation $\preceq$ is a partial order. If $[E,L]\neq0$, then $E\preceq L$ by Proposition~\ref{Prop:PathBasisHom}. The opposite is not necessarily true, since a path from $E$ to $L$ could be a zero-relation. 

\subsubsection*{Sectional paths} A \emph{sectional path} in $\Gamma_\QQ$ is a path in $\Gamma_\QQ$ which does not factor through an almost split sequence.  

\begin{lem}\label{Lem:SecPaths}
Let $L,E\in\ind(\QQ)$. The following are equivalent:
\begin{enumerate}
\item There is a sectional path from $L$ to $E$;
\item $[L,E]=1$ and there is a unique path from $L$ to $E$ in $\Gamma_\QQ$;
\item $[L,E]=1$ and $[E,L]^1=[L,\tau E]=0$.
\end{enumerate}
\end{lem}
\begin{proof}
The implication $\xymatrix@1@C=10pt{(i)\ar@{=>}[r]&(ii)}$ follows from Proposition~\ref{Prop:PathBasisHom} by induction on the length of a sectional path. To prove $\xymatrix@1@C=10pt{(ii)\ar@{=>}[r]&(iii)}$ let us suppose, for a contradiction, that $[L,\tau E]=[L,E]=1$; then by Corollary~\ref{Cor:Convexity}, $[\tau E,E]=1$ and thus by Lemma~\ref{Lem:MiddleTerm} the almost split sequence ending in $E$ has two indecomposable middle terms. It follows that there are at least two paths from $L$ to $E$.  To prove $\xymatrix@1@C=10pt{(iii)\ar@{=>}[r]&(i)}$ let $p$ be a path from $L$ to $E$ and suppose for a contradiction that $p$ factors through an almost split sequence. Then by Lemma~\ref{Lem:ZeroRelations} there is a path from $L$ to $\tau E$ and since $[L,\tau E]=0$, by Proposition~\ref{Prop:PathBasisHom} there exists a zero-relation from $L$ to $\tau E$. Then there exists a zero-relation from $L$ to $E$, too, against the hypothesis.
\end{proof}
A sectional path in $\Gamma_\QQ$ can be directed either towards the north-east or the south-east.  We call the former an \emph{$NE$ sectional path} and the latter an \emph{$SE$ sectional path}. Two consecutive sectional paths of $\Gamma_\QQ$ either form a sectional path (if they have the same direction) or they form a \emph{broken line} (if they have different directions).

In view of Lemma~\ref{Lem:SecPaths} (iii), if there is a sectional path from $L$ to $E$ then $[E,L]^1=0$. It then follows from the Happel-Ringel Lemma \cite[Lem.~4.1]{HR} that  a sectional path is either mono or epi. In particular if an arrow of a sectional path \eqref{Eq:PathInGamma} $f_i:L_i\rightarrow L_{i+1}$ is epi, then all the successive maps $f_{i+1},\cdots, f_k$ must be epi, too.  

\begin{lem}\label{Lem:SecEpMono}
The cokernel of a sectional mono is a sectional epi and the kernel of a sectional epi is a sectional mono. They form a broken line.

\end{lem}
\begin{proof}
Let $\xi:\xymatrix@C=20pt{0\ar[r]&L\ar|p[r]&E\ar|q[r]&C\ar[r]&0}$ be a short exact sequence such that $p$ is sectional mono. We apply $\Hom(L,-)$ to $\xi$ and get $[L,C]=0$, then we apply $\Hom(-,C)$  and get  $[C,C]=[E,C]$. We apply $\Hom(E,-)$ and get $[E,C]=[E,E]=1$. Thus $[C,C]=1$ and, hence, $C$ is indecomposable. By applying $\Hom(-,E)$, we get $[C,E]^1=0$. Thus $q$ is a sectional epi by Lemma~\ref{Lem:SecPaths}. If $q$ is sectional epi, by dual arguments one shows that $L$ is indecomposable and $p$ is sectional mono. The last statement is an immediate consequence of Lemma~\ref{Lem:SecPaths}.
\qedhere
\end{proof}
\begin{rem}
Lemma~\ref{Lem:SecEpMono} is not true for other Dynkin quivers. 
\end{rem}

The next lemma describes the position in $\Gamma_\QQ$ of the cokernel of a sectional mono and of the kernel of a sectional epi.
\begin{lem}\label{Lem:KerCokerSectional}
Let $p:\xymatrix{L\ar@{..>}[r]&E}$ be a sectional path. Then
\begin{enumerate}
\item $p$ is mono if and only if there exists a sectional path which starts from $E$ and contains a vertex that does not lie in the support of $\Hom(L,-)$. The cokernel of $p$ is the $\preceq$-minimal vertex of this sectional path that does not lie in the support of  $\Hom(L,-)$.
\item $p$ is epi if and only if there exists a sectional path which ends in $L$ and  contains a vertex that does not lie in the support of $\Hom(-,E)$. The kernel of $p$ is the $\preceq$-maximal vertex of this sectional path that does not lie in the support of  $\Hom(-,E)$.
\end{enumerate}
\end{lem}
\begin{proof}
We prove $(i)$. The proof of $(ii)$ is dual. If $p$ is mono then by Lemma~\ref{Lem:SecEpMono} there is a sectional path from $E$ to its cokernel $C$, and $[L,C]=0$. It follows that $C$ does not lie in the support of $\Hom(L,-)$. Since, $[L,\tau C]=[C,L]^1=1$, $\tau C$ lies in the support of $\Hom(L,-)$. Let $X$ be the middle term of the almost split sequence ending in $C$. Then, since $[L,C]=0$, we have $[L,X]=[L,\tau C]=1$. Thus there exists an indecomposable direct summand of $X$ which lies in the support of $\Hom(L,-)$. This shows that $C$ is $\preceq$-minimal among those vertices which do not belong to the support of $\Hom(L,-)$.  Viceversa if there exists a sectional path starting from $E$ which contains a vertex $D$ which do not belong to the support of $\Hom(L,-)$, let $C$ be the $\preceq$-minimal vertex of this sectional path with this property. Let $X$ be the middle term of the almost split sequence ending in $C$. Then $[L,\tau C]=[L,X]=1$. By the AR-formula, it follows that $[C,L]^1=1$. Since $[L,C]=0$, the middle term $E'$ of the non-split short exact sequence $\xi\in\Ext^1(C,L)$ is indecomposable and by Lemma~\ref{Lem:MiddleTerm}, $[L,E']=1=[E',C]$ and $[L,E']^1=0=[E',C]^1$. By Lemma~\ref{Lem:SecPaths}, there is a sectional path from $L$ to $E'$ and from $E'$ to $C$. It follows that $E'=E$. 
\end{proof}

\subsubsection*{Rectangles and extensions} 
Let us consider two subquivers of $\Gamma_\QQ$ of the form:

\begin{equation}\label{Rectangles}
\begin{array}{cc}
\vcenter{
\xymatrix@R=10pt@C=10pt{
&&F_1\ar@{..>}^{q_2}[dr]&\\
&&&E\\
L\ar@{..>}^{p_1}[uurr]\ar@{..>}_{p_2}[dr]&&&\\
&F_2\ar@{..>}_{q_1}[uurr]&&}}
&
\vcenter{
\xymatrix@R=10pt@C=10pt{
&&F\ar@{..>}^{q}[dr]&\\
&&&E\\
L\ar@{..>}^{p}[uurr]&&&}}
\end{array}
\end{equation}
where  $p_1$, $p_2$,  $q_1$, $q_2$, $p$ and $q$ are sectional paths, $q=\coker(p)$ and $p=\ker(q)$ (compare with Lemma~\ref{Lem:KerCokerSectional}). The subquiver on the left is called a  \emph{non-degenerate rectangle} of $\Gamma_\QQ$ and the one on the right is called a 
\emph{degenerate rectangle} of $
\Gamma_\QQ$. 

By a \emph{rectangle from $L$ to $E$} of $\Gamma_\QQ$ we mean either a non-degenerate or a degenerate rectangle shown in \eqref{Rectangles}. It is clear that there exists at most one such rectangle.

\begin{prop}\label{Prop:RectanglesBasisExt}
Let $E, L\in\ind(\QQ)$. Then $[E,L]^1=1$ if and only if there exists a  rectangle in $\Gamma_\QQ$ from $L$ to $E$.  This rectangle is degenerate if and only if $[L,E]=0$. 
\end{prop}
\begin{proof}
If $[E,L]^1=1$, let $\xi\in\Ext^1(E,L)$ be a basis and let $F$ be its middle term. Let $F_i$ be an indecomposable direct summand of $F$. By Lemmas~\ref{Lem:MiddleTerm} and \ref{Lem:SecPaths}, the non-zero morphisms $L\rightarrow F_i$ and $F_i\rightarrow E$ are sectional paths. Viceversa, we show that a rectangle gives rise to a short exact sequence. If the rectangle is degenerate, this is by definition. If there is a non-degenerate rectangle \eqref{Rectangles} from $L$ to $E$ in $\Gamma_\QQ$, then it does not contain zero-relations but it contains every path from $L$ to $E$ and it contains $\tau E$. By Proposition~\ref{Prop:PathBasisHom}, $[L,\tau E]=[L,E]=1$. Thus $[E,L]^1=1$. Let $\xi\in\Ext^1(E,L)$ be non-zero. By Lemma~\ref{Lem:MiddleTerm}, the middle term of $\xi$ has two non-isomorphic direct summands $F'_1$ and $F'_2$. By Lemmas~\ref{Lem:MiddleTerm} and \ref{Lem:SecPaths}, the morphisms $L\rightarrow F'_i$ and $F'_i\rightarrow E$ are sectional. It follows that $F_1'=F_1$ and $F_2'=F_2$.
\end{proof}
\noindent Proposition~\ref{Prop:RectanglesBasisExt} states that  the rectangles of $\Gamma_\QQ$ are cartesian. As a consequence we get the following useful properties of a non-degenerate rectangle \eqref{Rectangles}:
\begin{eqnarray}\label{Eq:rectangles}
&&\textrm{opposite sides are either both mono or both epi;
}\\
&& \ker(p_i)\simeq \ker(q_i) \textrm{ and } \coker(p_i)\simeq \coker(q_i).
\end{eqnarray}

\subsubsection*{Join and Meet}
Recall the partial order $\preceq$. Given $E,L\in\ind(\QQ)$ we denote 
$$
\begin{array}{c}
E\vee L=\sup\{E,L\}=\min\{F\in(\Gamma_\QQ)_0|\, E\preceq F \textrm{ and }L\preceq F\},\\
E\wedge L=\inf\{E,L\}=\max\{F\in(\Gamma_\QQ)_0|\, F\preceq E \textrm{ and }F\preceq L\}.
\end{array}
$$
Of course, $E\vee L$ and $E\wedge L$ might not exist. If $E\vee L$ exists then it is the intersection of two sectional paths, one starting from $E$ and one starting from $L$; dually, if $E\wedge L$ exists then it is the intersection of two sectional paths, one ending in $E$ and one ending in $L$. They both exist if and only if  there is a rectangle
$$
\xymatrix@C=10pt@R=10pt{
&&E\ar@{..>}[dr]&\\
&&&E\vee L\\
E\wedge L\ar@{..>}[uurr]\ar@{..>}[dr]&&&\\
&L\ar@{..>}[uurr]&&}
$$
in $\Gamma_\QQ$.

\subsection{Generic quotients}\label{Sec:GenericQuot} Let $\QQ$ be a quiver of type $A$. Let $L$ be an indecomposable $\QQ$-representation and $M$ be a $\QQ$-representation. In this section we give a necessary and sufficient condition for $L$ to embed into $M$ and we specify the position of the indecomposable direct summands of  the generic quotient of $M$ by $L$ in $\Gamma_\QQ$.

Suppose that $[L,M]\neq0$. Let $T^0_1,\cdots, T^0_r$ be the $\preceq$-minimal indecomposable direct summands (one for each isoclass) of $M$ which belong to the support of $\Hom(L,-)$.  We order them from top to bottom in $\Gamma_\QQ$ as depicted in Table~\ref{Table:Generic Quotient} for $r=4$.
Let $T^0=\oplus_{i=1}^rT^0_i$ and let $\widetilde{M}$ be the direct complement of $T^0$ in $M$ so that $M=T^0\oplus \widetilde{M}$. For every $i=1,\cdots, r-1$ let $T^1_i=T^0_i\vee T^0_{i+1}$. We also define $T^1_0$ and $T^1_r$ as follows: if there is an NE sectional path from $L$ to $T^0_1$ we define $T^1_0$ to be zero, otherwise $T^1_0$ is the $\preceq$-minimal vertex of the NE sectional path starting from $T^0_1$ not belonging to the support of $\Hom(L,-)$; similarly,  $T^1_{r}$ is zero if there is an SE sectional path from $L$ to $T^0_r$ and otherwise it is the $\preceq$-minimal vertex of the SE sectional path starting from $T^0_r$ not belonging to the support of $\Hom(L,-)$. (Compare with Table~\ref{Table:Generic Quotient}.)  Clearly, $T^1_i$ might not exist. In case they all exist, we put $T^1=\oplus_{i=0}^r T^1_i$. By construction, using Proposition~\ref{Prop:RectanglesBasisExt}, one  immediately sees that $T^0$ and $T^1$ are rigid (see Remark~\ref{Rem:Tilting}). For  $i=1,\cdots, r$ we denote by $L^{(i)}$ and $L_{(i)}$ the two indecomposables such that $L^{(i)}\wedge L_{(i)}=L$ and $L^{(i)}\vee L_{(i)}=T^0_i$. We place $L^{(1)},\cdots, L^{(r)}$ in the same NE sectional path starting from $L$. 
\begin{table}[hbtp]
\includegraphics[scale=0.8]{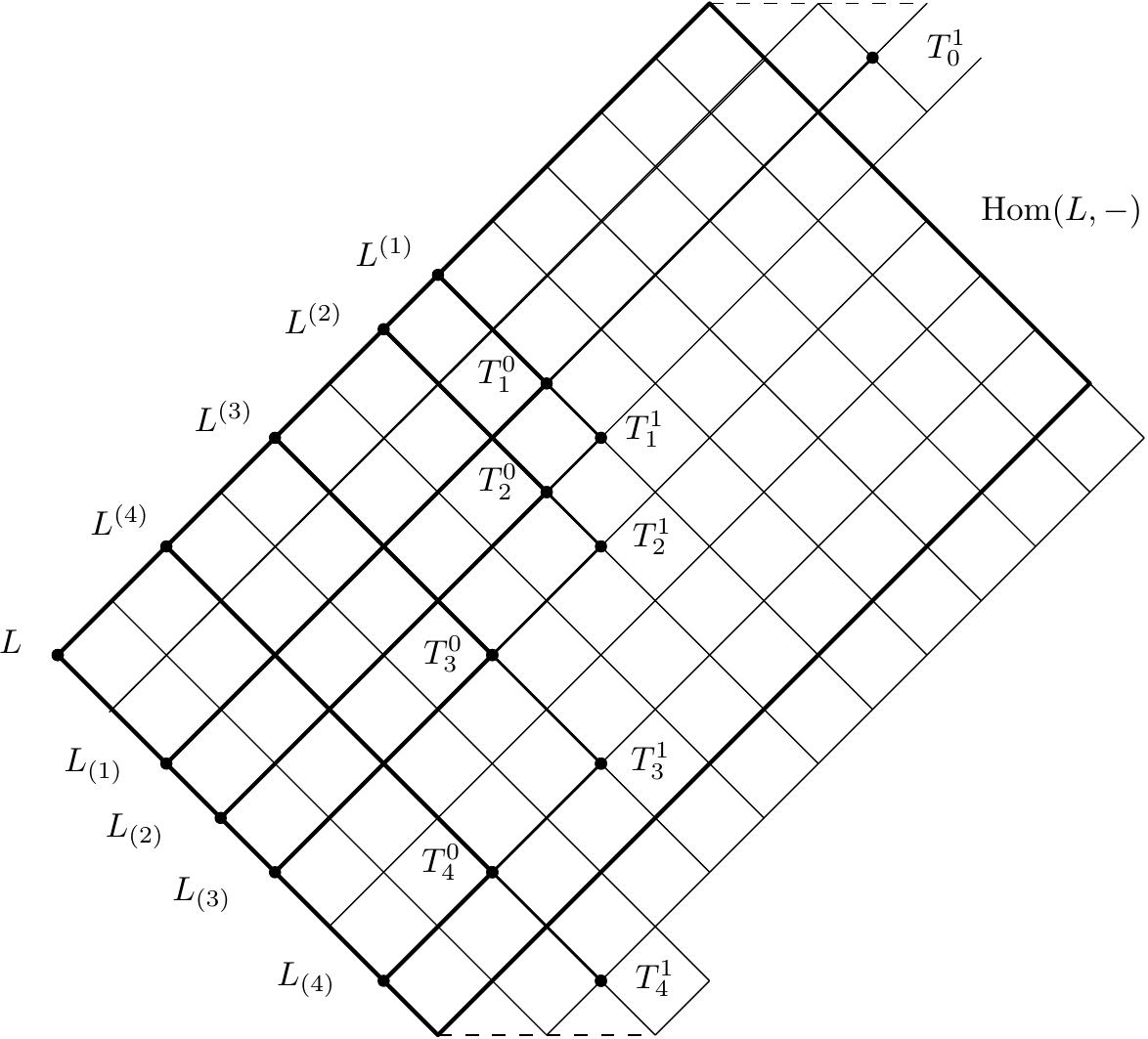}
\caption{The generic quotient of $M=T^0_1\oplus\cdots\oplus T^0_4$ by $L$ is $T^1_0\oplus\cdots \oplus T^1_4$. This picture represents a part of $\Gamma_\QQ$ containing the support of $\Hom(L,-)$. Every side of every square represents an arrow of $\Gamma_\QQ$ which points towards the north-east or the south-east.}
\label{Table:Generic Quotient}
\end{table}

\begin{prop}\label{Prop:GenQuot}
We retain all notations and objects introduced above. Then $L$ embeds into $M$ if and only if $T^1_i$ exists for every $i=0,\cdots, r$. In this case the generic quotient of $M$ by $L$ is $T^1\oplus\widetilde{M}$. 
\end{prop} 

\begin{proof}
We denote by $M_1,\cdots, M_N$ the indecomposable direct summands of $M$, ordered so that $M_i=T^0_i$ for $i=1,\cdots, r$. For a homomorphism  $f:L\rightarrow M$ we denote by $f_k:L\rightarrow M_k$ the induced homomorphism. Let $f:L\rightarrow M$ be a homomorphism such that $f_k\neq0$ whenever $[L,M_k]\neq0$. For every $i=1,\cdots, r$, let $\iota_i:=f_i:L\rightarrow T^0_i$ and set $\iota=(\iota_1,\cdots, \iota_r)^t:L\rightarrow T^0$. 
We claim that there exists an automorphism $\ell$ of $M$ such that $\ell\circ f=(\iota,0)^t$. Indeed, given two indices $i\neq j$ such that  $i\leq r$, $f_j\neq 0$ and $[M_i,M_j]=1$, i.e. either $M_i\simeq M_j$ or $M_j$ is not $\preceq$-minimal in the support of $\Hom(L,-)$, then there is a homomorphism $\ell_{ji}: M_i\rightarrow M_j$ and  by Lemma~\ref{Lemma:OneDimensional}, up to a non-zero scalar factor, $f_j=\ell_{ji}\circ f_i$. We then define the automorphism $\ell$ in the obvious way to get the claimed equality $(\ell\circ f)_i=\iota_i$ for $1\leq i\leq r$ and  $(\ell\circ f)_i=0$ for $i>r$. 

Since $f$ and $\ell\circ f$ have the same quotient, $f$ is a generic embedding of $L$ into $M$ if and only if $\iota$ is a generic embedding of $L$ into $T^0$ and in this case the generic quotient is the direct sum of $\widetilde{M}$ and the generic quotient of $T^0$ by $L$.

Since $T^1$ is rigid, it remains to prove that $\iota:L\rightarrow T^0$ is mono if and only if $T^1$ is defined, and in this case there exists a short exact sequence
\begin{equation}\label{Eq:SESLTilting}
\xymatrix{
0\ar[r]&L\ar^\iota[r]&T^0\ar^{\pi}[r]&T^1\ar[r]&0.
}
\end{equation}
We notice that, for $1\leq i\leq r-1$,  $T^1_i$ exists if and only if there exists a rectangle from $L_{(i)}$ to $T^1_i$. By Proposition~\ref{Prop:RectanglesBasisExt}, this happens if and only if there exists a short exact sequence $0\rightarrow L_{(i)}\rightarrow T^0_{i}\oplus L_{(i+1)}\rightarrow T^1_i\rightarrow 0$. Moreover, by Lemma~\ref{Lem:KerCokerSectional}, $T^1_0$ and $T^1_r$ exist if and only if the sectional paths $L\rightarrow  L_{(1)}$ and $L_{(r)}\rightarrow T^0_r$ are mono. We write 
$$
\dimv T^0-\dimv L=\sum_{i=0}^{r} \left(\dimv T^0_i+\dimv L_{(i+1)}-\dimv L_{(i)}\right)
$$
with the convention that $L_{(0)}=L$, $T^0_0$ and $L_{(r+1)}$ are zero. We show that $T^1$ is defined if and only if $\dimv T^0-\dimv L$ is a dimension vector. In this case, it equals $\dimv T^1$ by the discussion above. We put $\mathbf{t}_i:=\dimv T^0_i+\dimv L_{(i+1)}-\dimv L_{(i)}$.  

Let us prove that if there exists an index $j$ such that $T^1_j$ does not exist then $\dimv T^0-\dimv L$ is not a dimension vector. The argument we use is slightly different in the case when $j=0$, and we include this case in the general discussion by highlighting the differences in parentheses. The SE sectional path  starting from $T^0_j$ (or the $NE$ sectional path starting from $L_{(1)}$, if $j=0$)  terminates with an injective $I$. The almost split sequence ending in $I$ is not a zero-relation, because, otherwise, $I$ would not be in the support of $\Hom(L_{(j)},-)$  and $T^1_j$ would exist by Lemma~\ref{Lem:KerCokerSectional}. Thus, there exists an indecomposable injective  $I_k$ and  an NE arrow (or SE arrow for $j=0$) $I_k\rightarrow I$ in $\Gamma_\QQ$. Then, $[L_{(i)}, I_k]=1$ for every $0\leq i\leq j$, $[L_{(i)}, I_k]=0$ for every $i>j$ and $[T^0,I_k]=0$. By \eqref{Eq:KthEntryDimVectEulerForm} we get that the $k$-th component of $\mathbf{t}_i$ is
$$
(\mathbf{t}_i)_k=
\left\{
\begin{array}{cc}
-1&\textrm{if }i=j;\\
0&\textrm{otherwise.}
\end{array}
\right.
$$
We conclude that 
$\dimv T^0-\dimv L$ is not a dimension vector, and hence $L$ cannot embed into $T^0$ and a fortiori into $M$.

Suppose that $T^1$ is defined. Then $\dimv T^1=\dimv T^0-\dimv L$. Obviously, $L$ embeds into $L\oplus T^1$ and $[L,L\oplus T^1]^1=0$. Since $T^0$ is rigid, it is generic and hence $T^0\degg L\oplus T^1$. By Bongartz's Theorem~\ref{Thm:Bongartz} $L$ embeds into $T^0$ and since $T^1$ is rigid, $T^1$ is the generic quotient of $T^0$ by $L$, finishing the proof. 
\end{proof}
\begin{rem}\label{Rem:Tilting}
By construction both $T^0$ and $T^1$ are $\Hom$- and $\Ext$-orthogonal and $T^0\oplus T^1$ can be completed to a separating tilting object $T$ of $\Rep(\QQ)$ whose indecomposable direct summands form a section of $\Gamma_\QQ$ (see \cite{ASS} for the definitions). Then $L$ is a torsion-free module for $T$ and the short exact sequence \eqref{Eq:SESLTilting} is the $\textrm{add}\,T$ resolution of $L$. 
\end{rem}

\section{Symmetric quivers of type $A$}\label{Sec:SymmTypeA}
\noindent
Let $(\QQ,\sigma)$ be a symmetric quiver of type $A_n$. Let $\ee$ be $+1$ or $-1$. Let $(V,\Psi)$ be an $\ee$-quadratic space for $(\QQ,\sigma)$ of dimension vector $\mathbf{d}$. We denote the variety of $\QQ$-representations of dimension vector $\mathbf{d}$ either by $R(\A,V)$ or by $R_\mathbf{d}$ and the structure group by $\G_\mathbf{d}$. We denote by $R(\A,V)^{\Psi,\ee}\subset R(\A,V)$ the subvariety of $\ee$-representations. Sometimes, for simplicity of notation and in view of Remark~\ref{Rem:NoForm}, we use the shorter notation $R_\mathbf{d}^\ee\subset R_\mathbf{d}$ for the subvariety $R(\A,V)^{\Psi,\ee}$. By $\G_\mathbf{d}^\ee$ we denote the group of graded isometries. In this section we describe useful facts about $\ee$-representations of $(\QQ,\sigma)$. The main results are Propositions~\ref{prop:DeltaEven}, \ref{Prop:KeyLemma} and \ref{Prop:ESubquot}. They are crucial in our proof of Theorem~\ref{Thm:MainThm}. 
\subsection{Split and non-split types}
There are four possible types for the pair $(\QQ, \ee)$: $(A_{odd}, \pm 1)$ and $(A_{even}, \pm1)$. The main difference between the four types 
 arises in the structure of their indecomposable $\ee$--representations (see Section~\ref{Sec:IndecEpsilon}): only in types $(A_{odd}, 1)$ and $(A_{even}, -1)$, there are indecomposable $\ee$-representations of type $(I)$. In the other types $(A_{odd}, -1)$ and $(A_{even}, 1)$ the indecomposable $\ee$-representations are all of the form $L\oplus \nabla L$, for $L\in\ind(\QQ)$. This is due to the fact that the indecomposable $\QQ$-representations are thin and a symplectic vector space must be even dimensional, compare with \cite[Proposition~3.6 and 3.8]{DW}. For this reason we write that $(A_{odd}, -1)$ and $(A_{even}, 1)$ are the \emph{split types} while $(A_{odd}, 1)$ and $(A_{even}, -1)$ are the \emph{non-split types}. Accordingly, we say that $(\QQ, \ee)$ is of split or non-split type.
 \\[1ex]
 \noindent
In the next lemmas we use the notation introduced in \eqref{Eq:HomPMNabla}.

\begin{lem}\label{Lem:ImageERepresentation}
Let $M\in\Rep(\QQ)$ and let $f\in\Hom(M,\nabla M)^{\ee\nabla}$.
Then the image of $f$ is an $\ee$--representation.
\end{lem}
\begin{proof}
On $\im(f)$ we define the bilinear form $\langle f(x),f(y)\rangle=f(x)(y)$. This is a non-degenerate, $\sigma$-compatible $\ee$-form on the $\QQ_0$-graded vector space $\im(f)$. 
\end{proof}
\begin{lem}\label{Lem:ImageNableInvMorphism}
Let $L\in\ind(\QQ)$. 
\begin{enumerate}
\item In type $A_{odd}$, $\Hom(L,\nabla L)=\Hom(L,\nabla L)^{\nabla}$.
\item In type $A_{even}$, $\Hom(L,\nabla L)=\Hom(L,\nabla L)^{-\nabla}$.
\end{enumerate}
In other words, $\Hom(L,\nabla L)^{\ee\nabla}$ is zero if $(\QQ,\ee)$ is of split type.
\end{lem}
\begin{proof}
By Lemma~\ref{Lemma:OneDimensional}, $[L,\nabla L]\leq1$. If $[L,\nabla L]=0$, then there is nothing to prove. If  $[L,\nabla L]=1$, let $f:L\rightarrow \nabla L$ be a non-zero homomorphism. Suppose that $\Hom(L,\nabla L)=\Hom(L,\nabla L)^{\ee\nabla}$ for some $\ee\in\{+1,-1\}$. Then, by Lemma~\ref{Lemma:ImageIndecomposable}, the image of $f$ is indecomposable and by Lemma~\ref{Lem:ImageERepresentation} it is an $\ee$-representation. Thus $(\QQ,\ee)$ must be of non-split type, i.e.  $\ee=+1$ for $A_{odd}$ and $\ee=-1$ for $A_{even}$.
\end{proof}
\subsection{Properties of the AR-quiver} We study  properties of the Auslander-Reiten quiver of $\QQ$ in case $(\QQ,\sigma)$ is a symmetric quiver of type $A_n$. 

\subsubsection*{$\nabla$-invariant rectangles}
By \cite[Proposition~3.4]{DW}, the self-duality $\nabla$ fulfills 
\begin{equation} \nabla\tau=\tau^-\nabla.
\end{equation}
 Thus, $\nabla$ induces an involution of $\Gamma_\QQ$ which reverses the orientation of the arrows and fixes a vertical middle line (compare with Examples~\ref{Ex:ARQuiverA3}, \ref{Ex:ARQuiverAEvenAltern} and \ref{Ex:ARQuiverAOddAltern}). The fixed vertices and their direct sums are called \emph{$\nabla$-invariants}. For every $L\in\ind(\QQ)$, there is an integer $k$, such that $\nabla L=\tau^k L$. Assume that $[L,\nabla L]=1$ and $L\not\simeq \nabla L$. Then $\nabla L=\tau^{-k} L$ for some $k\geq 1$ and by Corollary~\ref{Cor:Convexity}, $[\nabla L,L]^1=[L,\tau\nabla L]=1$. Then by Lemma~\ref{Lem:MiddleTerm} there exists a non-degenerate rectangle (actually a square) in $\Gamma_\QQ$ from $L$ to $\nabla L$ whose middle vertices $G_1,G_2\in\ind(\QQ)$ are $\nabla$-invariants:
\begin{equation}\label{Eq:NablaInvSquare}
\xymatrix@R=10pt@C=10pt{
&G_1\ar@{..>}^{q_1}[dr]&\\
L\ar@{..>}^{p_1}[ur]\ar@{..>}_{p_2}[dr]&&\nabla L.\\
&G_2\ar@{..>}_{q_2}[ur]
}
\end{equation}
This is called a \emph{$\nabla$-invariant rectangle} of $\Gamma_\QQ$. It has the following property.
\begin{lem}
Either $q_1$ or $q_2$ is mono. 
\end{lem}
\begin{proof}
Since $G_1$ and $G_2$ are $\nabla$-invariants, $q_1$ and $q_2$ are non-zero scalar multiples of $\nabla p_1$ and $\nabla p_2$, respectively. Thus, if $q_2$ is epi, then $p_1$ is epi by \eqref{Eq:rectangles} and hence $\nabla p_1$ is mono by duality. It follows that $q_1$ is mono.
\end{proof}
\begin{cor}
Suppose that $q_1:G_1\rightarrow T$ is mono and let $p:\xymatrix@C=15pt{F\ar@{..>}[r]&G_1\ar@{..>}[r]&T}$ be a sectional path  which factors through $G_1$. Then $p$ is mono. 
\end{cor}
\subsubsection*{The function $\delta$} For every $M,N\in R_\mathbf{d}$ we consider the function $\delta_{M,N}:\Rep(\QQ)\rightarrow \ZZ$ defined as
$\delta_{M,N}(E)=[N,E]-[M,E]$.
This is a well-known additive function with many useful properties (cf. \cite{Zwara}). One of them is that if $E\in\ind(\QQ)$ is non-projective and $F$
is the middle term of the almost split sequence ending in $E$, then
\begin{equation}\label{Eq:DeltaAlmostSplit}
\delta_{M,N}(E)-\delta_{M,N}(F)+\delta_{M,N}(\tau E)=\mu(N,E)-\mu(M,E)
\end{equation}
where $\mu(N,E)$ (resp. $\mu(M,E)$) denotes the multiplicity of $E$ as a direct summand of $N$ (resp. $M$). This follows directly from the almost split properties.
\begin{lem}\label{Lem:DeltaSelfDual}
Let $M,N\in R_\mathbf{d}^\ee$ be two $\ee$-representations. Then
\begin{equation}
\delta_{M,N}(E)=\delta_{M,N}(\tau\nabla E)=\delta_{M,N}(\nabla\tau^- E)
\end{equation}
for every $E\in\ind(\QQ)$.
\end{lem}
\begin{proof}
We compute
\begin{eqnarray*}
\hspace{3.2cm}
\delta_{M,N}(E)&=&[N,E]-[M,E]\\
&=&[N,E]^1-[M,E]^1\\
&=&[\tau^- E, N]-[\tau^- E,M]\\
&=&[N,\nabla \tau^- E]-[M,\nabla \tau^- E]\\
&=&[N,\tau\nabla E]-[M,\tau\nabla E].
\hspace{3cm}
\qedhere
\end{eqnarray*}
\end{proof}
\noindent
\begin{definition}
We say that $F\in\Rep(\QQ)$ is \emph{$\delta$-fixed} if $F\simeq \tau\nabla F$.
\end{definition}
Clearly, $F$ is $\delta$-fixed if and only if all its indecomposable direct summands are. 
\begin{example}\label{Ex:Delta-fixed}
For the quiver $\QQ$ considered in Example~\ref{Ex:ARQuiverA3}, $P_2$ is the only  $\delta$-fixed indecomposable $\QQ$-representation. 
In Example \ref{Ex:ARQuiverAEvenAltern} ~the $\delta$-fixed indecomposables  are $U_{4,3}$ and $U_{1,3}$ and in Example~\ref{Ex:ARQuiverAOddAltern} ~they are $U_{3,5}$ and $U_{1,4}$.
\end{example}

\begin{lem}\label{Lem:DeltaFixedNablaInvariants}
Let $F\rightarrow G$ be an arrow of $\Gamma_\QQ$. Then
$F$ is $\delta$-fixed if and only if $G$ is $\nabla$-invariant.
\end{lem}
\begin{proof}
Suppose that $F$ is $\delta$-fixed. If $F$ is injective, then $\nabla F$ is projective and hence $F$ is, thus, not injective. Let $0\rightarrow F\rightarrow G'\rightarrow \tau^- F\rightarrow 0$ be the almost split sequence starting from $F$. Then $\tau^- F\simeq\nabla F$ and hence $G'$ is $\nabla$-invariant. Since $G$ is an indecomposable direct summand of $G'$, it is $\nabla$-invariant, too. Viceversa, suppose that $G$ is $\nabla$-invariant. Then $G$ is not projective. Let $0\rightarrow \tau G\rightarrow F'\rightarrow G\rightarrow 0$ be the almost split sequence ending in $G$. Then $0\rightarrow \nabla G\rightarrow \nabla F'\rightarrow \nabla \tau G\rightarrow 0$ is the almost split sequence starting in $\nabla G=G$. Since $\tau \nabla G=\tau G$ and $\tau \nabla\tau G=G$, necessarily, $\tau \nabla F'=F'$ and $F'$ is $\delta$-fixed. Since $F$ is a direct summand of $F'$, $F$ is $\delta$-fixed, too.
\end{proof}
\begin{lem}\label{Lem:ExistenceNablaInvariantQuasiSimple}
There exists a $\nabla$-invariant representation $G\in\ind(\QQ)$, such that there is only one arrow in $\Gamma_{\QQ}$ ending in $G$.
\end{lem}
\begin{proof}
Let $\omega$ be the unique source of $\QQ$ such that the path $\xymatrix@1{\omega\ar@{..>}[r]&\sigma(\omega)}$ is equioriented. We claim that  $G=U_{\omega,\sigma(\omega)}$ has the required property. Indeed, if $G$ is projective then its radical is indecomposable, since $G$ has a simple socle. If $G$ is not projective there exists an arrow $\omega\rightarrow j$ of $\QQ$ which does not belong to the path $\xymatrix@1{\omega\ar@{..>}[r]&\sigma(\omega)}$. 
Let $k\neq \omega$ be the source of $\QQ$ such that there exists an equioriented path $\xymatrix@1{k\ar@{..>}[r]&j}$  (if there is no such path, i.e. if $j$ is not a sink, 
put $k=j$). Then $\tau G=U_{k,j}$. Since $[U_{k,j},G]=0$ it follows from Lemma~\ref{Lem:MiddleTerm} that the middle term of the almost split sequence ending in $G$ is indecomposable (namely, it is $U_{k,\sigma(\omega)}$). 
\end{proof}

\begin{prop}\label{prop:DeltaEven}
Let $(\QQ,\ee)$ be of split type and let $M,N\in R_\mathbf{d}^{\ee}$. Then 
\begin{equation}\label{Eq:DeltaFEven}
\delta_{M,N}(F)\textrm{ is even for every }\delta\textrm{-fixed }F.
\end{equation}
\end{prop}
\begin{proof}
We prove the statement for every indecomposable $\delta$-fixed representation.
Let $G\in\ind(\QQ)$ be $\nabla$-invariant and let $0\rightarrow \tau G\rightarrow F\rightarrow G\rightarrow 0$ be the almost split sequence ending in $G$.  By Lemma~\ref{Lem:DeltaSelfDual}, $\delta_{M,N}(G)=\delta_{M,N}(\tau G)$ and by \eqref{Eq:DeltaAlmostSplit} we get
\begin{equation}\label{Eq:DeltaFixedEven}
\delta_{M,N}(F)=2\delta_{M,N}(G)-\mu(N,G)+\mu(M,G).
\end{equation}
Since $M$ and $N$ do not have $\ee$-indecomposable direct summands of type $(I)$ and $G$ is $\nabla$-invariant, its multiplicity both in $M$ and in $N$ is even. We hence see that $\delta_{M,N}(F)$ is even. Let $F_1,\cdots, F_k\in\ind(\QQ)$ be the $\delta$-invariant isoclasses, ordered from top to bottom in $\Gamma_\QQ$. Then $F_i\oplus F_{i+1}$ is the middle term of an almost split sequence ending in a $\nabla$-invariant indecomposable representation, and by \eqref{Eq:DeltaFixedEven}, $\delta_{M,N}(F_i)$ and $\delta_{M,N}(F_{i+1})$ have the same parity, for all $i=1,\cdots, k-1$. Thus, 
 $\delta_{M,N}(F_1),\cdots, \delta_{M,N}(F_k)$ have all the same parity. To conclude the proof it is enough to show that one of them is even. By Lemma~\ref{Lem:ExistenceNablaInvariantQuasiSimple}, either $F_1$ or $F_k$ is the middle term of an almost split sequence ending in an indecomposable $\nabla$-invariant representation, and hence, by \eqref{Eq:DeltaFixedEven}, either $\delta_{M,N}(F_1)$ or $\delta_{M,N}(F_k)$ is even. 
\end{proof}

\subsection{Generic isotropic embeddings}\label{Sec:KeyLemma}
We state and prove a surprising result that says that under certain mild hypotheses, it is always possible to embed  an indecomposable inside an $\ee$--representation so that the embedding is isotropic and generic.  In the split types this is obvious since every embedding is isotropic: 
\begin{lem}\label{Lem:IsoSplitTypes}
Let $(\QQ, \ee)$ be of split type. Let $M\in R(\A,V)^{\Psi,\ee}$ and $L\in\ind(\QQ)$. Then every embedding of $L$ into $M$ is isotropic.
\end{lem}
\begin{proof}
Let $\iota:L\hookrightarrow M$ be a monomorphism. Then $\nabla\iota\circ\Psi\circ\iota\in\Hom(L,\nabla L)^{\ee\nabla}$ is zero by Lemma~\ref{Lem:ImageNableInvMorphism}. The statement is now a consequence of Corollary~\ref{Cor:Isotropic}. 
\end{proof}
In the non-split types the statement of Lemma~\ref{Lem:IsoSplitTypes} is not true (see Examples~\ref{Ex:NonIsoEmbedding1} and \ref{Ex:NonIsoEmbedding2}) and one needs to add extra hypotheses on $L$ and $M$.
\begin{prop}\label{Prop:KeyLemma}
Let $(\QQ, \ee)$ be of non-split type. Let $M\in R(\A,V)^{\Psi,\ee}$ be an $\ee$--representation and let $L\in\ind(\QQ)$ be an indecomposable subrepresentation of $M$. Let $\iota:L\hookrightarrow M$ be a generic embedding with (generic) quotient $Q$. Assume that:
\begin{eqnarray}
\label{Eqn:Hypt2KeyLemma}&& \textrm{There is a surjective homomorphism }\xymatrix@1{Q\ar@{->>}[r]&\nabla L};\\
\label{Eqn:Hypt3KeyLemma}&& [L,K]^1=0\textrm{ where }K\textrm{ is the generic kernel of }\xymatrix@1{Q\ar@{->>}[r]&\nabla L}.
\end{eqnarray}
Then there exists an automorphism $g$ of $M$ such that $j=g\circ \iota$ is an isotropic embedding of $L$ into $M$. In particular, $j$ is an isotropic and generic embedding.
\end{prop}
\begin{proof}
Since $L$ is indecomposable,  we have $[L,\nabla L]\leq 1$ by Lemma~\ref{Lemma:OneDimensional}. 
\\[1ex]
If $[L,\nabla L]=0$ then every embedding $\iota:L\hookrightarrow M$ is isotropic by Corollary~\ref{Cor:Isotropic}.
\\[1ex]
Let us assume that $[L,\nabla L]=1$.
By Lemma \ref{Lem:ImageNableInvMorphism} $\Hom(L,\nabla L)=\Hom(L,\nabla L)^{\ee\nabla}=\mathrm{k} \varphi$ for a non-zero homomorphism $\varphi=\ee\nabla\varphi :L\rightarrow \nabla L$. We need to prove that there exists an automorphism $g$ of $M$ such that $j:=g\circ \iota$ is isotropic, i.e. $\nabla j\circ\Psi\circ j=0$.
\\[1ex]
Let $Z$ be an indecomposable direct summand of $M$ which is $\preceq$-minimal in the support of $\Hom(L,-)$. We say that $Z$ is a \emph{split} (resp.~\emph{non-split}) $\preceq$-minimal direct summand of $M$ relatively to $L$ if $Z\oplus \nabla Z$ is (resp.~is \emph{not}) a direct summand of $M$. Thus if $Z$ is non-split then $Z\simeq\nabla Z$, but the converse is not necessarily true: if $Z\simeq\nabla Z$ and there is more than one copy of $Z$ in $M$, then $Z$ is split even if it is $\nabla$-invariant (compare with example~\ref{Ex:SplitDirectsummand} below).
\\[1ex]
Let $S$ (for ``self-dual'') denote the direct sum of all the non-split $\preceq$-minimal direct summands of $M$ relatively to $L$ and let $T^0$ denote the direct sum of all the split $\preceq$-minimal direct summands of $M$ relatively to $L$. Then $M=S\oplus T^0\oplus \nabla T^0\oplus\overline{M}$. By Proposition~\ref{Prop:GenQuot}, 
the generic quotient of $M$ by $L$ has the the form $Q=T^1\oplus \nabla T^0\oplus\overline{M}$ (here $\widetilde{M}=\nabla T^0\oplus\overline{M}$).
We write $S=M(1)\oplus\cdots\oplus M(s)$ and $T^0=T^0_{s+1}\oplus\cdots\oplus T^0_N$ as a sum of their indecomposable direct summands. Thus, we write $M$ as a direct sum of $\ee$-representations as $M=M(1)\oplus\cdots\oplus M(N)\oplus\overline{M}$ where $M(k)=T^0_k\oplus\nabla T^0_k$ for $k=s+1,\cdots, N$. In view of Theorem~\ref{thm:isoclasses} we can assume that $\Psi$ is:
$$
\Psi=\bigoplus_{k=1}^N \Psi_k\oplus \overline{\Psi}:M(1)\oplus\cdots\oplus M(N)\oplus\overline{M}\rightarrow \nabla M(1)\oplus\cdots\oplus \nabla M(N)\oplus\nabla \overline{M}
$$
where  $\Psi_k:M(k)\rightarrow\nabla M(k)$ is a non-zero homomorphism for $k=1,\cdots, s$ and  
$$
\Psi_k:\xymatrix@C=90pt{M(k)=T^0_k\oplus\nabla T^0_k\ar^(.55){\left(\begin{array}{cc}{\scriptstyle 0}&{\scriptstyle Id_{\nabla T^0_k}}\\ {\scriptstyle \ee Id_{T^0_k}}&{\scriptstyle 0}\end{array}\right)}[r]& \nabla T^0_k\oplus T^0_k}.
$$ 
We denote by $\iota_k:L\rightarrow M(k)$ the component of $\iota$ along  the direct summand $M(k)$. The generic embedding $\iota$ has hence the form 
$$
\iota=(\iota_1,\cdots, \iota_N, 0)^t:L\hookrightarrow M(1)\oplus\cdots\oplus M(N)\oplus \overline{M}
$$
Let us discuss $\nabla \iota\circ\Psi\circ \iota$. With the choice we made for the form, we have $\nabla \iota\circ\Psi\circ \iota=\sum_{k=1}^N\nabla\iota_k\circ\Psi_k\circ \iota_k$. Every summand $\nabla\iota_k\circ\Psi_k\circ\iota_k$ is a multiple of $\varphi$. For $k\leq s$, $M(k)$ is indecomposable and $\nabla$-invariant, and thus $\nabla\iota_k\circ\Psi_k\circ\iota_k$ is non-zero by Lemma~\ref{Lem:ImageNableInvMorphism}. 
and the fact that $[L,\nabla L]\neq0$.
\\[1ex]
For $k>s$ we choose two homomorphisms $j_k:L\rightarrow T^0_k$ and $\ell_k:T^0_k\rightarrow \nabla T^0_k$ such that $j_k$ is non-zero and $\ell_k$ is non-zero whenever possible or stated differently. (Notice that $\ell_k$ can be chosen to be non-zero in case $[L,\nabla T^0_k]\neq 0$ by Corollary~\ref{Cor:Convexity}.) We can choose $\iota_k=(j_k,\ell_k\circ j_k)^t$, since the generic quotient $Q$ does not depend on $\ell_k$. By Lemma~\ref{Lem:ImageNableInvMorphism} and our assumption on $\ee$, we get $\nabla\iota_k\circ\Psi_k\circ \iota_k=\nabla j_k\circ\ell_k\circ j_k+\ee\nabla j_k\circ\nabla\ell_k\circ j_k=2 \nabla j_k\circ\ell_k\circ j_k$. In particular, if $[L,\nabla T^0_k]=0$, or $\ell_k$ is chosen to be zero, then $\nabla\iota_k\circ\Psi_k\circ \iota_k$ is zero.
\\[1ex]
Let us reorder the direct summands $M(k)$'s of $M$ so that $\nabla\iota_k\circ\Psi_k\circ\iota_k$ is non-zero for $1\leq k\leq h$ and it is zero for $k>h$. 
\\[1ex]If $h=0$ then $\nabla\iota\circ\iota$ is zero and thus $\iota(L)$ is isotropic. 
\\[1ex]Assume that $h\geq1$. For every $k=1,\cdots, h$ there exists a non-zero complex number $x_k$ such that $\nabla\iota_k\circ\Psi_k\circ\iota_k=x_k\varphi$. In particular we get 
$$
\nabla\iota\circ\iota=(x_1+\cdots+x_h)\varphi.
$$ 
If $h>1$ then let $(a_1,\cdots, a_h)\in(\CC\setminus\{0\})^h$ be a non-zero solution of the linear equation $x_1+\cdots+x_h=0$ (for example $(1,1,\cdots, 1, -(h-1))$). Let $g$ be the automorphism of $M$ which rescales the direct summand $M(k)$ by $\sqrt{\frac{a_k}{x_k}}$ for $k=1,\cdots, h$ and acts as the identity on all the other direct summands. Then $j=g\circ \iota$ is a generic embedding such that $\nabla j\circ j=(a_1+\cdots+a_k)\varphi=0$ which is what we wanted to prove.  (We notice that in the case when $h>1$ there is no need of the hypotheses of the proposition.)
\\[1ex]
Now we discuss the case $h=1$ which again splits up into two cases:
\begin{enumerate}
\item $S=M(1)$ is indecomposable and $[L,\nabla T^0]=0$. We show that this case contradicts our hypotheses. 
By hypothesis, there is a short exact sequence
$$
\xymatrix{
0\ar[r]&L\ar[r]&S\oplus T^0\ar[r]&T^1\ar[r]&0
}
$$
such that the induced homomorphism $\Hom(S\oplus T^0,\nabla L)\rightarrow\Hom(L,\nabla L)$ is surjective (this is because $\nabla \iota_1\circ\iota_1=x_1\varphi$ is a basis of $\Hom(L,\nabla L)$); it follows that $[T^1,\nabla L]=[T^0,\nabla L]=[L,\nabla T^0]=0$.

The generic quotient of $M$ by $L$ is $Q=T^1\oplus \nabla T^0\oplus\overline{M}$. By hypothesis~\eqref{Eqn:Hypt2KeyLemma} there exists a surjective homomorphism $\xymatrix{Q\ar@{->>}[r]&\nabla L}$ with generic kernel $K$; since $[T_1,\nabla L]=0$,  $T_1$ is a direct summand of $K$. Moreover, by the definition of $T^0$, and again because $[T^1,\nabla L]=0$, we see that $\nabla T^0$ is the direct sum of all the indecomposable direct summands of $Q$ which are $\preceq$-maximal in the support of $\Hom(-,\nabla L)$. By the dual version of Proposition~\ref{Prop:GenQuot}, we conclude that there is a surjective homomorphism $\xymatrix{\nabla T^0\ar@{->>}[r]&\nabla L}$ and that if we denote by $\overline{K}$ its kernel, then  $K=T^1\oplus \overline{M}\oplus \overline{K}$. We claim that $[L,K]^1\neq 0$. Indeed, we apply $\Hom(L,-)$ to the short exact sequence 
$$
\xymatrix{
0\ar[r]&K=T^1\oplus \overline{K}\oplus \overline{M}\ar[r]&Q=T^1\oplus\nabla T^0\oplus \overline{M}\ar[r]&\nabla L\ar[r]&0
}
$$
and get that the induced map $\Hom(L,\nabla L)\rightarrow\Ext^1(L,K)$ is injective. Since $[L,\nabla L]=1$, this contradicts hypothesis~\ref{Eqn:Hypt3KeyLemma}. 

\item $S$ is not present,  $[L,T^0_1\oplus \nabla T^0_1]=2$ and $[L,T^0_k\oplus \nabla T^0_k]=1$ for all $2\leq k\leq N$.  In this case we choose $\ell_1=0$ and $\iota_1=(j_1,0)^t: L\rightarrow T^0_1\oplus\nabla T^0_1$ and get that $\nabla \iota\circ\iota=0$ without further hypotheses. \qedhere
\end{enumerate}
\end{proof}
\subsubsection*{Examples} The following two examples show that  the  hypotheses of Proposition~\ref{Prop:KeyLemma} are necessary. 
\begin{example}\label{Ex:NonIsoEmbedding1}
Let $\QQ$ be of type $\stackrel{\rightarrow}{A_n}$ and let $M=U_{i,\sigma(i)}$ be an indecomposable $\ee$--representation of type $(I)$, for some vertex $i\leq n/2$. Let $L$ be a subrepresentation of $M$ such that $[L,\nabla L]=1$, e.g. $L=U_{j,\sigma(i)}$ for some $i\leq j\leq n/2$. Then the quotient $M/L$ does not surject onto $\nabla L$ (actually $[M/L,\nabla L]=0$), contradicting hypothesis \eqref{Eqn:Hypt2KeyLemma}. In this case $L$ is not an isotropic subrepresentation of $M$.
\end{example}
\begin{example}\label{Ex:NonIsoEmbedding2}
Let $\QQ$ be of type $\stackrel{\rightarrow}{A_5}$. We consider $L\in\ind(\QQ)$ and the \emph{orthogonal} $\QQ$--representation $M=S\oplus T^0\oplus\nabla T^0$ given by
$$
\xymatrix@!R=2pt@!C=2pt{
&&&&\bullet\ar@{->>}[dr]&&&&\\
&&&\nabla T^0\ar[ur]\ar@{->>}[dr]&&T^0\ar@{->>}[dr]&&&\\
&&\nabla T^1\ar[ur]\ar@{->>}[dr]&&\bullet\ar[ur]\ar@{->>}[dr]&&T^1\ar@{->>}[dr]&&\\
&H\ar[ur]\ar@{->>}[dr]&&L\ar[ur]\ar@{->>}[dr]&&\nabla L\ar[ur]\ar@{->>}[dr]&&\bullet\ar@{->>}[dr]&\\
\bullet\ar[ur]&&\bullet\ar[ur]&&S\ar[ur]&&\bullet\ar[ur]&&\bullet
}
$$
(the conclusion of this example applies in every situation where the mutual position of the summands in the AR-quiver is similar to the above). Then, $[L,M]=2$, $L$ embeds into $M$, the generic quotient of $M$ by $L$ is $Q=T^1\oplus \nabla T^0$ and there is a surjective homomorphism $\xymatrix@1@C=15pt{p:Q\ar@{->>}[r]&\nabla T^0\ar@{->>}[r]&\nabla L}$, thus  hypothesis~\eqref{Eqn:Hypt2KeyLemma} is satisfied. The kernel of $p$ is $K=T^1\oplus H$ and thus $[L,K]^1=1$ contradicting hypothesis \eqref{Eqn:Hypt3KeyLemma}. In this case it is not true that a generic embedding of $L$ into $M$ is isotropic. To see this, let us fix an orthogonal basis of $M$ as follows
$$
\xymatrix@R=0pt{
&&&u&&\\
M=&&v_1\ar[r]&v_2\ar[r]&v_3\ar[r]&v_4\\
&v_4^\ast\ar[r]&v_3^\ast\ar[r]&v_2^\ast\ar[r]&v_1^\ast&
}
$$ 
with $\langle u,u\rangle=\langle v_i,v_i^\ast\rangle=\langle v_i^\ast,v_i\rangle=1$. Then 
$$
\xymatrix@R=0pt{
\iota(L)=&&&xu+v_2^\ast\ar[r]&v_1^\ast&
}
$$
for some $x\neq 0$ and thus $\langle \iota(L),\iota(L)\rangle=\langle xu,xu\rangle=x^2\neq0$, proving that $\iota(L)$ is not isotropic. (If $x=0$, the map is still an embedding, it is isotropic but it is not generic, since the quotient is $\nabla H\oplus S\oplus\nabla T^0$ which is a degeneration of $Q$.)
\end{example}
Let us illustrate the proof of Proposition~\ref{Prop:KeyLemma} in an example.
\begin{example}
Let us consider the following indecomposable representation $L$ and the following \emph{orthogonal} representation $M=S\oplus T^0\oplus\nabla T^0$ of a quiver of type $\stackrel{\rightarrow}{A_5}$:
$$
\xymatrix@!R=3pt@!C=3pt{
&&&&S\ar@{->>}[dr]\ar@/^1pc/@{-->}^{\nabla\iota_S}[ddrr]&&&&\\
&&&\nabla T^1\ar[ur]\ar@{->>}[dr]&&T^1\ar@{->>}[dr]&&&\\
&&L\ar[ur]\ar@{->>}^{\iota_{T^0}}[dr]\ar@/^1pc/@{-->}^{\iota_S}[uurr]&&\bullet\ar[ur]\ar@{->>}[dr]&&\nabla L\ar@{->>}[dr]&&\\
&\bullet\ar[ur]\ar@{->>}[dr]&&T^0\ar[ur]\ar^\ell@{-->}[rr]\ar@{->>}[dr]&&\nabla T^0\ar_{\nabla\iota_{T^0}}[ur]\ar@{->>}[dr]&&\bullet\ar@{->>}[dr]&\\
\bullet\ar[ur]&&\bullet\ar[ur]&&\bullet\ar[ur]&&\bullet\ar[ur]&&\bullet
}
$$
We choose non-zero homomorphisms $\iota_S$, $\iota_{T^0}$ and $\ell$ as depicted. Then  the homomorphism
$
\iota=(\iota_S,\iota_{T^0},\ell\circ\iota_{T^0})^t:L\rightarrow S\oplus T^0\oplus\nabla T^0
$ is a generic embedding. The quotient is $Q=T^1\oplus\nabla T^0$. Then $\nabla\iota\circ \iota=\nabla\iota_S\circ\iota_S+2\nabla \iota_{T^0}\circ\ell\circ \iota_{T^0}$. There exists a non-zero complex number $x$ such that $\nabla \iota_S\circ\iota_S=x(2\nabla \iota_{T^0}\circ\ell\circ \iota_{T^0})$. We define $j_{T^0}:=\frac{\sqrt{-x}}{\sqrt{2}}\iota_{T^0}$ and $j:=(\iota_S, j_{T^0},\ell\circ j_{T^0})^t:L\rightarrow M$. Then $j$ is a generic embedding of $L$ into $M$ and $\nabla j\circ j=0$. 
\end{example}
\begin{example}\label{Ex:SplitDirectsummand}
Let $\QQ:1\rightarrow 2\rightarrow 3$ be of type $\overrightarrow{A}_3$. Its AR-quiver is shown in Example~\ref{Ex:ARQuiverA3}. Let $M=P_1^n$ and let $L=P_2$. Then $M$ is an orthogonal representation for every $n\geq1$ and it is also symplectic for $n$ even. If $n=1$, then $L$ does not embed isotropically into $M$ (compare with Example~\ref{Ex:NonIsoEmbedding1}). For $n\geq 2$, $L$ embeds isotropically into $M$ because we can change the form so that one copy of $P_1$ is \emph{isotropic}, i.e. we can write $M$ as $T^0\oplus\nabla T^0\oplus \overline{M}$ where $T^0\simeq \nabla T^0\simeq P_1$; then the embedding $P_2\hookrightarrow P_1$ is isotropic as well and it is generic. This example motivates our choice for the (unexpected) decomposition $M=T^0\oplus\nabla T^0\oplus\overline{M}$ which is not the natural decomposition into indecomposable $\ee$-representations (for $\ee=1$).
\end{example}

\subsection{Generic $\ee$-subquotients}\label{Sec:GenESubquot}
Let $M\in R(\A,V)^{\Psi,\ee}$ be an $\ee$-representation and let $L\in\ind(\QQ)$.
If there exists an isotropic embedding $\iota:L\rightarrow M$ then the quotient $\iota(L)^\perp/\iota(L)$ is called \emph{the $\ee$-subquotient of $M$ by $\iota(L)$}. In this section we study $\ee$-subquotients of $M$ induced by a generic isotropic embedding of $L$ into $M$.
It turns out that there are substantial differences between the split and the non-split types.

Let $\iota:L\rightarrow M$ be a monomorphism with generic quotient $\pi:\xymatrix{M\ar@{->>}[r]&Q}$ and assume further that $\iota(L)$ is isotropic. Let $\xymatrix@1{Y=\iota(L)^\perp/\iota(L)}$ be the $\ee$-subquotient. We use the notation introduced in Section~\ref{Sec:GenericQuot} and in the proof of Proposition~\ref{Prop:KeyLemma}
\begin{prop}\label{Prop:ESubquot}
\begin{enumerate}
\item
In the split  types,  $Y\simeq T^1\oplus\nabla T^1\oplus\overline{M}$. 
\item 
In the non-split types, there exists a generic embedding $\iota$ so that the corresponding $\ee$-subquotient
$Y$ is the generic quotient of $\nabla Q$ by $L$.
\end{enumerate}
\end{prop}
\begin{proof}
The proof of Proposition~\ref{Prop:ESubquot} needs a little preparation.
\\[1ex]
We decompose $M=S\oplus T^0\oplus\nabla T^0\oplus\overline{M}$ as in the proof of Proposition~\ref{Prop:KeyLemma}, with the convention that $S$ is zero in the split types. 
The direct summand $\overline{M}$ does not play any role, since it splits off in $M$, $Q$ and $Y$ for every choice of a generic embedding. Thus we assume, for simplicity of notation, that it is zero. 
\\[1ex]
By Theorem~\ref{thm:isoclasses}, or its Corollary~\ref{cor:NoMatterForm}, we can choose the $\ee$-form at our convenience. We thus assume that $\Psi:M\rightarrow\nabla M$ is given in the following block form 
$$
\xymatrix@C=100pt{
\Psi: S\oplus T^0\oplus\nabla T^0
\ar^
{
\left(
\begin{array}{ccc}
{\scriptstyle\Psi_S}&0&0\\
0&0&{\scriptstyle Id_{\scriptscriptstyle T^0}}\\
0& \ee {\scriptstyle Id}_{\scriptscriptstyle \nabla T^0}&0
\end{array}
\right)
}[r]&\nabla S\oplus\nabla T^0\oplus T^0.
}
$$
\\[1ex]
For every index $k$, such that $[L,\nabla T^0_k]=1$, we choose arbitrarily a morphism $\ell_k:T^0_k\rightarrow\nabla T^0_k$ and let $\ell=\textrm{diag}(\ell_k):T^0\rightarrow\nabla T^0$ be the induced diagonal morphism. 
\\[1ex]
As in the proof of Proposition~\ref{Prop:KeyLemma}, and by Lemma~\ref{Lem:IsoSplitTypes}, we assume that the generic embedding $\iota$ depends on the choice of $\ell$ and it is given by $\iota=(\iota_S,\,\iota_{T^0},\,\ell\circ\iota_{T^0})^t$. Thus, there is a short exact sequence of the form
$$
\xymatrix@C=80pt{
0\rightarrow L\ar^(.45){\iota=(\iota_S,\,\iota_{T^0},\,\ell\circ\iota_{T^0})^t}[r]&S\oplus T^0\nabla T^0\ar^(.45){\left(\begin{array}{ccc}{\scriptstyle p_S}&{\scriptstyle p_{T^0}}&0\\0&{\scriptstyle -\ell}&
{\scriptstyle 1}\end{array}\right)}[r]& T^1\oplus\nabla T^0\rightarrow 0.
}
$$
A short computation shows that condition $\nabla\iota\circ\Psi\circ\iota=0$ is equivalent to 
$$
\nabla \iota_S\circ\Psi_S\circ\iota_S+\nabla\iota_{T^0}\circ\ell\circ\iota_{T^0}+\ee\nabla\iota_{T^0}\circ\nabla\ell\circ\iota_{T^0}=0.
$$
In the split types this condition is redundant, by Lemma~\ref{Lem:ImageNableInvMorphism}. On the other hand, in the non-split types, since $\ee\nabla\ell=\ell$, this becomes $\nabla \iota_S\circ\Psi_S\circ\iota_S+2\nabla\iota_{T^0}\circ\ell\circ\iota_{T^0}=0$.
\\[1ex] 
By Lemma~\ref{Lem:NperpNabla}, since $\iota(L)$ is isotropic, it is contained in the kernel of the map $\nabla\iota\circ\Psi: M\rightarrow\nabla L$. Thus, there exists $j:T^1\rightarrow \nabla L$ and a surjective map that we denote by $f:T^1\oplus\nabla T^0\rightarrow \nabla L$ which makes the following diagram commutative:
$$
\xymatrix@C=80pt{
0\rightarrow L\ar^{\iota=\left(\begin{array}{c} {\scriptstyle\iota_S}\\ {\scriptstyle\iota_{T^0}}\\ 
{\scriptstyle\ell\circ\iota_{T^0}}\end{array}\right)}[r]&S\oplus T^0\nabla T^0
\ar^{\pi=\left(\begin{array}{ccc}{\scriptstyle p_S}&{\scriptstyle p_{T^0}}&{\scriptstyle 0}\\ {\scriptstyle 0}&{\scriptstyle -\ell}&{\scriptstyle 1}\end{array}\right)}[r]
\ar_{({\scriptstyle \nabla\iota_S\circ \Psi_S,\,\nabla\iota_{T^0}\circ\nabla\ell,\, \ee\nabla\iota_{T^0}})}^{{\scriptstyle=\nabla\iota\circ\Psi}}[d]& T^1\oplus\nabla T^0\rightarrow 0\ar^{{\scriptstyle f=(j,\ee\nabla\iota_{T^0})}}[d]\\
&\nabla L\ar@{=}[r]&\nabla L
}
$$
We notice that $j$ satisfies the following equalities
\begin{eqnarray}\label{Eq:EqualityJ1}
j\circ p_S&=&\nabla\iota_S\circ\Psi_S,\\\label{Eq:EqualityJ2}
j\circ p_{T^0}-\ee\nabla\iota_{T^0}\circ\ell&=&\nabla\iota_{T^0}\circ\nabla\ell.
\end{eqnarray}
In the split types, equation \eqref{Eq:EqualityJ1} is not present since $S$ is zero; moreover, since by Lemma~\ref{Lem:ImageNableInvMorphism} $\ee\nabla\ell=-\ell$,  equation \eqref{Eq:EqualityJ2} becomes $j\circ p_{T^0}=0$; since $p_{T^0}$ is surjective (since $(p_S,p_{T^0})=p_{T^0}$ is surjective) we arrive at  the following crucial observation:
\begin{equation}\label{Eq:JIsZeroSplitTypes}
j=0\textrm{ in the split types}.
\end{equation}
In the non-split types, since by Lemma~\ref{Lem:ImageNableInvMorphism} $\ee\ell=\nabla\ell$, equation~\eqref{Eq:EqualityJ2} becomes 
\begin{equation}\label{Eq:JNSplitTypes}
j\circ p_{T^0}=2\nabla\iota_{T^0}\circ\nabla\ell\textrm{ in the non-split types.}
\end{equation} 
We have constructed the following commutative diagrams whose rows and columns are short exact sequences 
\begin{equation}\label{Eq:DiagIsotropicEmbedding}
\vcenter{\xymatrix{
L\ar^(.3){\nabla f}@{^(->}[r]\ar@{=}[d]&\nabla Q=\nabla T^1\oplus T^0\ar@{->>}[r]\ar^{\Psi^{-1}\circ\nabla\pi}@{^(->}[d]&Y\ar@{^(->}[d]\\
L\ar^(.3)\iota@{^(->}[r]&M=S\oplus T^0\oplus\nabla T^0\ar^\pi@{->>}[r]\ar^{\nabla\iota\circ\Psi}@{->>}[d]&Q=T^1\oplus \nabla T^0\ar^f@{->>}[d]\\
&\nabla L\ar@{=}[r]&\nabla L.
}}
\end{equation}
We remark that all the maps involved depend on the choice of $\ell$. We need to show that in the split types $Y\simeq T^1\oplus\nabla T^1$, for every choice of $\ell$, and in the non-split types, there is a choice of $\ell$ such that $\nabla f$ is a generic embedding of $L$ into $\nabla Q$. 
\\[1ex]
In the split types, by \eqref{Eq:JIsZeroSplitTypes}, we have  $\nabla f =(0, \ee\iota)^t$ and thus $Y\simeq T^1\oplus\nabla T^1$ for every choice of $\ell$, proving the first part of the proposition.
\\[1ex]
In the non-split types, we choose $\ell_k\neq0$ for every possible index $k$ (see the proof of Proposition~\ref{Prop:KeyLemma}). Then $\nabla f=(\nabla j,\ee\iota _{T^0})$. We claim that we can choose $j$ such that this map is generic. In view of Proposition~\ref{Prop:GenQuot}, it is enough to show that $\nabla f$ hits all the $\preceq$-minimal indecomposable direct summands of $\nabla T^1\oplus T^0$ contained in the support of $\Hom(L,-)$. By definition of $\iota_{T^0}$, every direct summand of $T^0$ is hit by the map. Thus, we only need to discuss the direct summands of $\nabla T^1$. Let $\nabla T^1_i$ be a $\preceq$-minimal direct summand of $\nabla T^1\oplus T^0$ contained in the support of $\Hom(L,-)$. We claim that we can choose $j$ such that the component $\nabla j_i:L\rightarrow \nabla T^1_i$  is non-zero. Since $T^0_i\preceq T^1_i$, necessarily $T^1_i\not\preceq \nabla T^1_i$ and hence we are in one of the  following situations in $\Gamma_\QQ$
$$
\begin{array}{ccc}
\vcenter{
\xymatrix{
&T^0_i\ar@{..>}[dr]\ar@{~>}^{\ell_i}[r]&\nabla T^0_i\ar@{~>}[dr]&\\
L\ar@{~>}^{\nabla j_i}[r]\ar@{~>}^{(\iota_{T^0})_i}[ur]\ar@{~>}_{(\iota_{T^0})_{i+1}}[dr]&\nabla T^1_{i}\ar@{..>}[dr]\ar@{..>}[ur]& T^1_{i}\ar@{~>}^{j_i}[r]&\nabla L\\
&T^0_{i+1}\ar@{..>}[ur]\ar@{~>}_{\ell_{i+1}}[r]&\nabla T^0_{i+1}\ar@{~>}[ur]&
}}
&
\textrm{or}
&
\vcenter{
\xymatrix@C=5pt{
                                                                               &&T^0_i\ar@{..>}[drr]\ar@{~>}^{\ell_i}[rr]  &                                             &\nabla T^0_i\ar@{~>}[drr] &&\\
L\ar@{~>}^{\nabla j_i}[rr]\ar@{~>}^{(\iota_{T^0})_i}[urr]\ar@{~>}_{(\iota_S)_j}[drrr]&&\nabla T^1_{i}\ar@{..>}[dr]\ar@{..>}[urr]&                                             &T^1_{i}\ar@{~>}^{j_i}[rr]    &&\nabla L\\
								           &&                                                              &S_j\ar@{..>}[ur]\ar@{~>}[urrr]&                                         &&
}}
\end{array}
$$
where the dotted arrows $\xymatrix{\ar@{..>}[r]&}$ denote sectional paths, the wavy arrows $\xymatrix{\ar@{~>}[r]&}$ denote a composition of sectional paths and $S_j$ denotes an indecomposable direct summand of $S$. In the situation on the right, $\nabla j_i$ is non-zero by \eqref{Eq:EqualityJ1}. In the left-hand situation, if either $\ell_i$ or $\ell_{i+1}$ is non-zero then $\nabla j_i$ is non zero by \eqref{Eq:JNSplitTypes}. If both $\ell_i$ and $\ell_{i+1}$ are zero, then any choice of $\nabla j_i$ makes the equation~\ref{Eq:JNSplitTypes} true; in particular we can choose $(\nabla j)_i$ to be non-zero. This concludes the proof.
\end{proof}
\begin{definition}
The representation $Y$ constructed in Proposition~\ref{Prop:ESubquot} is called \emph{the generic $\ee$-subquotient of $M$ by $L$}.
\end{definition}
\begin{rem}
We called $Y$ \emph{generic} since it comes from a generic embedding. In the proof of Theorem~\ref{Thm:MainThm} we will see that  $Y$ has  ``generic  properties'' (see remark~\ref{Rk:GenericSubquotient}).
\end{rem}

\begin{example}
Let $\QQ$ be the quiver of Example~\ref{Ex:ARQuiverAEvenAltern} and let us consider the following indecomposable $\QQ$-representations:
$$
\xymatrix@!R=3pt@!C=3pt{
&T^0\ar[dr]&&\nabla T^0\ar[dr]&\\
L\ar[ur]\ar[dr]&&Y\ar[ur]\ar[dr]&&\nabla L\\
&\nabla T^1\ar[ur]\ar[dr]&&T^1\ar[ur]\ar[dr]&\\
\bullet\ar[ur]&&S\ar[ur]&&\bullet
}
$$
For $\ee=-1$ let $M=T^0\oplus\nabla T^0\oplus S$. The generic quotient of $M$ by $L$ is $Q=T^1\oplus\nabla T^0$. Then $\nabla Q=\nabla T^1\oplus T^0$ and the generic quotient of $\nabla Q$ by $L$ is $Y$ which, by Proposition~\ref{Prop:ESubquot} is the generic $(-1)$-subquotient of $M$ by $L$. 
\\[1ex]
For $\ee=1$ let $M=T^0\oplus\nabla T^0\oplus S\oplus \nabla S$. 
Then the generic quotient of $M$ by $L$ is $Q=T^1\oplus \nabla T^0\oplus \nabla S$ and 
by Proposition~\ref{Prop:ESubquot} the generic $(1)$-subquotient of $M$ by $L$ is $T^1\oplus\nabla T^1$.
\end{example}
\begin{example}\label{Ex:GenericLperpmoduloL}
Let $\QQ=\stackrel{\rightarrow}{A_4}$. Let us consider the $\QQ$-representations $M=T^0\oplus\nabla T^0$ and $L$ given by 
$$
\xymatrix@!R=3pt@!C=3pt{
&&&\bullet\ar@{->>}[dr]&&&\\
&&T^0\ar@{..>}^\ell[rr]\ar[ur]\ar@{->>}[dr]&&\nabla T^0\ar@{->>}[dr]&&\\
&L\ar[ur]\ar@{->>}[dr]&&S\ar[ur]\ar@{->>}[dr]&&\nabla L\ar@{->>}[dr]&\\
\bullet\ar[ur]&&\nabla T^1\ar[ur]&&T^1\ar[ur]&&\bullet
}
$$
Since $[L,\nabla L]=0$, every embedding of $L$ into $M$ is isotropic.
The generic quotient of $M$ by $L$ is $Q=T^1\oplus\nabla T^0$ and thus $\nabla Q=T^0\oplus\nabla T^1$. By Proposition~\ref{Prop:ESubquot}, the generic $\ee$-subquotient of $M$ by $L$ is $S$ if $\ee=-1$ and it is $T^1\oplus\nabla T^1$ if $\ee=1$.
\end{example}

\section{Main result}\label{Sec:MainResult}
\noindent
In this section we state and prove the main result of the paper. For convenience of the reader, we recall the notation used throughout the paper. Let $(\QQ,\sigma)$ be a symmetric quiver of type $A_n$, for some fixed $n\geq 2$. Let $\A=\mathrm{k}\QQ$ be its complex path algebra. Let $\ee$ be $+1$ or $-1$. Let $(V,\Psi)$ be an $\ee$-quadratic space for $(\QQ,\sigma)$ (see Section~\ref{Subsec:Psi}) and let $\mathbf{d}=\dimv V$. For simplicity of notation, we denote by $R_\mathbf{d}$ the variety $R(\A,V)$ of $\QQ$-representations with underlying vector space $V$ and by $R_\mathbf{d}^\ee$ the subvariety $R(\A,V)^{\Psi,\ee}$ of $\ee$-representations with respect to $(V,\Psi)$ (see \eqref{Eq:RepVarPsi} for the definition). We denote by $\G_\mathbf{d}=\prod_{i\in\QQ_0}\GL(V_i)$ the structure group and by $\G^\ee_\mathbf{d}:=\G^\bullet(V,\Psi)$ the subgroup of graded isometries of $(V, \Psi)$.  Given $M,N\in R_\mathbf{d}^\ee$ we define the $\ee$-degeneration order by $M\degg^\ee N$ if $N\in \overline{\G^\ee M}$. 

\begin{thm}\label{Thm:MainThm}
For every $M,N\in R_\mathbf{d}^\ee$ the following equivalences hold:
\begin{equation}\label{Eq:MainThm}
\xymatrix{
M\degg N\ar@{<=>}[r]&M\degg^\ee N\ar@{<=>}[r]&M\leq_{\Ext}^\ee N.
}
\end{equation}
\end{thm}
\begin{proof}
The implication $\xymatrix@1@C=10pt{\degg&\degg^\ee\ar@{=>}[l]}$ is obvious since $\G_\mathbf{d}^\ee$ is a subgroup of $\GL^\bullet(V)$, and the implication $\xymatrix@1@C=10pt{\degg^\ee&\leq_{\Ext}^\ee\ar@{=>}[l]}$ holds in general by Corollary~\ref{cor:ExtDeg}. We hence prove $\xymatrix@1@C=10pt{\degg\ar@{=>}[r]&\leq_{\Ext}^\ee}.$ Let $M,N\in R_\mathbf{d}^\ee$ such that $M\degg N$. If $M$ and $N$ are isomorphic as $\QQ$-representations then they are isomorphic as $\ee$-representations by Theorem~\ref{thm:isoclasses}. In the rest of the proof we use the symbol $\simeq$ to denote that two $\ee$-representations or two $\QQ$-representations are isomorphic.
\\[1ex]
We hence assume that $M\lneq_{\deg} N$. Since $\Rep(\QQ)$ is representation-directed, there exists an indecomposable direct summand $L$ of $N$ such that $[L,N]^1=0$. Then $[L,M]^1\leq [L,N]^1=0$.   By \eqref{Eq:HomologicalInterpretationEulerForm} we get that $[L,N]=[L,M]$. By  Theorem~\ref{Thm:Bongartz},  $L$ embeds into $M$. Let $\iota:L\hookrightarrow M$ be a generic embedding and let $Q=M/\iota(L)$ be the (generic) quotient. 
\\[1ex]
Let us discuss, first, the case when $L\simeq \nabla L$. 
Then $N\simeq E\oplus X$ where $E=L$ (in the non-split types) or $E=L\oplus \nabla L$ (in the split types) is an indecomposable $\ee$--representation and $X$ is an $\ee$--representation. We have $[E,N]^1=0$ and also $[N,E]^1=0$ by duality. By Theorem~\ref{Thm:Bongartz}, $E$ embeds into $M$ and the generic quotient $M/E$ degenerates to $X$; it then follows that $[M/E,E]^1\leq [X,E]^1=0$ and thus $M\simeq E\oplus M/E$. Then $M/E$ is an $\ee$--representation and it degenerates to $X$. By  induction, we get that $M/E\leq^\ee_{\Ext} X$ and hence $M\simeq E\oplus M/E\leq_{\Ext}^\ee N$. 
\\[1ex]
We assume that $L\not\simeq \nabla L$. Then $N\simeq L\oplus \nabla L\oplus X$ where $X$ is an $\ee$--representation. 
\\[1ex]
In the split types the generic embedding $\iota$ is isotropic by Lemma~\ref{Lem:IsoSplitTypes}. 
\\[1ex]
In the non-split types,  we claim that we can apply Proposition~\ref{Prop:KeyLemma}. Indeed, by Theorem~\ref{Thm:Bongartz}, $Q\degg X\oplus \nabla L$. Then $[Q,\nabla L]^1\leq [X\oplus\nabla L,\nabla L]^1\leq [N,\nabla L]^1=[L,N]^1=0$ and hence $[Q,\nabla L]=[X\oplus\nabla L,\nabla L]$. Again by  Theorem~\ref{Thm:Bongartz}, since there is a surjection from $X\oplus \nabla L$ onto $\nabla L$, there is a surjection from $Q$ onto $\nabla L$, too. Let $K$ be the generic kernel of $\xymatrix@1{Q\ar@{->>}[r]&\nabla L}$; then, again by Theorem~\ref{Thm:Bongartz}, $K\degg X$ and hence $[L,K]^1\leq [L,X]^1=0$. We hence see that Proposition~\ref{Prop:KeyLemma} applies and we can thus assume that a generic embedding $\iota:L\hookrightarrow M$ is isotropic. 
\\[1ex]
We denote by $Y$ the generic $\ee$-subquotient of $M$ by $L$. We claim that 
\begin{equation}\label{Eq:MainThmMainClaim}
Y\degg X.
\end{equation}
If the claim holds then the proof finishes by induction as follows: 
$$
M\leq_{\Ext}^\ee L\oplus Y\oplus \nabla L\leq_{\Ext}^\ee L\oplus X\oplus \nabla L=N
$$
(the first inequality holds by Theorem~\ref{Thm:IsoDeg}, and the second is given by induction since $Y$ and $X$ are $\ee$-representations of smaller dimension.)
\\[1ex]
It remains to prove \eqref{Eq:MainThmMainClaim}. 
\\[1ex]
In the non-split types the claim is true since, by Proposition~\ref{Prop:ESubquot}, we can choose $\iota$ so that $Y$ is the generic quotient of $\nabla Q$ by $L$; thus, since $\nabla Q\degg L\oplus X$ and $[L,\nabla Q]=[L,L\oplus X]$, by Theorem~\ref{Thm:Bongartz} we get $Y\degg X$. 
\\[1ex]
We prove \eqref{Eq:MainThmMainClaim} in the split types. We write $M=T^0\oplus\nabla T^0\oplus\overline{M}$, where $T^0$ is the multiplicity-free direct summand of $M$ consisting of $\preceq$-minimal direct summands which lie in the support of $\Hom(L,-)$. By Proposition~\ref{Prop:GenQuot}, $Q\simeq T^1\oplus\nabla T^0\oplus\overline{M}$ and $\nabla Q\simeq T^0\oplus\nabla T^1\oplus\overline{M}$. By Proposition~\ref{Prop:ESubquot},  $Y=T^1\oplus\nabla T^1\oplus \overline{M}$. Thus there is a commutative diagram with exact rows and columns
\begin{equation}\label{Eq:MainThmCommDiag}
\xymatrix@R=10pt@C=10pt{
&&0\ar[d]&0\ar[d]&\\
0\ar[r]&\ar[r]L\ar@{=}[d]&\ar[r]\nabla Q=T^0\oplus \nabla T^1\oplus\overline{M}\ar[d]&\ar[r]Y=T^1\oplus\nabla T^1\oplus\overline{M}\ar[d]&0\\
0\ar[r]&L\ar[r]&\ar[r]M=T^0\oplus \nabla T^0\oplus\overline{M}\ar[d]&Q=T^1\oplus\nabla T^0\oplus\overline{M}\ar[d]\ar[r]&0\\
&&\nabla L\ar[d]\ar@{=}[r]&\nabla L\ar[d]&\\
&&0&0&
}
\end{equation}
\noindent
In view of Theorem~\ref{Cor:EquivalenceOrdersDynkin} it is enough to prove that 
\begin{equation}\label{Eq:MainThmMainClaimReformulated}
[Y,E]\leq [X,E]\qquad \forall E\in\textrm{ind}(\QQ).
\end{equation}
If $[L,E]=0$, then $[Y,E]=[\nabla Q,E]\leq [X\oplus L,E]=[X,E]$.
\\[1ex]
If $[\nabla L,E]^1=0$ then $[Y,E]\leq[Q,E]-[\nabla L,E]\leq [X,E]$. 
\\[1ex] Thus we can assume that $[L,E]=1=[\nabla L,E]^1$ (implying that $[\nabla L,E]=0$). In this case, there is a commutative diagram with exact rows and columns
\begin{equation}\label{Eq:MainThmCommDiagHom}
\xymatrix{
          &0\ar[d]                         &0\ar[d]               &\\
0\ar[r]&\Hom(Q,E)\ar[r]\ar[d]&\Hom(M,E)\ar^{h_3}[r]\ar^{h_1}[d]&\Hom(L,E)=\mathrm{k}\ar@{=}[d]\\
0\ar[r]&\Hom(Y,E)\ar[r]\ar^{k_2}[d]&\Hom(\nabla Q,E)\ar^{h_2}[r]\ar^{k_1}[d]&\Hom(L,E)=\mathrm{k}\\
&\Ext^1(\nabla L,E)\ar@{=}[r]&\Ext^1(\nabla L,E)=\mathrm{k}&
}
\end{equation}
We get $[Y,E]\leq [Q,E]+1\leq [X,E]+1$. Thus we can assume that $[Q,E]=[X,E]$. Similarly, we get  $[Y,E]\leq [\nabla Q,E]\leq [X,E]+1$. Thus we can assume that $[\nabla Q,E]=[X,E]+1$.
If $[\nabla Q, E]=[M,E]$ then $h_1$ is surjective; this implies that  $k_1=0$ and hence $k_2=0$. Thus, $[Y,E]=[Q,E]= [X,E]$ in this case.  If  $[Q,E]<[M,E]$ then $h_3$ is surjective; this implies that $h_2$ is surjective and hence $[Y,E]=[\nabla Q,E]-1=[X,E]$.  It remains to treat the case when $E$ is such that 
\begin{equation}\label{Eq:AssumptionsMainThm}
\begin{array}{cc}
[Q,E]=[M,E]=[X,E],&[\nabla Q,E]=[M,E]+1=[X,E]+1.
\end{array}
\end{equation}
We claim that under the assumptions \eqref{Eq:AssumptionsMainThm} there exists an index $i$ such that $E$ is contained in the $\nabla$-invariant square 
\begin{equation}\label{Eq:SquareMainThm}
\vcenter{\xymatrix{
&G_1\ar@{..>}[dr]&\\
\nabla T^1_i\ar@{..>}[ur]\ar@{..>}[dr]&&T^1_i\\
&G_2\ar@{..>}[ur]&
}}
\end{equation}
and moreover $G_1\not\preceq E$ and $G_2\not\preceq E$ (i.e. $E$ does not lie in the sectional paths ending in $T^1_i$). To prove this claim it is enough to prove that $[\nabla T^1,E]\neq 0$ and $[T^1,E]=[T^0,E]=0$. We have $[\nabla T^1,E]\neq 0$, since $[\nabla Q,E]>[M,E]$. To prove that  $[T^1,E]=[T^0,E]=0$ we consider the short exact sequence 
$$
0\rightarrow L\rightarrow T^0\rightarrow T^1\rightarrow 0.
$$
If $[T^1,E]\neq 0$ then $[T^1,E]^1=0$ (this is because  $T^1_i\not\preceq T^1_j$ for $i\neq j$, or because $T^1$ is part of a separating tilting object, see Remark~\ref{Rem:Tilting}) and hence $[T^0,E]=[T^1,E]+[L,E]>[T^1,E]$ which contradicts the hypothesis $[Q,E]=[M,E]$. Thus, $[T^1,E]=[T^0,E]=0$ and hence both $E$ and $\tau\nabla E$ are contained in the square \eqref{Eq:SquareMainThm}. If $E=\tau\nabla E$, i.e. if $E$ is $\delta$-invariant, then $\delta_{Y,X}(E)=[X,E]-[Y,E]$ is even by Proposition~\ref{prop:DeltaEven}. Then  $[Y,E]\neq [X,E]+1$ and hence $[Y,E]\leq [X,E]$. Suppose that $\tau\nabla E\neq E$. Without loss of generality we can assume that $\tau\nabla E\prec E$, since $\delta_{Y,E}(E)=\delta_{Y,X}(\tau\nabla E)$ by Lemma~\ref{Lem:DeltaSelfDual}. The rectangle from $\tau\nabla E$ to $E$ is a square
$$
\xymatrix{
&F_1\ar@{..>}[dr]&\\
\tau \nabla E\ar@{..>}[ur]\ar@{..>}[dr]&&E\\
&F_2\ar@{..>}[ur]&
}
$$
and its middle vertices $F_1$ and $F_2$ are $\delta$-invariants. This square is all contained in the rectangle \eqref{Eq:SquareMainThm} and gives rise to a short exact sequence 
$$
0\rightarrow \tau\nabla E\rightarrow F=F_1\oplus F_2\rightarrow E\rightarrow 0.
$$
The only indecomposable direct summand of $Y$ which could belong to this square is $\nabla T^1_i=\tau\nabla E$, since, if there are others then there would be a $\preceq$-minimal direct summands of $\overline{M}$ belonging to the support of $\Hom(L,-)$, against the definition of $T^0$. This implies that the induced map $\Hom(Y,F)\rightarrow\Hom(Y,E)$ is surjective and hence we have
$$
[Y,F]=[Y,\tau\nabla E]+[Y,E].
$$ 
Thus
$$
[Y,\tau\nabla E]+[Y,E]=[Y,F]\leq [X,F]\leq [X,\tau\nabla E]+[X,E]
$$
which implies by Lemma~\ref{Lem:DeltaSelfDual}
$$
2\delta_{Y,X}(E)\geq0
$$
proving the claim.
\end{proof}
\begin{rem}\label{Rk:GenericSubquotient}
In the proof of Theorem~\ref{Thm:MainThm} we see that if $M$ is an $\ee$-representation which degenerates to an $\ee$-representation of the form $L\oplus\nabla L\oplus X$ and $[L,N]=[L,M]$ then the generic $\ee$-subquotient of $M$ by $L$ degenerates to $X$. This suggests the existence of a ``cancellation theorem'' for isotropic subrepresentations. 
\end{rem}

\section{Examples}\label{Sec:Example}
\noindent
In this section we illustrate the proof of Theorem~\ref{Thm:MainThm} in some examples. Given two $\ee$-representations $M$, $N$ of a symmetric quiver of type $A$ such that $M\degg N$ we find a sequence $M=M(0), M(1), \cdots, M(h)=N$ of $\ee$-representations such that there is a one parameter subgroup $\lambda_i(t)\in\G_\mathbf{d}^\ee$ such that $\lim_{t\rightarrow 0}\lambda_i(t)\cdot M(i)=M(i+1)$. We write $\xymatrix@C=10pt{M(i)\ar@{=>}[r]&M(i+1)}$.  To construct $M(i+1)$ from $M(i)$ we follow the strategy of the proof of Theorem~\ref{Thm:MainThm}: we choose an indecomposable direct summand $L$ of $N$ such that $[L,N]=[L,M(i)]$ and we find the generic $\ee$-subquotient  of $M(i)$ by $L$ described in Proposition~\ref{Prop:ESubquot}. Then Theorem~\ref{Thm:IsoDeg} provides the aforementioned  one-parameter subgroup. If $L$ is a direct summand of both $N$ and $M(i)$ then we can remove it from both. The representations $M(i)$'s are described by the multiplicities of their direct summands. At  step $i$ we highlight as $\xymatrix{*+[F]{\bullet}}$ the indecomposable direct summand $L$ of $N$ used to construct $M(i+1)$ from $M(i)$. Different choices of $L$ give rise to different degeneration paths from $M$ to $N$.
\begin{example} 
Let $\QQ$ be $\vcenter{\xymatrix@R=5pt@C=5pt{1\ar[dr]&&3\ar[dr]\ar[dl]&\\&2&&4\\}}$ and $\ee=1$ (split type). ($\Gamma_\QQ$ is shown in Example~\ref{Ex:ARQuiverAEvenAltern}.)
$$
\def\g#1{\save [].[dddrrrr]!C="g#1"*+<5pt>[F-:<2pt>]\frm{}\restore}%
\xymatrix@R=2pt@C=2pt{
\g1 M&\bullet&&\bullet& && \g2\phantom{M}&1&&1& && \g3 N&1\ar@{-}[dr]&&1&        \\
*+[F]{\bullet}\ar@{-}[ur]&&4\ar@{-}[ur]\ar@{-}[ul]&&\bullet\ar@{-}[ul]              &&  1\ar@{-}[ur]&&2\ar@{-}[ur]\ar@{-}[ul]&&1\ar@{-}[ul]  &&
1\ar@{-}[ur]&&\bullet\ar@{-}[ur]\ar@{-}[dl]&&1\ar@{-}[ul]           \\
&\bullet\ar@{-}[ur]\ar@{-}[ul]&&\bullet\ar@{-}[ur]\ar@{-}[ul]&                    &&  &*+[F]{\bullet}\ar@{-}[ur]\ar@{-}[ul]&&\bullet\ar@{-}[ur]\ar@{-}[ul]&    &&
&1\ar@{-}[ur]\ar@{-}[ul]&&1\ar@{-}[ur]\ar@{-}[ul]&           \\
  \ar@{-}[ur]\bullet &&\bullet\ar@{-}[ur]\ar@{-}[ul]&&\bullet\ar@{-}[ul]        && 1\ar@{-}[ur]&&\bullet\ar@{-}[ur]\ar@{-}[ul]&&1\ar@{-}[ul] && 
 2\ar@{-}[ur]&&\bullet\ar@{-}[ur]\ar@{-}[ul]&&2\ar@{-}[ul]      
\ar @{=>} "g1"; "g2"
\ar @{=>} "g2"; "g3"
}
$$
\end{example}

\begin{example} 
Let $\QQ$ be $\vcenter{\xymatrix@R=5pt@C=5pt{1\ar[dr]&&3\ar[dr]\ar[dl]&\\&2&&4\\}}$ and $\ee=-1$ (non-split type). 
$$
\def\g#1{\save [].[dddrrrr]!C="g#1"*+<5pt>[F-:<2pt>]\frm{}\restore}%
\xymatrix@R=2pt@C=2pt{
\g1 M&\bullet&&\bullet& && \g2\phantom{M}&\bullet&&\bullet& && \g3\phantom{M}&*+[F]{\bullet}&&\bullet&     && \g4 N&1&&1& \\
*+[F]{\bullet}\ar@{-}[ur]&&4\ar@{-}[ur]\ar@{-}[ul]&&\bullet\ar@{-}[ul]              &&  1\ar@{-}[ur]&&3\ar@{-}[ur]\ar@{-}[ul]&&1\ar@{-}[ul]  &&
1\ar@{-}[ur]&&1\ar@{-}[ur]\ar@{-}[ul]&&1\ar@{-}[ul]     &&         1\ar@{-}[ur]&&\bullet\ar@{-}[ur]\ar@{-}[ul]&&1\ar@{-}[ul]              \\
&\bullet\ar@{-}[ur]\ar@{-}[ul]&&\bullet\ar@{-}[ur]\ar@{-}[ul]&                    &&  &*+[F]{\bullet}\ar@{-}[ur]\ar@{-}[ul]&&\bullet\ar@{-}[ur]\ar@{-}[ul]&    &&
&1\ar@{-}[ur]\ar@{-}[ul]&&1\ar@{-}[ur]\ar@{-}[ul]&         &&      &1\ar@{-}[ur]\ar@{-}[ul]&&1\ar@{-}[ur]\ar@{-}[ul]&          \\
  \ar@{-}[ur]\bullet &&\bullet\ar@{-}[ur]\ar@{-}[ul]&&\bullet\ar@{-}[ul]        &&  1\ar@{-}[ur]&&\bullet\ar@{-}[ur]\ar@{-}[ul]&&1\ar@{-}[ul] && 
 2\ar@{-}[ur]&&\bullet\ar@{-}[ur]\ar@{-}[ul]&&2\ar@{-}[ul]     &&     \ar@{-}[ur] 2 &&\bullet\ar@{-}[ur]\ar@{-}[ul]&&2\ar@{-}[ul]        
\ar @{=>} "g1"; "g2"
\ar @{=>} "g2"; "g3"
\ar @{=>} "g3"; "g4" 
}
$$
\end{example}

\begin{example}
Let $\QQ$ be of type $\stackrel{\rightarrow}{A_5}$ and $\ee=-1$ (split type):
$$
\def\g#1{\save [].[ddddrrrrrrrr]!C="g#1"*+<5pt>[F-:<2pt>]\frm{}\restore}
\xymatrix@R=2pt@C=2pt{
\g1 M&&&&*{6}\ar@{-}[dr]&&&&  &&&&   \g2\phantom{M}&&&&4\ar@{-}[dr]&&&&  \\
&&&\bullet\ar@{-}[ur]\ar@{-}[dr]&&\bullet\ar@{-}[dr]&&&   &&&&     &&&1\ar@{-}[ur]\ar@{-}[dr]&&1\ar@{-}[dr]&&& \\
&&\bullet\ar@{-}[ur]\ar@{-}[dr]&&\bullet\ar@{-}[ur]\ar@{-}[dr]&&\bullet\ar@{-}[dr]&&  &&&&   &&\bullet\ar@{-}[ur]\ar@{-}[dr]&&\bullet\ar@{-}[ur]\ar@{-}[dr]&&\bullet\ar@{-}[dr]&& \\
&\bullet\ar@{-}[ur]\ar@{-}[dr]&&\bullet\ar@{-}[ur]\ar@{-}[dr]&&\bullet\ar@{-}[ur]\ar@{-}[dr]&&\bullet\ar@{-}[dr]& &&&& &*+[F]{\bullet}\ar@{-}[ur]\ar@{-}[dr]&&\bullet\ar@{-}[ur]\ar@{-}[dr]&&\bullet\ar@{-}[ur]\ar@{-}[dr]&&\bullet\ar@{-}[dr]& \\
*+[F]{\bullet}\ar@{-}[ur]&&\bullet\ar@{-}[ur]&&\bullet\ar@{-}[ur]&&\bullet\ar@{-}[ur]&&\bullet		&&&&		 1\ar@{-}[ur]&&\bullet\ar@{-}[ur]&&\bullet\ar@{-}[ur]&&\bullet\ar@{-}[ur]&&1
\\
&&&&&&&& &&&& &&&&&&&& \\
\g5 N&&&&2\ar@{-}[dr]&&&&        &&&&  \g3\phantom{M}&&&&4\ar@{-}[dr]&&&&      \\
&&&1\ar@{-}[ur]\ar@{-}[dr]&&1\ar@{-}[dr]&&&        &&&& &&&\bullet\ar@{-}[ur]\ar@{-}[dr]&&\bullet\ar@{-}[dr]&&&       \\
&&1\ar@{-}[ur]\ar@{-}[dr]&&\bullet\ar@{-}[ur]\ar@{-}[dr]&&1\ar@{-}[dr]&&    
&&&&   &&*+[F]{\bullet}\ar@{-}[ur]\ar@{-}[dr]&&\bullet\ar@{-}[ur]\ar@{-}[dr]&&\bullet\ar@{-}[dr]&&     \\
&1\ar@{-}[ur]\ar@{-}[dr]&&\bullet\ar@{-}[ur]\ar@{-}[dr]&&\bullet\ar@{-}[ur]\ar@{-}[dr]&&1\ar@{-}[dr]&         &&&&   &1\ar@{-}[ur]\ar@{-}[dr]&&1\ar@{-}[ur]\ar@{-}[dr]&&1\ar@{-}[ur]\ar@{-}[dr]&&1\ar@{-}[dr]&     \\
1\ar@{-}[ur]&&\bullet\ar@{-}[ur]&&\bullet\ar@{-}[ur]&&\bullet\ar@{-}[ur]&&1
&&&&	1\ar@{-}[ur]&&\bullet\ar@{-}[ur]&&\bullet\ar@{-}[ur]&&\bullet\ar@{-}[ur]&&1	
\ar @{=>} "g1"; "g2"
\ar @{=>} "g2"; "g3" <20pt>
{\ar @{=>} "g3"; "g5" }
}
$$
\end{example}
\begin{example}
Let $\QQ$ be of type $\stackrel{\rightarrow}{A_5}$ and $\ee=1$ (non-split type):
$$
\def\g#1{\save [].[ddddrrrrrrrr]!C="g#1"*+<5pt>[F-:<2pt>]\frm{}\restore}
\xymatrix@R=2pt@C=1pt{
\g1 M&&&&*{6}\ar@{-}[dr]&&&&  &&   \g2\phantom{M}&&&&5\ar@{-}[dr]&&&&  &&    \g3\phantom{M}&&&&4\ar@{-}[dr]&&&&\\
&&&\bullet\ar@{-}[ur]\ar@{-}[dr]&&\bullet\ar@{-}[dr]&&&   &&     &&&\bullet\ar@{-}[ur]\ar@{-}[dr]&&\bullet\ar@{-}[dr]&&& &&    &&&\bullet\ar@{-}[ur]\ar@{-}[dr]&&\bullet\ar@{-}[dr]&&&\\
&&\bullet\ar@{-}[ur]\ar@{-}[dr]&&\bullet\ar@{-}[ur]\ar@{-}[dr]&&\bullet\ar@{-}[dr]&&  &&   &&\bullet\ar@{-}[ur]\ar@{-}[dr]&&1\ar@{-}[ur]\ar@{-}[dr]&&\bullet\ar@{-}[dr]&& &&   &&*+[F]{\bullet}\ar@{-}[ur]\ar@{-}[dr]&&1\ar@{-}[ur]\ar@{-}[dr]&&\bullet\ar@{-}[dr]&&\\
&\bullet\ar@{-}[ur]\ar@{-}[dr]&&\bullet\ar@{-}[ur]\ar@{-}[dr]&&\bullet\ar@{-}[ur]\ar@{-}[dr]&&\bullet\ar@{-}[dr]& && &*+[F]{\bullet}\ar@{-}[ur]\ar@{-}[dr]&&\bullet\ar@{-}[ur]\ar@{-}[dr]&&\bullet\ar@{-}[ur]\ar@{-}[dr]&&\bullet\ar@{-}[dr]& && &1\ar@{-}[ur]\ar@{-}[dr]&&\bullet\ar@{-}[ur]\ar@{-}[dr]&&\bullet\ar@{-}[ur]\ar@{-}[dr]&&1\ar@{-}[dr]&\\
*+[F]{\bullet}\ar@{-}[ur]&&\bullet\ar@{-}[ur]&&\bullet\ar@{-}[ur]&&\bullet\ar@{-}[ur]&&\bullet		&&		 1\ar@{-}[ur]&&\bullet\ar@{-}[ur]&&\bullet\ar@{-}[ur]&&\bullet\ar@{-}[ur]&&1
&&		 1\ar@{-}[ur]&&\bullet\ar@{-}[ur]&&1\ar@{-}[ur]&&\bullet\ar@{-}[ur]&&1
\\
&&&&&&&& && &&&&&&&& && &&&&&&&&\\
&&&&&&&&      &&           \g5 N&&&&2\ar@{-}[dr]&&&&        &&           \g4 \phantom{M}&&&&3\ar@{-}[dr]&&&&\\
&&&&&&&&       &&         &&&1\ar@{-}[ur]\ar@{-}[dr]&&1\ar@{-}[dr]&&&        &&                   &&&*+[F]{\bullet}\ar@{-}[ur]\ar@{-}[dr]&&\bullet\ar@{-}[dr]&&&\\
&&&&&&&&     &&        &&1\ar@{-}[ur]\ar@{-}[dr]&&\bullet\ar@{-}[ur]\ar@{-}[dr]&&1\ar@{-}[dr]&&    
&&   
&&1\ar@{-}[ur]\ar@{-}[dr]&&1\ar@{-}[ur]\ar@{-}[dr]&&1\ar@{-}[dr]&&\\
&&&&&&&&      &&         &1\ar@{-}[ur]\ar@{-}[dr]&&\bullet\ar@{-}[ur]\ar@{-}[dr]&&\bullet\ar@{-}[ur]\ar@{-}[dr]&&1\ar@{-}[dr]&         &&               &1\ar@{-}[ur]\ar@{-}[dr]&&\bullet\ar@{-}[ur]\ar@{-}[dr]&&\bullet\ar@{-}[ur]\ar@{-}[dr]&&1\ar@{-}[dr]&\\
&&&&&&&&	&&		 1\ar@{-}[ur]&&\bullet\ar@{-}[ur]&&\bullet\ar@{-}[ur]&&\bullet\ar@{-}[ur]&&1
&&		1\ar@{-}[ur]&&\bullet\ar@{-}[ur]&&\bullet\ar@{-}[ur]&&\bullet\ar@{-}[ur]&&1
\\
\ar @{=>} "g1"; "g2"
\ar @{=>} "g2"; "g3"
{\ar @{=>} "g3"; "g4" <20pt>}
\ar @{=>} "g4"; "g5"
}
$$
\end{example}
\begin{example}
Let $\QQ$ be 
$\vcenter{\xymatrix@R=5pt@C=5pt{&2\ar[dl]\ar[dr]&&&\\1&&3\ar[dr]&&5\ar[dl]\\&&&4&}}$ and $\ee=1$ (non-split type). ($\Gamma_\QQ$ is shown in Example~\ref{Ex:ARQuiverAOddAltern}.) 
$$
\def\g#1{\save [].[ddddrrrrrr]!C="g#1"*+<5pt>[F-:<2pt>]\frm{}\restore}
\xymatrix@R=2pt@C=2pt{
\g1 M&\bullet\ar@{-}[dr]&&\bullet\ar@{-}[dr]\ar@{-}[dl]&&\bullet\ar@{-}[dl]&                                    &&      \g2 \phantom{M}&1\ar@{-}[dr]&&\bullet\ar@{-}[dr]\ar@{-}[dl]&&1\ar@{-}[dl]&                                       &&
\g3 \phantom{M}&1\ar@{-}[dr]&&\bullet\ar@{-}[dr]\ar@{-}[dl]&&1\ar@{-}[dl]&                                  \\
&&\bullet\ar@{-}[dr]\ar@{-}[dl]&&\bullet\ar@{-}[dr]\ar@{-}[dl]&&                                                       &&      &&\bullet\ar@{-}[dr]\ar@{-}[dl]&&\bullet\ar@{-}[dr]\ar@{-}[dl]&&                                                       &&
&&\bullet\ar@{-}[dr]\ar@{-}[dl]&&\bullet\ar@{-}[dr]\ar@{-}[dl]&&                                                  \\
&\bullet\ar@{-}[dr]\ar@{-}[dl]&&4\ar@{-}[dr]\ar@{-}[dl]&&\bullet\ar@{-}[dr]\ar@{-}[dl]&             &&      &\bullet\ar@{-}[dr]\ar@{-}[dl]&&4\ar@{-}[dr]\ar@{-}[dl]&&\bullet\ar@{-}[dr]\ar@{-}[dl]&            &&
&\bullet\ar@{-}[dr]\ar@{-}[dl]&&3\ar@{-}[dr]\ar@{-}[dl]&&\bullet\ar@{-}[dr]\ar@{-}[dl]&        \\
*+[F]{\bullet}\ar@{-}[dr]&&\bullet\ar@{-}[dr]\ar@{-}[dl]&&\bullet\ar@{-}[dr]\ar@{-}[dl]&&\bullet\ar@{-}[dl]   &&      1\ar@{-}[dr]&&\bullet\ar@{-}[dr]\ar@{-}[dl]&&\bullet\ar@{-}[dr]\ar@{-}[dl]&&1\ar@{-}[dl]   &&
1\ar@{-}[dr]&&*+[F]{\bullet}\ar@{-}[dr]\ar@{-}[dl]&&\bullet\ar@{-}[dr]\ar@{-}[dl]&&1\ar@{-}[dl]  \\
&1&&\bullet&&1&                                                                                                                &&     &*+[F]{\bullet}&&\bullet&&\bullet&   &&    &1&&1&&1&\\
&&&&&& && &&&&&& && &&&&&& && \\
\g6 N&3\ar@{-}[dr]&&2\ar@{-}[dr]\ar@{-}[dl]&&3\ar@{-}[dl]&                                    &&      \g5 \phantom{M}&*+[F]{2}\ar@{-}[dr]&&1\ar@{-}[dr]\ar@{-}[dl]&&2\ar@{-}[dl]&                                       &&
\g4 \phantom{M}&*+[F]{1}\ar@{-}[dr]&&\bullet\ar@{-}[dr]\ar@{-}[dl]&&1\ar@{-}[dl]&                                  \\
&&\bullet\ar@{-}[dr]\ar@{-}[dl]&&\bullet\ar@{-}[dr]\ar@{-}[dl]&&                                                       &&      &&\bullet\ar@{-}[dr]\ar@{-}[dl]&&\bullet\ar@{-}[dr]\ar@{-}[dl]&&                                                       &&
&&\bullet\ar@{-}[dr]\ar@{-}[dl]&&\bullet\ar@{-}[dr]\ar@{-}[dl]&&                                                  \\
&\bullet\ar@{-}[dr]\ar@{-}[dl]&&\bullet\ar@{-}[dr]\ar@{-}[dl]&&\bullet\ar@{-}[dr]\ar@{-}[dl]&             &&      &\bullet\ar@{-}[dr]\ar@{-}[dl]&&1\ar@{-}[dr]\ar@{-}[dl]&&\bullet\ar@{-}[dr]\ar@{-}[dl]&            &&
&\bullet\ar@{-}[dr]\ar@{-}[dl]&&2\ar@{-}[dr]\ar@{-}[dl]&&\bullet\ar@{-}[dr]\ar@{-}[dl]&        \\
1\ar@{-}[dr]&&1\ar@{-}[dr]\ar@{-}[dl]&&1\ar@{-}[dr]\ar@{-}[dl]&&1\ar@{-}[dl]   &&      1\ar@{-}[dr]&&1\ar@{-}[dr]\ar@{-}[dl]&&1\ar@{-}[dr]\ar@{-}[dl]&&1\ar@{-}[dl]   &&
1\ar@{-}[dr]&&1\ar@{-}[dr]\ar@{-}[dl]&&1\ar@{-}[dr]\ar@{-}[dl]&&1\ar@{-}[dl]  \\
&1&&\bullet&&1&                                                                                                                &&     &1&&\bullet&&1&   &&    &1&&\bullet&&1&\\
\ar @{=>} "g1"; "g2"
\ar @{=>} "g2"; "g3"
\ar @{=>} "g3"; "g4" <10pt>
\ar @{=>} "g4"; "g5"
\ar @{=>} "g5"; "g6"
}
$$
\end{example}
\begin{example}
Let $\QQ$ be 
$\vcenter{\xymatrix@R=5pt@C=5pt{&2\ar[dl]\ar[dr]&&&\\1&&3\ar[dr]&&5\ar[dl]\\&&&4&}}$ and $\ee=-1$ (split type). 
$$
\def\g#1{\save [].[ddddrrrrrr]!C="g#1"*+<5pt>[F-:<2pt>]\frm{}\restore}%
\xymatrix@R=2pt@C=2pt{
\g1 M&\bullet\ar@{-}[dr]&&\bullet\ar@{-}[dr]\ar@{-}[dl]&&\bullet\ar@{-}[dl]&                                    &&      \g2 \phantom{M}&1\ar@{-}[dr]&&\bullet\ar@{-}[dr]\ar@{-}[dl]&&1\ar@{-}[dl]&                                       &&
\g3 \phantom{M}&2\ar@{-}[dr]&&\bullet\ar@{-}[dr]\ar@{-}[dl]&&2\ar@{-}[dl]&                                  \\
&&\bullet\ar@{-}[dr]\ar@{-}[dl]&&\bullet\ar@{-}[dr]\ar@{-}[dl]&&                                                       &&      &&\bullet\ar@{-}[dr]\ar@{-}[dl]&&\bullet\ar@{-}[dr]\ar@{-}[dl]&&                                                       &&
&&\bullet\ar@{-}[dr]\ar@{-}[dl]&&\bullet\ar@{-}[dr]\ar@{-}[dl]&&                                                  \\
&\bullet\ar@{-}[dr]\ar@{-}[dl]&&4\ar@{-}[dr]\ar@{-}[dl]&&\bullet\ar@{-}[dr]\ar@{-}[dl]&             &&      &*+[F]{\bullet}\ar@{-}[dr]\ar@{-}[dl]&&4\ar@{-}[dr]\ar@{-}[dl]&&\bullet\ar@{-}[dr]\ar@{-}[dl]&            &&
&1\ar@{-}[dr]\ar@{-}[dl]&&2\ar@{-}[dr]\ar@{-}[dl]&&1\ar@{-}[dr]\ar@{-}[dl]&        \\
*+[F]{\bullet}\ar@{-}[dr]&&1\ar@{-}[dr]\ar@{-}[dl]&&1\ar@{-}[dr]\ar@{-}[dl]&&\bullet\ar@{-}[dl]   &&      1\ar@{-}[dr]&&\bullet\ar@{-}[dr]\ar@{-}[dl]&&\bullet\ar@{-}[dr]\ar@{-}[dl]&&1\ar@{-}[dl]   &&
1\ar@{-}[dr]&&1\ar@{-}[dr]\ar@{-}[dl]&&1\ar@{-}[dr]\ar@{-}[dl]&&1\ar@{-}[dl]  \\
&\bullet&&\bullet&&\bullet&                                                                                                                &&     &\bullet&&2&&\bullet&   &&    &*+[F]{\bullet}&&\bullet&&\bullet&\\
&&&&&& && &&&&&& && &&&&&& && \\
&&&&&&                              &&      \g5 N&2\ar@{-}[dr]&&\bullet\ar@{-}[dr]\ar@{-}[dl]&&2\ar@{-}[dl]&                                       &&
\g4 \phantom{M}&2\ar@{-}[dr]&&\bullet\ar@{-}[dr]\ar@{-}[dl]&&2\ar@{-}[dl]&                                  \\
&&&&&&                                                     &&      &&\bullet\ar@{-}[dr]\ar@{-}[dl]&&\bullet\ar@{-}[dr]\ar@{-}[dl]&&                                                       &&
&&\bullet\ar@{-}[dr]\ar@{-}[dl]&&\bullet\ar@{-}[dr]\ar@{-}[dl]&&                                                  \\
&&&&&&             &&      &1\ar@{-}[dr]\ar@{-}[dl]&&\bullet\ar@{-}[dr]\ar@{-}[dl]&&1\ar@{-}[dr]\ar@{-}[dl]&            &&
&1\ar@{-}[dr]\ar@{-}[dl]&&2\ar@{-}[dr]\ar@{-}[dl]&&1\ar@{-}[dr]\ar@{-}[dl]&        \\
&&&&&&   &&      1\ar@{-}[dr]&&1\ar@{-}[dr]\ar@{-}[dl]&&1\ar@{-}[dr]\ar@{-}[dl]&&1\ar@{-}[dl]   &&
1\ar@{-}[dr]&&\bullet\ar@{-}[dr]\ar@{-}[dl]&&\bullet\ar@{-}[dr]\ar@{-}[dl]&&1\ar@{-}[dl]  \\
&&&&&&                                                                                                            &&     &2&&2&&2&   &&    &*+[F]{1}&&2&&1&\\
\ar @{=>} "g1"; "g2"
\ar @{=>} "g2"; "g3"
\ar @{=>} "g3"; "g4" <10pt>
\ar @{=>} "g4"; "g5"
}
$$
\end{example}

\bibliographystyle{amsplain}

\end{document}